\documentclass[12pt,a4paper,reqno]{amsart}
\allowdisplaybreaks

\usepackage[a4paper,left=3.5cm,right=3.5cm]{geometry}
\usepackage[foot]{amsaddr}
\usepackage{amsthm}
\usepackage{amssymb}
\usepackage{amsmath}
\usepackage{mathtools}
\usepackage{amsfonts}
\usepackage{bbm}
\usepackage{xcolor}
\usepackage{enumitem}
\usepackage{mathabx}
\usepackage{tikz-cd}
\usepackage{caption}
\usepackage{ulem}	

\usepackage{tikz}
\usetikzlibrary{intersections,fpu}

\usepackage{hyperref}
\hypersetup{
  colorlinks   = true, 
  urlcolor     = blue, 
  linkcolor    = blue, 
  citecolor   = blue, 
}

\newcommand{\rleq}{\trianglelefteq}
\newcommand{\rgeq}{\trianglerighteq}

\newtheorem{proposition}{Proposition}
\newtheorem{theorem}[proposition]{Theorem}
\newtheorem{lemma}[proposition]{Lemma}
\newtheorem{corollary}[proposition]{Corollary}

\theoremstyle{definition}
\newtheorem{definition}[proposition]{Definition}

\newtheorem{remark}[proposition]{Remark}

\newcommand{\mc}[1]{\mathcal{#1}}
\newcommand{\ms}[1]{\mathsf{#1}}

\newcommand{\mb}[1]{\mathbb{#1}}

\newcommand{\N}{\mathbb N}
\newcommand{\R}{\mathbb R}
\newcommand{\Z}{\mathbb Z}

\newcommand{\Q}{\mathbb Q}
\newcommand{\T}{\mathbb T}

\def\<{\langle}
\def\>{\rangle}

\newcommand{\fii}{\varphi}
\newcommand{\eps}{\varepsilon}
\newcommand{\tj}{\vartheta}

\let\originalleft\left
\let\originalright\right
\renewcommand{\left}{\mathopen{}\mathclose\bgroup\originalleft}
\renewcommand{\right}{\aftergroup\egroup\originalright}

\newcommand{\major}{\preceq}
\newcommand{\minor}{\succeq}

\newcommand{\majo}{\succeq}

\makeatletter
\newcommand*{\saved@uline}{}
\let\saved@uline\uline
\newcommand*{\mathuline}{%
  \mathpalette{\math@uline\saved@uline}%
}
\newcommand*{\math@uline}[3]{%
  \mbox{#1{$#2#3\m@th$}}%
}

\begin{document}

\title[Matrix majorization in large samples]{Matrix majorization in large samples}

    \author{Muhammad Usman Farooq$^1$, Tobias Fritz$^2$, \\ Erkka Haapasalo$^3$, Marco Tomamichel$^{3,4}$}
    \address{$^1$Department of Mathematics, City University of Hong Kong, Hong Kong}
    \address{$^2$Department of Mathematics, University of Innsbruck, Austria}
    \address{$^3$Centre for Quantum  Technologies,  National University of Singapore, Singapore}
    \address{$^4$Department of Electrical and Computer Engineering, National University of Singapore, Singapore}
\maketitle

\begin{abstract}
One tuple of probability vectors is more informative than another tuple when there exists a single stochastic matrix transforming the probability vectors of the first tuple into the probability vectors of the other. This is called matrix majorization. Solving an open problem raised by Mu {\it et al}, we show that if certain monotones\,---\,namely multivariate extensions of R\'enyi divergences\,---\,are strictly ordered between the two tuples, then for sufficiently large $n$, there exists a stochastic matrix taking the $n$-fold Kronecker power of each input distribution to the $n$-fold Kronecker power of the corresponding output distribution. The same conditions, with non-strict ordering for the monotones, are also necessary for such matrix majorization in large samples.

Our result also gives conditions for the existence of a sequence of statistical maps that asymptotically (with vanishing error) convert a single copy of each input distribution to the corresponding output distribution with the help of a catalyst that is returned unchanged. Allowing for transformation with arbitrarily small error, we find conditions that are both necessary and sufficient for such catalytic matrix majorization.

We derive our results by building on a general algebraic theory of preordered semirings recently developed by one of the authors. This also allows us to recover various existing results on majorization in large samples and in the catalytic regime as well as relative majorization in a unified manner.
\end{abstract}

\section{Introduction}

\emph{Statistical experiments} are at the basis of much of classical statistics and its applications, such as hypothesis testing \cite{Heyer1982,Strasser85,Torgersen91}: every hypothesis corresponds to a probability distribution over possible observations, and therefore a statistical experiment is simply a family of probability distributions on the same space of outcomes. Given two statistical experiments, one may naturally ask which one is more ``informative'', for example by asking which one yields the higher expected payoff in the varied tests one might apply. Thus, the comparison of statistical experiments becomes a highly relevant question, going back to the work by Blackwell~\cite{Blackwell51,Blackwell53} and continued, e.g.,\ by Stein~\cite{Stein1951} and by Torgersen~\cite{Torgersen85,Torgersen91}. 

Blackwell's formalization of the idea that a statistical experiment $P=\{p^\tj\}_{\tj\in \Theta}$ with probability measures $p^\tj$ on some (standard Borel) measurable space is more informative than another statistical experiment $Q=\{q^\tj\}_{\tj\in\Theta}$ is by requiring that there be a statistical map $T$ (in the form of a Markov kernel) such that $Tp^\tj=q^\tj$ for all $\tj\in\Theta$. 
One may think of this equation as saying that the experiment $Q$ can be simulated by first executing $P$ and then applying a (possibly random) function to its outcome. By Blackwell's celebrated theorem, this is equivalent to obtaining at least as much expected payoff in any test with $P$ than with $Q$. In this situation it is said that $P$ {\it majorizes} $Q$, written $P \majo Q$.

We are particularly interested in the case where both the outcome space for the probability measures and the parameter space $\Theta$ are finite. Within this setting, the majorization of statistical experiments coincides with the notion of {\it matrix majorization} \cite{dahl99}: A matrix $P$ with non-negative entries majorizes another matrix $Q$ with non-negative entries if there is a stochastic matrix $T$ such that $TP=Q$. If the columns $p^{(k)}$ of $P$ and $q^{(k)}$ of $Q$ for $k=1,\ldots,d$ are probability vectors (i.e.,\ $P$ and $Q$ are stochastic), this simply means that the experiment $(p^{(1)},\ldots,p^{(d)})$ majorizes $(q^{(1)},\ldots,q^{(d)})$ in the above sense. We note that for $d=2$ this problem is also known as relative majorization.

In addition to studying majorization between single statistical experiments, we can also look for conditions when one experiment is more informative than another when each is repeated independently. In the language of matrix majorization, we are interested in whether the condition
\begin{equation}
\left((p^{(1)})^{\otimes n},\ldots,(p^{(d)})^{\otimes n}\right)  \majo  \left((q^{(1)})^{\otimes n},\ldots,(q^{(d)})^{\otimes n}\right)
\end{equation}
holds for a sufficiently large number $n\in\N$ of repetitions. We call this {\it matrix majorization in large samples}. It is well known that there are pairs of matrices (or statistical experiments, in general) where neither majorizes the other, but majorization still holds in large samples \cite{Stein1951,Torgersen85}. Moreover, we say that $P$ {\it catalytically majorizes} $Q$ if there are probability vectors $r^{(1)},\ldots,r^{(d)}$ such that
\begin{equation}
(p^{(1)}\otimes r^{(1)},\ldots,p^{(d)}\otimes r^{(d)})\succeq (q^{(1)}\otimes r^{(1)},\ldots,q^{(d)}\otimes r^{(d)}),
\end{equation}
i.e.,\ having the catalytic matrix $R=(r^{(1)},\ldots,r^{(d)})$ enables the transformation of $P$ into $Q$. Finally, we say that $P$ \emph{asymptotically catalytically majorizes} $Q$ if there is a sequence $(Q_n)_n$ of matrices converging to $Q$ such that $P$ catalytically majorizes $Q_n$ for all $n\in\N$.

As we will see, majorization in large samples implies catalytic majorization, and this follows from a known general construction (see e.g.,\ \cite{DuFeLiYi2005}). Sufficient and generically necessary conditions for majorization in large samples in the case $d=2$ were determined by Mu {\it et al.}~in \cite{Mu_et_al_2020}, and analogous conditions for the case of general $d$ were conjectured. In this work, we prove a minor variation of their conjecture (see Remark~\ref{rem:TropicalDiff} for the difference). This provides sufficient and generically necessary conditions for matrix majorization in large samples in general. Our proof uses the real-algebraic machinery derived by one of the authors in \cite{Fritz2021a,Fritz2021b}, namely the theory of preordered semirings. According to these results, the ordering in large samples on certain types of preordered semirings can be characterized in terms of inequalities involving monotone homomorphisms to $\R_+$ and a number of similar monotone maps.
When applied to a suitable preordered semiring describing matrix majorization, these monotones turn out to match the quantities considered by Mu {\it et al.}~very closely.

In more detail, assume that the matrices $P$ and $Q$ are such that all vectors $p^{(k)}$ and $q^{(k)}$ have full support. In this case, the relevant monotones include the quantities of the form
\begin{equation} \label{eq:intro_monotones}
D_{\mathuline{\alpha}}(P) = \frac{1}{\max_{1\leq k\leq d}\alpha_k - 1}\log{\sum_i \prod_{k=1}^d\left(p^{(k)}_i\right)^{\alpha_k}}, 
\end{equation}
where $\mathuline{\alpha}=(\alpha_1,\ldots,\alpha_d)$ is any $d$-tuple of real parameters summing up to $1$, and such that either $0\leq\alpha_1,\ldots,\alpha_d<1$ or $\alpha_k>1$ for some $k$ and $\alpha_\ell\leq0$ for all $\ell \neq k$. Various pointwise limits of these maps are also relevant and discussed when we formally present our results.

For $d=2$, the maps $D_{\mathuline{\alpha}}(P)$ specialize to certain {\it R\'{e}nyi divergences}, namely to $D_{\alpha_1}(p^{(1)}\|p^{(2)})$ for $\alpha_1 \in [\frac{1}{2},1) \cup (1,\infty)$ and similarly to $D_{\alpha_2}(p^{(2)}\|p^{(1)})$ for $\alpha_2 \in [\frac{1}{2}, 1) \cup (1, \infty)$, where
\begin{equation}
D_\alpha(p\|q) = \frac{1}{\alpha-1}\log{\sum_i p_i^\alpha q_i^{1-\alpha}} \,.
\end{equation}
These quantities already appear in known necessary and sufficient conditions for relative majorization in large samples and in the catalytic regime (see, e.g.,~\cite{Brandao2014,GoToma2022,Mu_et_al_2020}). Thus, it is not surprising that the maps presented above for general $d\geq2$ are natural extensions of the R\'{e}nyi divergences; in particular, they are additive under tensor products and satisfy a data-processing inequality, i.e., they are monotonically non-increasing when a stochastic map is applied to all vectors.

Our two main results can now be informally summarized as follows:
\begin{itemize}
\item
\textbf{We give sufficient (and generically necessary) conditions for matrix majorization in large samples.}
In Theorem~\ref{theor:MatrMajSuff} of Section~\ref{sec:MatrixMaj}, we show that if
\begin{equation}
D_{\mathuline{\alpha}}(P) > D_{\mathuline{\alpha}}(Q)\,,\label{eq:intro_suff}
\end{equation}
for all the parameter tuples $\mathuline{\alpha}$ specified above and if furthermore these strict inequalities also hold for certain pointwise limits, then $P$ majorizes $Q$ in large samples. Furthermore, this sufficient condition is almost necessary: if $P$ majorizes $Q$ in large samples, then we must have
\begin{equation}
D_{\mathuline{\alpha}}(P) \geq D_{\mathuline{\alpha}}(Q)\label{eq:intro_nec}
\end{equation}
for all such $\mathuline{\alpha}$.
We also say that our sufficient condition with strict inequality is ``generically necessary'' in order to indicate that strict inequality can be expected to hold in generic cases.

\item
\textbf{We give necessary and sufficient conditions for asymptotic catalytic matrix majorization.}
As mentioned previously, majorization in large samples implies catalytic majorization and thus the conditions in~\eqref{eq:intro_suff} are sufficient for catalytic majorization as well (and they are still generically necessary). Strengthening this, we show in Theorem~\ref{theor:MatrApproxCatMaj} of Section~\ref{sec:MatrixMaj} that asymptotic catalytic majorization is possible if and only if the conditions in~\eqref{eq:intro_nec} are met. More precisely, we show that the following two statements are equivalent:
\begin{enumerate}[label=\alph*)]
\item For all valid parameter tuples $\mathuline{\alpha}$, it holds that
\begin{equation}
D_{\mathuline{\alpha}}(P) \geq D_{\mathuline{\alpha}}(Q) \,.\label{eq:intro_cata}
\end{equation}
\item There exists a sequence of stochastic matrices $Q_n$ converging to $Q$ such that $P$ catalytically majorizes $Q_n$ for all $n$. Moreover, we can assume that the last column of $Q_n$ is always equal to the last column of $Q$.
\end{enumerate}

\end{itemize}

For $d=2$, the latter statement recovers exactly Theorem 17 in the Supplemental Materials of~\cite{Brandao2014}, at least for the case of distributions with full support.\footnote{The paper~\cite{GoToma2022} cast some doubt on the completeness of the proof in~\cite{Brandao2014}.} It thus in particular strengthens Theorem 20 in~\cite{GoToma2022}. While both earlier works relied on lifting results from simple majorization to the relative setting by means of an embedding, our proof avoids this technique and is more direct, allowing us to consider $d > 2$ as well. See also the discussion in Subsection~\ref{subsec:d2}.

Another theme that we study in this paper is (simple or absolute) majorization between probability vectors. A vector $p$ with non-negative entries majorizes another such vector $q$ if there is $n\in\N$ such that, writing the uniform $n$-entry probability vector as $u_n=(1/n,\ldots,1/n)$, we have $(p,u_n)\succeq(q,u_n)$. The Hardy-Littlewood-P\'{o}lya theorem gives an easily checked condition for this type of majorization in the form of inequalities of the form \eqref{eq:majorizationProb} given in the appendices. Necessary and sufficient conditions for catalytic simple majorization have been provided previously by Klimesh~\cite{klimesh2007inequalities} and Turgut~\cite{Turgut2007}. Large-sample results have been proven, e.g.,\ by Jensen~\cite{Jensen2019}, which we extend in the case of vectors with equal support. Aubrun and Nechita derived an asymptotic version of these results in~\cite{AubrunNechita2007}. In the appendices, we demonstrate the utility of the novel semiring methods by giving straightforward new proofs for these results.

\smallskip

The remainder of this paper is organized as follows. We start by introducing relevant definitions and concepts central to our work in Section~\ref{sec:preliminaries}. In particular, we summarize the real-algebraic results of \cite{Fritz2021a,Fritz2021b} relevant for our applications. We formally introduce the concept of matrix majorization together with the preordered semiring relevant to this setting in Section~\ref{sec:MatrixMaj}, and this culminates in the results summarized above. In Section~\ref{sec:conclusion} we quickly touch on applications of our results and future prospects. We also mention some connections and possible extensions to quantum information theory. The appendices complement this work with a rederivation of some established results on simple majorization: In Section~\ref{sec:majsemirings} we set up the algebraic framework for studying simple majorization in large samples and in the catalytic regime. In Section~\ref{sec:classicalResults}, we demonstrate the applicability of the preordered semiring methods by rederiving central results concerning majorization in large samples and in the catalytic regime previously proven in \cite{AubrunNechita2007,klimesh2007inequalities,Jensen2019} by more {\it ad hoc} methods.

\section{Preliminaries}\label{sec:preliminaries}

With our upcoming applications in mind, in this section we review the mathematical prerequisites pertaining to probability vectors and preordered semirings, and we review existing separation theorems (\emph{Vergleichsstellens\"atze}) which characterize the induced large-sample preorders on sufficiently well-behaved preordered semirings.

\subsection{Basic notation}

We present some basic notation and definitions for finite-outcome probability distributions that we will use throughout this work. We follow the convention $\N:=\{1,2,\ldots\}$. We denote by $\log{}$ and, respectively, $\exp{}$ the logarithm and, respectively, exponential function with respect to a fixed but otherwise undetermined basis greater than 1. For every $n\in\N$, we define the map $\|\cdot\|_1$ on $\R^n$ through
\begin{equation}
\|x\|_1:=|x_1|+\cdots+|x_n|
\end{equation}
for all $x=(x_1,\ldots,x_n)\in\R^n$. For $x=(x_1,\ldots,x_n)\in\R^n$ and $y=(y_1,\ldots,y_m)\in\R^m$ ($n,m\in\N$) we define the direct sum $x\oplus y\in\R^{m+n}$ and the Kronecker product $x\otimes y\in\R^{mn}$ through
\begin{align}
x\oplus y & = (x_1,\ldots,x_n,y_1,\ldots,y_m),\\
x\otimes y & = (x_i y_j)_{i,j},
\end{align}
where the latter one is viewed as a vector of length $mn$; the ordering of the elements is irrelevant throughout this work, so we do not specify it further.

Recall that $p\in\R^n$, for some $n\in\N$, is a probability vector if its entries are non-negative and $\|p\|_1=1$. We denote the set of all $n$-entry probability vectors by $\mc P_n$ and define the set of all finite probability vectors through
\begin{equation}
\mc P_{<\infty}:=\bigcup_{n\in\N}\mc P_n.
\end{equation}
For every $s\in\N$, we denote the uniform probability distribution on $s$ outcomes (probability vector in $\R^s$) by
\begin{equation}
	u_s := \big( \underbrace{ 1/s, \ldots, 1/s}_{s \textrm{ entries}} \big).
\end{equation}
Given a finite set $I\subset\N$, we similarly denote by $u_I$ the distribution with $(u_I)_i=1/|I|$ for $i\in I$ and $(u_I)_i=0$ otherwise, and consider $u_I \in \R^n$ for all $n \ge \max I$. Note that, with a slight abuse of notation, $u_I$ does not explicitly depend on the size $n$; the zeros at the end will not be essential in our context.

\subsection{Preordered semirings and separation theorems}
\label{sec:preorderedSemirings}

\begin{definition}[preordered commutative semiring]\label{def:preorderSemiring}
	A \emph{preordered commutative semiring} is a tuple ($S, +, \cdot, 0, 1, \rleq$) where $S$ is a set equipped with two binary operations $+$, $\cdot$ that are both commutative and associative, have respective neutral elements $0$ and $1$, and satisfy
	\begin{equation}
		0 \cdot x = 0, \qquad (x+y) \cdot z= x \cdot z + y \cdot z
	\end{equation}
	for all $x,y,z\in S$. The preorder $\rleq$ (a reflexive and transitive binary relation) is required to be such that $a \rleq b$ implies $a+c \rleq b+c$ and $ac \rleq bc$ for all $a,b,c\in S$.
\end{definition}

We will use the notation $a \rleq b$ and $b \rgeq a$ interchangeably. 
Note that unlike in rings, subtraction is typically not possible in semirings.
From now on and throughout the paper, we use `semiring' as shorthand for `commutative semiring'.

\begin{definition}[preordered semidomain~\cite{Fritz2021a}] A preordered semiring $S$ is a {\it preordered semidomain} if it has no zero divisors, i.e.,\ $xy=0$ for $x,y\in S$ implies that $x=0$ or $y=0$, and if $x = 0$ is the only $x \in S$ with $0 \rleq x \rleq 0$.
\end{definition}

Moreover, a semiring $S$ is {\it zerosumfree} if $x+y=0$ for $x,y\in S$ implies $x=0=y$. Further, we define an equivalence relation $\sim$ in a preordered semiring $S$ via $x \sim y$ if there is a finite sequence $z_0,\ldots, z_n\in S$ such that for all $i$ we have $z_{i} \rleq z_{i+1}$ or $z_i \rgeq z_{i+1}$, and also $z_0 = x$ and $z_n = y$, or equivalently if there is a zigzag $x\rleq z_1\rgeq z_2\rleq\cdots\rgeq z_n\rleq y$. In other words, $\sim$ is the equivalence relation generated by $\rleq$.

We denote by $\R_+$ the preordered semiring of non-negative real numbers equipped with the usual sum and product and with the usual $\leq$-relation as preorder. We further define the semiring of {\it tropical reals} $\T\R_+$ as coinciding with $\R_+$ as a preordered set, and where the multiplication and ordering of $\T\R_+$ are the usual ones, but where the sum $+$ is given by $x+y=\max\{x,y\}$ instead. We further utilize $\R_+^{\rm op}$ and $\T\R_+^{\rm op}$, where the superscript $^{\rm op}$ stands for reversing the preorder. (Reversing the preorder on a preordered semiring produces another preordered semiring.)

An important property that a preordered semiring $S$ may or may not have is {\it polynomial growth}, defined as follows.

\begin{definition}[polynomial growth~\cite{Fritz2021a}]
For a preordered semiring $S$, we say that $S$ is of {\it polynomial growth} if it possesses a {\it power universal pair} $(u_-,u_+)$, i.e.,\ nonzero $u_\pm\in S$ such that:
\begin{itemize}
	\item $u_-\rleq u_+$,
	\item For all non-zero $x,y\in S$ such that $x\rleq y$, there is $k\in\N$ such that $u_-^k y\rleq u_+^k x$.
\end{itemize}
If $u_-$ has a multiplicative inverse, this pair can be replaced with a single {\it power universal} $u \rgeq 1$ which has the property that, whenever $x\rleq  y$ for non-zero $x,y\in S$, there is $k\in\N$ such that $y\rleq  x u^k$.
\end{definition}

\begin{definition}[monotone homomorphism]\label{def:monotonehomomorphism} Given a preordered semiring $S$, let $\mb K\in\{\R_+,\R_+^{\rm op},\T\R_+,\T\R_+^{\rm op}\}$. Then a \emph{monotone homomorphism} to $\mb K$ is a map $\Phi:S \to \mb K$ such that for all $a, b \in S$,

\begin{itemize}
    \item $\Phi$ is additive, i.e.~$\Phi(a+b)=\Phi(a)+\Phi(b)$,
    \item $\Phi$ is multiplicative, i.e.~$\Phi(ab)=\Phi(a) \cdot \Phi(b)$, 
    \item $\Phi$ is normalized, i.e.~$\Phi(0)=0$ and $\Phi(1)=1$,
    \item and $\Phi$ is monotone, i.e.~$a \rleq b\implies \Phi(a) \rleq \Phi(b)$.
\end{itemize}
\end{definition}

\noindent As we will see, the last condition specializes to that of Schur convexity in the case of the majorization semiring.

Finally, when $S$ is a preordered semiring, we say that $x,y\in S$ are {\it ordered in large samples} if
\begin{align}\label{eq:asymptotic}
x^n\rleq y^n 
\end{align}
for all sufficiently large $n\in\N$. Moreover, $x$ and $y$ are {\it catalytically ordered} if there is non-zero $a\in S$ such that
\begin{align}\label{eq:catalytic}
   xa \rleq ya.
\end{align}

Now that we have the basic tools of the theory, we state the separation theorems for preordered semirings formally known as \emph{Vergleichsstellens\"atze} in~\cite{Fritz2021a,Fritz2021b}. Roughly speaking, these separation theorems provide us with criteria for deciding whether two given elements are ordered in large samples or catalytically. We present the main theorems in a slightly shorter form with only the results relevant to our present applications.

\begin{theorem}[Theorem 7.15 in \cite{Fritz2021a}]\label{thm:Fritz7.15}
Let $S$ be a zerosumfree preordered semiring of polynomial growth with a power universal element $u$ and such that $0 \rleq 1$. Let $x,y\in S$ be such that $x \neq 0$ and $y$ is power universal. If, for every $\mb K\in\{\R_+,\T\R_+\}$ and every monotone homomorphism $\Phi:S\to\mb K$ such that $\Phi(u)>1$, we have $\Phi(x) < \Phi(y)$, then
\begin{enumerate}[label=(\alph*)]
\item $x^n \rleq y^n$ for all sufficiently large $n\in\N$, and
\item $ax\rleq ay$ for a catalyst of the form $a=\sum_{\ell=0}^n x^\ell y^{n-\ell}$ for all sufficiently large $n\in\N$.
\end{enumerate}
\end{theorem}

\begin{remark}
	\label{tropical_normalization}
	For a tropical homomorphism $\Phi : S \to \T\R_+$ as in the theorem, every power $\Phi(\cdot)^\lambda$ for $\lambda > 0$ is again such a tropical homomorphism, and the inequalities $\Phi(x)^\lambda \le \Phi(y)^\lambda$ hold if and only if the original inequalities $\Phi(x) \le \Phi(y)$ do. This shows that one can restrict to those tropical homomorphisms that satisfy the normalization condition $\log \Phi(u) = 1$.
\end{remark}

If instead $1 \rleq 0$ in $S$, then applying the theorem to the opposite preorder $S^{\mathrm{op}}$ shows that the large-sample and catalytic orders can be characterized in terms of monotone homomorphisms to $\mb K\in\{\R_+^{\rm op},\T\R_+^{\rm op}\}$. This will be helpful in our derivation for the submajorization semiring in the appendices. Although the submajorization semiring discussed later has $0\rleq 1$, it does not have a power universal. The oppositely ordered submajorization semiring has a power universal as required by Theorem~\ref{thm:Fritz7.15}, but there $0$ majorizes $1$.

Theorem~\ref{thm:Fritz7.15} requires $0$ and $1$ to be comparable which restricts its applicability in, e.g.,\ information-theoretic settings. In some cases, e.g.,\ submajorization discussed in the appendices, we do have $0\rleq 1$. This is why we can (re)derive results on approximate large-sample and catalytic majorization of probability vectors using Theorem~\ref{thm:Fritz7.15} as a stepping stone in Section \ref{sec:AuNe}. However, typically in probability and information theory there are naturally occurring preordered semirings where neither $1 \rgeq 0$ nor $1 \rleq 0$ holds. This is often because probability measures are normalized by definition to $1$, and we only want measures of the same normalization to be comparable. In particular, the zero measure will not be comparable to any normalized measure, and this results in $1 \not\rleq 0$ and $1 \not\rgeq 0$. In light of this, we state an additional separation theorem proven in \cite{Fritz2021b} which, in order to account for the incomparability of $1$ and $0$, requires the additional consideration of monotone homomorphisms to $\R_{+}^{\rm op}$ and $\T\R_{+}^{\rm op}$ as well as {\it monotone derivations} that account for certain infinitesimal information.

Before stating the theorem, we briefly introduce the additionally relevant concepts. Let $S$ and $T$ be preordered semirings. We say that a monotone homomorphism $\Phi : S \to T$ is {\it degenerate} if it factors through $S / \!\sim$, or equivalently if for all $x, y \in S$,
\begin{equation}
	x \rleq y \quad\Longrightarrow\quad \Phi(x) = \Phi(y).
\end{equation}
Otherwise $\Phi$ is {\it nondegenerate}. To state the theorem we need one more relevant concept. Let $S$ be a semiring and $\Phi : S \to \R_+$ a homomorphism. Then a {\it $\Phi$-derivation} is an additive map $\Delta : S \to \R$ such that the {\it Leibniz rule}
\begin{equation}
	\Delta(xy) = \Phi(x) \Delta(y) + \Delta(x) \Phi(y)
\end{equation}
holds for all $x, y \in S$.

The next separation theorem applies quite generally, but lets us conclude catalytic orderings only. Let us first make a couple of technical definitions. We say that a (preordered) semiring $S$ has {\it quasi-complements} if, for every $x\in S$, there are $n\in\N$ and $y\in S$ such that $x+y=n$ (where $n$ is the $n$-fold sum of the unit element). If $S$ is also a semidomain, it has {\it quasi-inverses} if, for every non-zero $x\in S$, there are $n\in\N$ and $y\in S$ such that $xy=n$. Let $S$ be a preordered semiring and $\sim$ be the equivalence relation generated by the preorder through sequences with zigzag ordering as earlier. We denote the set of formal fractions over $S/\!\sim$ by $\ms{Frac}(S/\!\sim)$. Introducing additive inverses to this set corresponds to considering the tensor product $\ms{Frac}(S/\!\sim)\otimes\Z$.

\begin{theorem}[Theorem 7.1 in \cite{Fritz2021b}]\label{thm:Fritz7.1}
	Let $S$ be a zerosumfree preordered semidomain of polynomial growth with a power universal pair $(u_-,u_+)$. Also assume that $S/\!\sim$ has quasi-complements and quasi-inverses and that $\ms{Frac}(S/\!\sim)\otimes\Z$ is a finite product of fields. Let nonzero $x,y\in S$ with $x\sim y$ be given. Suppose that, for every $\mb K\in\{\R_+,\R_+^{\rm op},\T\R_+,\T\R_+^{\rm op}\}$ and every monotone homomorphism $\Phi:S\to\mb K$ with trivial kernel, we have:
\begin{enumerate}[label=(\roman*)]
\item If $\Phi$ is nondegenerate, then $\Phi(x)<\Phi(y)$.
\item If $\mb K=\R_+$ and $\Phi$ is degenerate, then $\Delta(x)<\Delta(y)$ for every nonzero monotone $\Phi$-derivation $\Delta$ with $\Delta(u_+)= \Delta(u_-)+1$.
\end{enumerate}
Then there is non-zero $a\in S$ such that $ax \rleq ay$. Conversely, if such $a$ exists, then the above inequalities hold non-strictly.
\end{theorem}

To characterize ordering in large samples as well, we need some further conditions on the preordered semiring involved. It can be shown \cite{Fritz2021b} that the conditions listed in the theorem below imply the conditions required in Theorem \ref{thm:Fritz7.1}.

\begin{theorem}[Theorem 8.6 in \cite{Fritz2021b}]\label{thm:Fritz8.6}
Let $S$ be a zerosumfree preordered semidomain with a power universal element $u$. Assume that for some $d\in\N$ there is a surjective homomorphism $\|\cdot\|:S\to\R_{>0}^d\cup\{(0,\ldots,0)\}$ with trivial kernel and such that
\begin{equation}
a\rleq b\ \Rightarrow\ \|a\|=\|b\|\quad {\rm and} \quad \|a\|=\|b\|\ \Rightarrow\ a\sim b.
\end{equation}
Denote the component homomorphisms of $\|\cdot\|$ by $\|\cdot\|_{(j)}$, $j=1,\ldots,d$. Let $x,y\in S\setminus\{0\}$ with $\|x\|=\|y\|$ and where $y$ is power universal. If,
\begin{enumerate}[label=(\roman*)]
	\item for every $\mb K\in\{\R_+,\R_+^{\rm op},\T\R_+,\T\R_+^{\rm op}\}$ and every nondegenerate monotone homomorphism $\Phi : S \to \mb K$ with trivial kernel, we have $\Phi(x) < \Phi(y)$ and
	\item $\Delta(x) < \Delta(y)$ for every monotone $\|\cdot\|_{(j)}$-derivation $\Delta : S \to \R$ with $\Delta(u) = 1$ for all component indices $j = 1,\ldots,d$,
\end{enumerate}
then
\begin{enumerate}[label=(\alph*)]
\item $x^n \rleq y^n$ for all sufficiently large $n\in\N$, and
\item $ax\rleq ay$ for a catalyst of the form $a=\sum_{\ell=0}^n x^\ell y^{n-\ell}$ for all sufficiently large $n\in\N$.
\end{enumerate}
Conversely if either of these properties holds for any $n$ or $a$, then the above inequalities hold non-strictly.
\end{theorem}

\begin{remark}
	We actually do not have to consider all derivations above, but it is enough to focus on representative derivations whose differences do not factor through $\sim$-equivalence, i.e.,\ when the derivations are not {\it interchangeable}~\cite{Fritz2021b}. Remark~\ref{tropical_normalization} on the normalization of tropical homomorphisms also still applies, and likewise the homomorphisms $S \to \T\R_+^{\rm op}$ can be normalized to $\log \Phi(u) = -1$.
\end{remark}

The intended applications of the theory presented here deal with questions of transformability of one finite statistical experiment (typically described by a matrix) to another in large samples or catalytically. This comparison of experiments gives the preorder we study for the majority of this work. In this framework, the monotone homomorphisms and derivations giving conditions for the large-sample and catalytic transformability in the above separation theorems can be interpreted as monotones quantifying the information content in an experiment. Later we will briefly discuss the applications to quantum theory where these maps should be thought of as resource monotones for quantum state resources.

\section{Matrix majorization}\label{sec:MatrixMaj}

We now turn our attention to the comparison of tuples of probability vectors. Given two tuples of probability vectors $(p^{(1)}, \dots, p^{(d)})$ and $(q^{(1)}, \dots, q^{(d)})$, our main question is to determine whether, for $n\in\N$ large enough, there is a stochastic map $T$ such that
\begin{equation}
T\left(p^{(k)}\right)^{\otimes n}=\left(q^{(k)}\right)^{\otimes n} \qquad \forall k = 1, \ldots, d,
\end{equation}
which is \emph{matrix majorization in large samples}. We also address the related question of \emph{catalytic matrix majorization}. We study these questions by constructing a suitable preordered semiring, studying its monotone homomorphisms into $\mb K\in\{\R_+,\R_+^{\rm op},\T\R_+,\T\R_+^{\rm op}\}$ as well as monotone derivations on it, and then applying the separation theorem \ref{thm:Fritz8.6}.

Let us start by introducing some notation that we will use throughout this section. We fix $d\in\N$ and denote, for all $n\in\N$, the set of $(n\times d)$-matrices with non-negative entries and columns with a common support by $\mc V_n^d$. We identify a matrix $P\in\mc V_n^d$ with the associated tuple of columns $P = (p^{(1)},\ldots,p^{(d)})$, and we generally use the corresponding lower-case symbol for these columns. In terms of this, every $P=(p^{(1)},\ldots,p^{(d)})\in\mc V^d_n$ satisfies
\begin{align}
{\rm supp}\,p^{(1)}&=\cdots={\rm supp}\,p^{(d)}.
\end{align}
We also define
\begin{equation}\label{eq:PositiveCone}
\mc V^d_{<\infty}:=\bigcup_{n=1}^\infty \mc V^d_n.
\end{equation}
Let us define two binary operations on $\mc V^d_{<\infty}$: For $P,Q\in\mc V^d_{<\infty}$, we define a new matrix given by stacking $P$ and $Q$ on top of each other,
\begin{equation}
P\boxplus Q=\left(\begin{array}{c}
P\\
\hline
Q
\end{array}\right).
\end{equation}
In terms of our tuple notation above, this is equivalent to
\begin{equation}
	P\boxplus Q=(p^{(1)}\oplus q^{(1)},\ldots,p^{(d)}\oplus q^{(d)}).
\end{equation}
We similarly define $P\boxtimes Q$ as the matrix whose columns are the Kronecker products of the columns of $P$ and $Q$,
\begin{equation}
	P\boxtimes Q=(p^{(1)}\otimes q^{(1)},\ldots,p^{(d)}\otimes q^{(d)}),
\end{equation}
where the ordering of rows will be irrelevant modulo the equivalence $\approx$ introduced below.

Let us also denote, for every $n\in\N$, by $0_{n\times d}$ the $(n\times d)$-matrix whose entries are all zero. For $P,Q\in\mc V^d_{<\infty}$ where $P\in\mc V_m^d$ and $Q\in\mc V_n^d$, we write $P\approx Q$ if there are $t\geq m,n$ and a $(t\times t)$-permutation matrix $\Pi$ such that
\begin{equation}
\Pi(P\boxplus 0_{(t-m)\times d})=Q\boxplus 0_{(t-n)\times d}.
\end{equation}
This means that, modulo padding by all-zero rows, $Q$ is obtained from $P$ by a row permutation. We denote the equivalence classes in $\mc V^d_{<\infty}/\!\approx$ by $[P]$ for all $P\in\mc V^d_{<\infty}$, i.e.,\ $[P]:=\{Q\in\mc V^d_{<\infty}\,|\,Q\approx P\}$. In the sequel, we are only interested in the $\approx$-equivalence classes of the matrices under consideration and usually, when referring to a matrix $P\in\mc V^d_{<\infty}$, we actually refer to the equivalence class $[P]$. This also allows us to assume that, given $P,Q\in\mc V^d_{<\infty}$, $P$ and $Q$ are of the same height, i.e.,\ $P,Q\in\mc V^d_n$ for some large enough $n\in\N$.

Let us define $0=[(0\,\cdots\,0)]$ and $1=[(1\,\cdots\,1)]$ as the equivalence classes of $(1\times d)$-row matrices, and note that addition and multiplication in $\mc V^d_{<\infty}/\!\approx$ in the form
\begin{equation}
[P]+[Q]=[P\boxplus Q],\qquad [P]\cdot[Q]=[P\boxtimes Q]
\end{equation}
are well-defined. We define the preorder $\rleq$ on $\mc V^d_{<\infty}/\!\approx$ by declaring that $[Q]\rleq[P]$ if there is a column-stochastic matrix $T$ such that $Q=TP$. In this case we also denote $P\succeq Q$ or $Q\preceq P$ and say that $P$ majorizes $Q$. In terms of the columns, this means $q^{(k)}=Tp^{(k)}$ for all $k=1,\ldots,d$. One may easily check that $(\mc V^d_{<\infty}/\!\approx,0,1,+,\cdot,\succeq)$ is a preordered semiring and we make the following definition:

\begin{definition}
We denote the preordered semiring $(\mc V^d_{<\infty}/\!\approx,0,1,+,\cdot,\succeq)$ by $S^d$ and call it the {\it matrix majorization semiring (of length $d$)}.
\end{definition}

\begin{remark}
For algebraically minded readers, the following generalization of Remark~\ref{S_group_semiring} may be helpful: as a plain semiring, $S^d$ coincides with the group semiring $\N[\R_{>0}^d]$, since writing every matrix as the $\boxplus$-sum of its rows show that the semiring elements can be identified with formal sums of elements of the multiplicative group $\R_{>0}^d$, and this identification is such that the usual addition and multiplication of formal sums correspond to the operations introduced above.
\end{remark}

Recall from Section~\ref{sec:preorderedSemirings} that $\sim$ denotes the equivalence relation generated by $\rleq$, i.e.,\ $[P]\sim[Q]$ (or, simply, $P\sim Q$) if there are $n\in\N$ and $R_1,\ldots,R_n\in\mc V^d_{<\infty}$ such that
\begin{equation}
P\preceq R_1\succeq R_2\preceq\cdots\succeq R_n\preceq Q.
\end{equation}
In fact, the above sequence $R_1,\ldots,R_n$ can be chosen to be of length $1$. To see this, note that such a chain of inequalities implies $\|p^{(k)}\|_1=\|q^{(k)}\|_1$ for all column indices $k=1,\ldots,d$, since multiplying by a stochastic matrix does not change the normalization of any column. But then we also obtain 
\begin{equation}
P\succeq R\preceq Q,
\end{equation}
where $R:=\big(\|p^{(1)}\|_1,\ldots,\|p^{(d)}\|_1\big)$ has just a single row. From this we also see that
\begin{equation}
S^d / \!\sim\ \cong \R_{>0}^d\cup\{(0,\ldots,0)\},
\end{equation}
which is a semiring with quasi-complements and quasi-inverses and such that $\ms{Frac}(S^d / \!\sim) \otimes \Z \cong \R^d$ is a finite product of fields. Defining the degenerate homomorphism $\|\cdot\|:S^d\to\R_{>0}^d\cup\{(0,\ldots,0)\}$ by
\begin{equation}
\|P\| := \big(\|p^{(1)}\|_1,\ldots,\|p^{(d)}\|_1\big),
\end{equation}
we conclude
\begin{equation}
P\succeq Q\ \Rightarrow\ \|P\|=\|Q\| \quad {\rm and}\quad \|P\|=\|Q\|\ \Rightarrow\ P\sim Q.
\end{equation}
Therefore we are in a situation where the auxiliary assumptions of Theorem \ref{thm:Fritz8.6} hold.

\begin{lemma}
	\label{lem:MatrPowUniv}
	The matrix majorization semiring is of polynomial growth: any stochastic matrix $P=(p^{(1)},\ldots,p^{(d)})$ with $p^{(k)}\neq p^{(\ell)}$ for $k\neq\ell$ is a power universal.
\end{lemma}

\begin{proof}
To show that such $P$ is a power universal, we need to start with $R,Q\in\mc V^d_{<\infty}\setminus\{0\}$ such that $Q \succeq R$ and show that there is $n \in \N$ such that
\begin{equation}
	P^{\boxtimes n} \boxtimes R \succeq Q.
\end{equation}
By columnwise renormalization, we may assume that both $R$ and $Q$ are stochastic, and we also assume without loss of generality that all entries of $Q$ are strictly positive. If we prove that there is $n\in\N$ such that $P^{\boxtimes n}\succeq Q$, then we have
\begin{equation}
P^{\boxtimes n}\boxtimes R\succeq P^{\boxtimes n}\boxtimes(1\,\cdots\,1)=P^{\boxtimes n}\succeq Q,
\end{equation}
and we are done. Thus we will show that there is $n \in \N$ with $P^{\boxtimes n}\succeq Q$.

Using a result from multiple hypothesis testing~\cite{LeJo97}, there is $(\eps^{(n)}_{k,\ell})_{k,\ell=1,\ldots,d}$ such that
\begin{equation}
P^{\boxtimes n}\succeq U^{(n)}=\left(\begin{array}{cccc}
1-\varepsilon_{1,1}^{(n)}&\varepsilon_{1,2}^{(n)}&\cdots&\varepsilon_{1,d}^{(n)}\\
\varepsilon_{2,1}^{(n)}&1-\varepsilon_{2,2}^{(n)}&\cdots&\varepsilon_{2,d}^{(n)}\\
\vdots&\vdots&\ddots&\vdots\\
\varepsilon_{d,1}^{(n)}&\varepsilon_{d,2}^{(n)}&\cdots&1-\varepsilon_{d,d}^{(n)}
\end{array}\right)
\end{equation}
and such that all the errors $\varepsilon_{k,\ell}^{(n)}$ decay exponentially in $n$ at a rate given by the multiple Chernoff divergence of $P$,
\begin{equation}
C(P):=\min_{k,\ell: k\neq\ell}\max_{0\leq\alpha\leq1}(1-\alpha)D_\alpha(p^{(k)}\|p^{(\ell)}),
\end{equation}
where $D_\alpha(\cdot\|\cdot)$ are the R\'{e}nyi divergences presented later in\eqref{eq:Renyirelentr}. This quantity is strictly positive, which can be seen e.g.~by noting that the Bhattacharyya distance $D_{1/2}$ satisfies $D_{1/2}(p^{(k)}\|p^{(\ell)}) \neq 0$ as soon as $p^{(k)} \neq p^{(\ell)}$, which we have assumed for all $k \neq \ell$. Therefore we obtain $\eps^{(n)}_{k,\ell} \to 0$ as $n \to \infty$, and hence $U^{(n)}\to I_d$ as $n\to\infty$, where $I_d$ is the $(d\times d)$-identity matrix. In particular, $U^{(n)}$ is invertible for sufficiently large $n$, and we write $V^{(n)} := (U^{(n)})^{-1}$ for its inverse. Since the inverse of a matrix with normalized columns also has normalized columns\footnote{With $e = (1 \cdots 1)$, we have $e V^{(n)} = e U^{(n)} V^{(n)} = e$, where the first step uses that $U^{(n)}$ has normalized columns and the second that $V^{(n)}$ is the inverse of $U^{(n)}$.}, we know that this applies to $V^{(n)}$.

In order to show that $U^{(n)}\succeq Q$ for sufficiently large $n$, consider now the matrix
\begin{equation}
	T^{(n)} := Q V^{(n)}.
\end{equation}
Then $T^{(n)} U^{(n)} = Q$ holds by construction, and hence it is enough to show that $T^{(n)}$ is stochastic. But since we already know that it has normalized columns since both $V^{(n)}$ and $Q$ have\footnote{Use $e Q V^{(n)} = e V^{(n)} = e$.}, we only still need to prove that its entries are nonnegative. And this is the case since all entries of $Q$ are strictly positive, and $V^{(n)} \to I_d$ as $n \to \infty$ by continuity of matrix inversion and $U^{(n)} \to I_d$.
\end{proof}

Let us define an important family of real functions on $S^d$. With a slight abuse of notation, in order to avoid multiple parenthesis in our formulas, we actually treat these as functions on $\mc V^d_{<\infty}$ where individual matrices are identified with their $\approx$-equivalence class. For every $P \in\mc V^d_n\subset\mc V^d_{<\infty}$, where we omit zero-rows (as they do not contribute to these functions), we write
\begin{align}
	f_{\mathuline{\alpha}}(P)= {}&\sum_{i=1}^n\prod_{k=1}^d \left( p^{(k)}_i \right)^{\alpha_k},\label{eq:MatrMonHomR}\\
	f^\T_{\mathuline{\beta}}(P)= {}&\max_{1\leq i\leq n} \prod_{k=1}^d \left( p_i^{(k)} \right)^{\beta_k}, \label{eq:MatrMonHomTR}\\
	\Delta^{(k)}_{\mathuline{\gamma}}(P)= {}&\sum_{i=1}^n\sum_{\ell:\,\ell\neq k} \gamma_\ell \, p^{(k)}_i\log{\frac{p^{(k)}_i}{p^{(\ell)}_i}} \nonumber\\
	= {}&\sum_{\ell:\,\ell\neq k} \gamma_\ell D_1(p^{(k)}\|p^{(\ell)}),\label{eq:MatrDeriv}
\end{align}
for all tuples $\mathuline{\alpha}, \mathuline{\beta}, \mathuline{\gamma} \in \R^d$ and $k = 1, \dots, d$. Note that $D_1(\cdot\|\cdot)$ in \eqref{eq:MatrDeriv} is the Kullback-Leibler divergence. However, we will mostly be interested in more particular parameter ranges. Let us thus define the parameter sets
\begin{align}
A\  & := {}\{\mathuline{\alpha} \in \R^d \,|\,\alpha_1+\cdots+\alpha_d=1\}, \\
A_- & := {}\bigcup_{k=1}^d \: \left\{\mathuline{\alpha} \in A\,\middle|\,\alpha_k\geq 1 \:\land\: \alpha_\ell\leq0\;\: \forall \ell\neq k\right\}, \label{A-} \\
A_+ & := {}\left\{\mathuline{\alpha} \in A\,\middle|\,\alpha_\ell\geq0\;\: \forall \ell\right\},\\[4pt]
B\  & := {}\{\mathuline{\beta} \in \R^d \,|\,\beta_1+\cdots+\beta_d=0\}, \\
B_- & := {}\bigcup_{k=1}^d \: \left\{\mathuline{\beta} \in B\,\middle|\,\beta_k\geq 0 \:\land\: \beta_\ell\leq0\;\: \forall \ell\neq k\right\}. \label{B-}
\end{align}
For example, $A_-$ consists of all normalized tuples of numbers with exactly one positive entry, which by the normalization is necessarily $\ge 1$, while $A_+$ is the standard probability simplex in $\R^d$. The quantity $f_{\mathuline{\alpha}}(P)$ considered as a function of $\mathuline{\alpha} \in A_+$ is also known as the \emph{Hellinger transform}~\cite[Example~1.4.4]{Torgersen91}.

\subsection{Monotone homomorphisms and derivations on the matrix majorization semiring}

We next characterize the monotone homomorphisms on $S^d$ in terms of \eqref{eq:MatrMonHomR} and \eqref{eq:MatrMonHomTR}, writing $e_1, \dots, e_d \in \R^d$ for the standard basis vectors.

\begin{proposition}\label{prop:MatrMonHom}
The nondegenerate monotone homomorphisms $S^d\to\mb K$ are exactly the following:
\begin{enumerate}[label=(\roman*)]
\item For $\mb K=\R_+$, the maps $f_{\mathuline{\alpha}}$ for all tuples $\mathuline{\alpha}\in A_- \setminus \{e_1, \ldots, e_d\}$.
\item For $\mb K=\R_+^{\rm op}$, the maps $f_{\mathuline{\alpha}}$ for all tuples $\mathuline{\alpha}\in A_+ \setminus \{e_1, \ldots, e_d\}$.
\item For $\mb K=\T\R_+$, the maps $f^\T_{\mathuline{\beta}}$ for all tuples $\mathuline{\beta}\in B_- \setminus \{0\}$.
\item For $\mb K=\T\R_+^{\rm op}$, there is none.
\end{enumerate}
\end{proposition}

\begin{proof}
Let $\Phi:S^d\to\mb K$ be such a monotone homomorphism. In each case, let us define the functions $\fii_k:\R_{>0}\to\R_+$ for all $k=1,\ldots,d$ through
\begin{equation}
\fii_k(x)=\Phi(1\,\cdots\,1\,x\,1\,\cdots\,1),
\end{equation}
where the $x$ appears in the $k$-th position. One immediately sees that these functions are multiplicative, i.e.,\ $\fii_k(xy)=\fii_k(x)\fii_k(y)$ for all $x,y>0$ and $\fii_k(1) = 1$. We next show that they are power functions $x \mapsto x^{\alpha_k}$ for certain exponents $\alpha_k$. By the multiplicativity and the standard theory of the Cauchy functional equation, it is enough to show that they are bounded on some interval $[x,y]$ for $0<x<y$. Defining, for all $t\in[0,1]$, the matrices
\begin{align}
T_t&=\left(\begin{array}{cc}
t&1-t\\
1-t&t
\end{array}\right),\\
P_t&=\left(\begin{array}{ccccccc}
	1 & \cdots & 1 & tx+(1-t)y & 1 & \cdots & 1 \\
	1 & \cdots & 1 & (1-t)x+ty & 1 & \cdots & 1
\end{array}
\right),
\end{align}
we see that $T_t$ is stochastic and
\begin{equation}
P_{1/2} = T_{1/2} P_t\preceq P_t=T_t P_1\preceq P_1.
\end{equation}
Applying $\Phi$ to these inequalities, and using the fact that $P_t$ can be written as the sum of its two rows in $S^d$, we find that
\begin{align}
\textrm{case (i): } 2\fii_k\left(\frac{1}{2}(x+y)\right)\leq{}&\fii_k\big(tx+(1-t)y\big)+\fii_k\big((1-t)x+ty\big)\\
\leq{}&\fii_k(x)+\fii_k(y),\\
\textrm{case (ii): } 2\fii_k\left(\frac{1}{2}(x+y)\right)\geq{}&\fii_k\big(tx+(1-t)y\big)+\fii_k\big((1-t)x+ty\big)\\
\geq{}&\fii_k(x)+\fii_k(y),\\
\textrm{case (iii): } \fii_k\left(\frac{1}{2}(x+y)\right)\leq{}&\max_{s\in\{t,1-t\}}\fii_k\big(sx+(1-s)y\big)\\
\leq{}&\max\{\fii_k(x),\fii_k(y)\},\\
\textrm{case (iv): } \fii_k\left(\frac{1}{2}(x+y)\right)\geq{}&\max_{s\in\{t,1-t\}}\fii_k\big(sx+(1-s)y\big)\\
\geq{}&\max\{\fii_k(x),\fii_k(y)\}.
\end{align}
Since $\fii_k$ takes non-negative values, we now see that, for all $z\in[x,y]$,
\begin{align}
\textrm{case (i): }& 0\leq\fii_k(z)\leq\fii_k(x)+\fii_k(y),\\
\textrm{case (ii): }& 0\leq\fii_k(z)\leq2\fii_k\left(\frac{1}{2}(x+y)\right),\\
\textrm{case (iii): }& 0\leq\fii_k(z)\leq\max\{\fii_k(x),\fii_k(y)\},\\
\textrm{case (iv): }& 0\leq\fii_k(z)\leq\fii_k\left(\frac{1}{2}(x+y)\right).
\end{align}
In particular, $\fii_k$ is indeed bounded on the interval $[x,y]$. Thus, there is $\alpha_k\in\R$ such that $\fii_k(x)=x^{\alpha_k}$ for all $x>0$.

In cases (i) and (ii), we now have, for all $P$ of full support,
\begin{align}
\Phi(P)&=\sum_{i=1}^n \Phi \big( (p^{(1)}_i\,\cdots\,p^{(d)}_i) \big)\\
&=\sum_{i=1}^n\Phi\big((p^{(1)}_i\,1\,\cdots\,1)\boxtimes\cdots\boxtimes(1\,\cdots\,1\,p^{(d)}_i)\big)\\
&=\sum_{i=1}^n\prod_{k=1}^d \fii_k(p^{(k)}_i)=\sum_{i=1}^n\prod_{k=1}^d(p^{(k)}_i)^{\alpha_k},\label{eq:MatrMonHomRapu}
\end{align}
where $n$ is the number of rows of $P$.
Similarly, in cases (iii) and (iv), we find that
\begin{equation}\label{eq:MatrMonHomTRapu}
\Phi(P)=\max_{1\leq i\leq n}\prod_{k=1}^d (p^{(k)}_i)^{\alpha_k}.
\end{equation}

Since the functions of this form are indeed clearly homomorphisms, it only remains to determine for which parameter tuples they are monotone and nondegenerate. One easily sees
\begin{equation}
\left(\begin{array}{ccc}
1&\cdots&1\\
1&\cdots&1
\end{array}
\right)\succeq\big(2\,\cdots\,2\big)\succeq\left(\begin{array}{ccc}
1&\cdots&1\\
1&\cdots&1
\end{array}
\right),
\end{equation}
implying, due to monotonicity, that
\begin{equation}
\Phi\left(\begin{array}{ccc}
1&\cdots&1\\
1&\cdots&1
\end{array}
\right)=\Phi(2\,\cdots\,2).
\end{equation}
In cases (i) and (ii), this means that $\alpha_1+\dots+\alpha_d=1$ and, in cases (iii) and (iv), this means that $\alpha_1+\cdots+\alpha_d=0$.

Let us focus on (i) and (ii) for the moment. There, we now know that $\Phi$ must be of the form $f_{\mathuline{\alpha}}$ with some parameters $\mathuline{\alpha} \in A$. Let us now show that $\mathuline{\alpha} \in A_-$ in case (i) and $\mathuline{\alpha} \in A_+$ in case (ii) are necessary. For $P,Q\in\mc V^d_n$ and $t\in[0,1]$, and with the $(n\times 2n)$-column-stochastic matrix $T$ defined by $T(a\oplus b)=a+b$ for all $a,b\in\R^n$, note that in case (i) we have
\begin{align}
	\Phi\big(tP+(1-t)Q\big)= {}&\Phi\big(T\big[tP\boxplus(1-t)Q\big]\big) \\
		\le {} & \Phi\big(tP\boxplus(1-t)Q\big)\\
		= {}&\Phi(tP)+\Phi\big((1-t)Q\big) \\
		= {}& t\Phi(P)+(1-t)\Phi(Q), \label{eq:PhiConvex}
\end{align}
i.e.,\ $\Phi$ is convex. The final equality holds due to $\alpha_1 + \dots + \alpha_d = 1$. Similarly, in case (ii), we find that $\Phi$ must be concave. Let us define the single-row function $\fii:\R_{>0}^d\to\R_+$ through
\begin{equation}
\fii(x_1,\ldots,x_d) := \Phi \big( (x_1\,\cdots\,x_d) \big) = x_1^{\alpha_1}\cdots x_d^{\alpha_d}
\end{equation}
for all $(x_1,\ldots,x_d)\in\R_{>0}^d$. Naturally, $\fii$ must also be convex in case (i) and concave in case (ii). To analyze the convexity properties of $\fii$, let us compute its Hessian matrix,
\begin{align}
H_\fii(x_1,\ldots,x_d)&=\left(\frac{\partial^2\fii}{\partial x_i\, \partial x_j}(x_1,\ldots,x_d)\right)_{i,j=1}^d\\
&=\fii(x_1,\ldots,x_d)\left(\frac{1}{x_i x_j}\right)_{i,j=1}^d\star\underbrace{(\alpha_i\alpha_j-\delta_{i,j}\alpha_i)_{i,j=1}^d}_{=: \mathcal{A}}
\end{align}
where $\star$ denotes the Schur product, i.e.,\ the entry-wise product of matrices of the same shape. Because Schur multiplying with a positive semi-definite matrix preserves the semi-definiteness of a matrix and $\fii$ takes positive values, we find that, upon Schur multiplying with the positive semi-definite matrix $(x_i x_j)_{i,j=1}^d$, that $H_\fii$ is positive (respectively negative) semi-definite if and only if $\mathcal{A}$ is positive (respectively negative) semi-definite. According to Lemma 8 of \cite{Mu_et_al_2020}, we have:
\begin{itemize}
	\item $\mathcal{A}$ is positive semi-definite if and only if $\mathuline{\alpha} \in A_-$.
	\item $\mathcal{A}$ is negative semi-definite if and only if $\mathuline{\alpha} \in A_+$.
\end{itemize}
Therefore the convexity (concavity) of $\fii$ implies the conditions for the parameters $\mathuline{\alpha}$ mentioned in the claim, where the case that $\mathuline{\alpha}$ is a standard basis vector $e_k$ can be excluded since then $f_{\mathuline{\alpha}}$ is given by $f_{\mathuline{\alpha}}(P) = \|p^{(k)}\|_1$, which is degenerate.

Next we show that all these homomorphisms are indeed monotone and nondegenerate. So for case (i), let $\mathuline{\alpha} \in A_-$, in which case we know per the above that $\fii$ is convex. Let $P \in\mc V^d_n$ with rows $p_1, \dots, p_n$ and let $T=(T_{i,j})_{i,j}$ be an $(m\times n)$-stochastic matrix, for which we assume no zero rows without loss of generality and write $T_i:=\sum_{j=1}^n T_{i,j}$. We write $Q := TP \in\mc V^d_m$ with rows $q_1, \dots, q_m$. Since $\fii$ also preserves scalar multiplication due to $\alpha_1 + \dots + \alpha_d = 1$, we have
\begin{align}
f_{\mathuline{\alpha}}(Q)&=\sum_{i=1}^m \fii(q_i) = \sum_{i=1}^m \fii \left( \sum_{j=1}^n T_{i,j} p_j \right)=\sum_{i=1}^m T_i \, \fii \left( \sum_{j=1}^n \frac{T_{i,j}}{T_i} p_j \right) \\
&\le \sum_{i=1}^m T_i \, \sum_{j=1}^n \frac{T_{i,j}}{T_i} \fii(p_j)=\sum_{i=1}^m \sum_{j=1}^n T_{i,j} \fii(p_j) = \sum_{j=1}^n \fii(p_j) = f_{\mathuline{\alpha}}(P),
\end{align}
where the second step is by the preservation of scalar multiplication and the inequality follows by the convexity of $\fii$. Thus, $f_{\mathuline{\alpha}}:S^d\to\R_+$ is indeed monotone for $\alpha \in A_-$. For case (ii), the same argument with the inequality the other way shows that $f_{\mathuline{\alpha}}:S^d\to\R_+^{\rm op}$ is monotone for all $\alpha \in A_+$. To finish the proof in cases (i) and (ii), it remains to prove nondegeneracy whenever $\mathuline{\alpha}$ is not a standard basis vector. Let us show that, whenever there is an index $k$ with $\alpha_k\neq 0,1$, then $f_{\mathuline{\alpha}}$ is nondegenerate. Define the matrix
\begin{equation}\label{eq:Qt}
P_t:=\left(\begin{array}{ccccccc}
	\frac{1}{2} & \cdots & \frac{1}{2} & 1-t&\frac{1}{2}&\cdots&\frac{1}{2}\\[2pt]
	\frac{1}{2} & \cdots & \frac{1}{2} & t &\frac{1}{2}&\cdots&\frac{1}{2}
\end{array}\right)\in\mc V^d_{<\infty},
\end{equation}
where the nontrivial column is the $k$-th. Note that, whenever $0<s\leq t\leq \frac{1}{2}$, then $P_s\succeq P_t$. One easily sees that the values
\begin{equation}
f_{\mathuline{\alpha}}(P_t)=2^{\alpha_k-1}\big((1-t)^{\alpha_k}+t^{\alpha_k}\big)
\end{equation}
are not constant in $t$ whenever $\alpha_k \neq 0,1$, showing that $f_{\mathuline{\alpha}}$ is non-degenerate whenever $\mathuline{\alpha}$ is not a standard basis vector. Thus we have proven the claim as far as cases (i) and (ii) are concerned.

Let us now turn to the tropical cases (iii) and (iv). In case (iii), we can use reasoning analogous to case (i) to show that $\Phi$ is quasi-convex, i.e.,
\begin{equation}
\Phi\big(tP+(1-t)Q\big)\leq\max\{\Phi(P),\Phi(Q)\}
\end{equation}
for all $P,Q\in\mc V^d_{<\infty}$ and $t\in(0,1)$. In particular, the single-row function $\fii$ has to be a quasi-convex function on $\R^d_{>0}$. This quasi-convexity is equivalent to convexity of the sub-level sets
\begin{equation}
S_c:=\{(x_1,\ldots,x_d)\in\R_{>0}^d\,|\,x_1^{\alpha_1} \cdots x_d^{\alpha_d} \leq c\}
\end{equation}
for all $c > 0$. Let us assume that $\alpha_1 > 0$ without loss of generality (by permuting columns if necessary). Considering then the first coordinate $x_1$ as the dependent variable, we observe that $S_c$ contains precisely all points that are non-strictly below the graph of the function $w:\R_{>0}^{d-1}\to\R$ given by
\begin{equation}
w(x_2,\ldots,x_d)=\frac{x_2^{\mu_2}\cdots x_d^{\mu_d}}{c},
\end{equation}
where $\mu_k := -\frac{\alpha_k}{\alpha_1}$ for $k = 2, \dots, d$, which are now coefficients that sum to $1$ again. As we already saw earlier, the Hessian matrix of $w$ is easily found to be
\begin{equation}
H_w(x_2,\ldots,x_d)=w(x_1,\ldots,x_d)\left(\frac{1}{x_i x_j}\right)_{i,j=2}^{d}\star\underbrace{(\mu_i\mu_j-\delta_{i,j}\mu_i)_{i,j=2}^d}_{=:\,\mathcal{B}}.
\end{equation}
So for the sets $S_c$ to be convex, $w$ must be concave, or equivalently $H_w$ must be negative semi-definite. This is equivalent, by the same argument as above, to $\mu_2, \dots, \mu_d \ge 0$, or equivalently $\alpha_k \le 0$ for $k \ge 2$. Therefore $\Phi = f^\T_{\mathuline{\alpha}}$ is monotone only if $\alpha \in B_-$. Since $f^\T_0$ is the homomorphism that simply maps all nonzero $P$ to $1$, this homomorphism is degenerate, and we can conclude $\mathuline{\alpha} \in B_- \setminus \{0\}$. This establishes one direction of (iii).

For the same and only direction in case (iv), let now $\Phi : S^d \to \T\R_+^{\rm op}$ be a nondegenerate monotone homomorphism. Similarly as in case (iii), we now have
\begin{align}\label{eq:tropicalopposite}
\Phi\big(tP+(1-t)Q\big)\geq\max\{\Phi(P),\Phi(Q)\}.
\end{align}
for all $P, Q \in\mc V^d_{<\infty}$ and $t \in (0,1)$. We next show that, if $P \neq Q$ have the same support, then $\Phi(P)=\Phi(Q)$. In that case, we can find $s,t\in(0,1)$ and $P',Q'\in\mc V^d_{<\infty}$ such that
\begin{equation}
sP+(1-s)P'=Q,\qquad tQ+(1-t)Q'=P.
\end{equation}
Thus, according to \eqref{eq:tropicalopposite}, we have
\begin{equation}
\Phi(P)\geq\max\{\Phi(Q),\Phi(Q')\}\geq\Phi(Q)\geq\max\{\Phi(P),\Phi(P')\}\geq\Phi(P),
\end{equation}
i.e.,\ $\Phi(P)=\Phi(Q)$. This means that $\Phi$ is simply a function of the support of its argument. But then if $Q$ is a stochastic matrix of rank $1$, then
\begin{equation}
Q\succeq (1\,\cdots\,1) \succeq Q,
\end{equation}
so that $\Phi(Q)=1$. But then if $P\in\mc V^d_{<\infty}$ is nonzero, we can find any stochastic matrix of rank $1$ with the same support, and therefore conclude $\Phi(P) = 1$. Therefore there is no nondegenerate homomorphism to $\T\R_+^{\rm op}$, and case (iv) is finished.

Finally, we need to show that the maps $f^\T_{\mathuline{\alpha}}$ with $\mathuline{\alpha} \in B_- \setminus \{0\}$ are nondegenerate monotone homomorphisms $S^d \to \T\R_+$. The fact that they are homomorphisms is straightforward to verify, so let us focus on monotonicity. Per the above, we know that the single-row function $\fii$ is quasi-convex. We also know that $f^\T_{\mathuline{\alpha}}(t P) = f^\T_{\mathuline{\alpha}}(P)$ for all positive scalars $t$ by $\alpha_1 + \dots + \alpha_d = 0$. Then with the same $P$ and $Q=TP$ as above in the corresponding proof for case (i), we have a similar argument as given there,
\begin{align}
f^\T_{\mathuline{\alpha}}(Q) & = \max_{1\leq i\leq m} \fii(q_i) = \max_{1\leq i\leq m} \fii \left( \sum_{j=1}^n T_{i,j} p_j \right)\\
& = \max_{1\leq i\leq m} \fii \left( T_i\sum_{j=1}^n \frac{T_{i,j}}{T_i} p_j \right) = \max_{1\leq i\leq m} \fii \left(\sum_{j=1}^n \frac{T_{i,j}}{T_i} p_j \right)\\
& \le \max_{1\leq i\leq m} \max_{1\leq j\leq 1} \fii(p_j) = \max_{1\leq j\leq n} \fii(p_j) = f^\T_{\mathuline{\alpha}}(P).
\end{align}
Let us show that all these homomorphisms are non-degenerate if $\mathuline{\alpha} \neq 0$, or equivalently if there is $k$ with $\alpha_k > 0$. With $P_t$ as in \eqref{eq:Qt}, we get 
\begin{equation}
f^\T_{\mathuline{\alpha}}(P_t)= \left( 2 (1-t) \right)^{\alpha_k}
\end{equation}
for all $t\in(0,\frac{1}{2}]$. This is clearly a non-constant function, meaning that $f^\T_{\mathuline{\alpha}}$ is non-degenerate.
\end{proof}

Let us now go on to studying the monotone derivations at the degenerate homomorphisms $f_{e_k} : S^d\to\R_+$ with $f_{e_k}(P) = \|p^{(k)}\|$ for $k=1,\ldots,d$. For these, the Leibniz rule is
\begin{equation}\label{eq:MatrDerivProperty}
\Delta(P\boxtimes Q)=\Delta(P)\|q^{(k)}\|_1+\|p^{(k)}\|_1\Delta(Q).
\end{equation}
Turning our attention to the maps $\Delta^{(k)}_{\mathuline{\gamma}}$ of \eqref{eq:MatrDeriv}, we easily see that these are additive (under $\boxplus$) and satisfy \eqref{eq:MatrDerivProperty}. We next show that these maps with $\mathuline{\gamma}\in \R^d_+$ are precisely the monotone derivations at $f_{e_k}$.

\begin{proposition}\label{prop:MatrDerivation}
	Let $k\in\{1,\ldots,d\}$. On matrices $P \in \mc V^d_{<\infty}$ with $\|p^{(k)}\|_1 \in \Q$, the monotone derivations on $S^d$ at $f_{e_k}$ are precisely the $\Delta^{(k)}_{\mathuline{\gamma}}$ with $\mathuline{\gamma}\in \R_+^d$.
\end{proposition}

Note that $\Delta^{(k)}_{\mathuline{\gamma}}$ does not actually depend on $\gamma_k$, so the value of this component is completely arbitrary.

\begin{proof}
Let us fix $k=1$; the other cases are similar. Let $\Delta:S^d\to\R$ be a monotone derivation at $f_{e_1}$. Define $\delta:\R_{>0}^d\cup\{(0,\ldots,0)\}\to\R$ as the one-row function
\begin{equation}
\delta(x_1,\ldots,x_d)=\Delta\big(x_1\,\cdots\,x_d\big).
\end{equation}
We first show that $\delta(x,\ldots,x)=0$ for all $x\in\Q_+$. The Leibniz rule immediately gives that $\delta(0,\ldots,0)=0=\delta(1,\ldots,1)$. For every $n\in\N$, we easily see that
\begin{equation}
\left(\begin{array}{ccc}
1&\cdots&1\\
\vdots&\ddots&\vdots\\
1&\cdots&1
\end{array}\right)\succeq (n\,\cdots\,n)\succeq \left(\begin{array}{ccc}
1&\cdots&1\\
\vdots&\ddots&\vdots\\
1&\cdots&1
\end{array}\right)
\end{equation}
where the all-1 matrix is $(n\times d)$. Using the monotonicity and additivity of $\Delta$, we have
\begin{equation}
\delta(n,\ldots,n)=\Delta \left(\left(\begin{array}{ccc}
1&\cdots&1\\
\vdots&\ddots&\vdots\\
1&\cdots&1
\end{array}\right)\right) = n\delta(1,\ldots,1)=0.
\end{equation}
The Leibniz rule then easily gives $\delta(1/n,\ldots,1/n)=-n^{-2}\delta(n,\ldots,n)=0$ as well. Finally, for all $m,n\in\N$,
\begin{equation}
\delta\left(\frac{m}{n},\ldots,\frac{m}{n}\right)=m\underbrace{\delta\left(\frac{1}{n},\ldots,\frac{1}{n}\right)}_{=0}+\underbrace{\delta(m,\ldots,m)}_{=0}\frac{1}{n}=0,
\end{equation}
showing that $\delta(x,\ldots,x)=0$ for all $x\in\Q_+$.

Let now $p\in\R_+^n$ be a column vector of any length $n \in \N$ and such that $\|p\|_1\in\Q$. Using the easily seen fact that
\begin{equation}
(p,\ldots,p)\succeq(\|p\|_1\,\cdots\,\|p\|_1)\succeq(p,\ldots,p),
\end{equation}
the monotonicity of $\Delta$, and the vanishing property proven above for $\delta$, we see that
\begin{equation}
\Delta(p,\ldots,p)=\delta(\|p\|_1,\ldots,\|p\|_1)=0.
\end{equation}
Let then $P \in\mc V^d_n$ be such that $\|p^{(1)}\|_1\in\Q$, where we assume full support without loss of generality. Using the Leibniz rule several times and the above vanishing property, we obtain
\begin{align}
\Delta(P)&=\sum_{i=1}^n\delta(p^{(1)}_i,\ldots,p^{(d)}_i)=\sum_{i=1}^n\delta\left(p^{(1)}_i\cdot1,p^{(1)}_i\cdot\frac{p^{(2)}_i}{p^{(1)}_i},\ldots,p^{(1)}_i\cdot\frac{p^{(d)}_i}{p^{(1)}_i}\right)\\
&=\sum_{i=1}^n\left[\delta(p^{(1)}_i,\ldots,p^{(1)}_i)+p^{(1)}_i \delta\left(1,\frac{p^{(2)}_i}{p^{(1)}_i},\ldots,\frac{p^{(d)}_i}{p^{(1)}_i}\right)\right]\\
&=\underbrace{\Delta(p^{(1)},\ldots,p^{(1)})}_{=0}+\sum_{i=1}^n p^{(1)}_i\delta\left(1,\frac{p^{(2)}_i}{p^{(1)}_i},\ldots,\frac{p^{(d)}_i}{p^{(1)}_i}\right)\\
&=\sum_{i=1}^n p^{(1)}_i\delta\left(1,\frac{p^{(2)}_i}{p^{(1)}_i}\cdot1,1\cdot\frac{p^{(3)}_i}{p^{(1)}_i},\ldots,1\cdot\frac{p^{(d)}_i}{p^{(1)}_i}\right)\\
&=\sum_{i=1}^n p^{(1)}_i\left[\delta\left(1,\frac{p^{(2)}_i}{p^{(1)}_i},1,\ldots,1\right)+\delta\left(1,1,\frac{p^{(3)}_i}{p^{(1)}_i},\ldots,\frac{p^{(d)}_i}{p^{(1)}_i}\right)\right]\\
&=\cdots=\sum_{i=1}^n\sum_{\ell=2}^d p^{(1)}_i g_\ell\left(\frac{p^{(\ell)}_i}{p^{(1)}_i}\right),\label{eq:MatrDerivForm}
\end{align}
where, for all $\ell = 2,\ldots,d$, we have used
\begin{equation}
	g_\ell(x) := \delta(1,\ldots,1,x,1,\ldots,1),
\end{equation}
where $x$ appears in the $\ell$-th slot. Using the Leibniz rule, one easily sees that, for all $x,y>0$ and all $\ell$,
\begin{equation}
	g_\ell(xy)=g_\ell(x)+g_\ell(y),
\end{equation}
so that $g_\ell$ satisfies the functional equation associated with the logarithm. We show that $g_\ell$ is a multiple of the logarithm by showing that it is continuous on $\R_{>0}$. Let $x,y>0$ and $t\in(0,1)$. Using the fact that
\begin{align}
&\left(\begin{array}{ccccccc}
t&\cdots&t&tx&t&\cdots&t\\
1-t&\cdots&1-t&(1-t)y&1-t&\cdots&1-t
\end{array}\right)\\
&\succeq\big(1\,\cdots\,1\,tx+(1-t)y\,1\,\cdots\,1\big),
\end{align}
where $x$ and $y$ appear in the $\ell$-th column, applying $\Delta$ gives, by monotonicity and the derivation property of $\delta$,
\begin{equation}
	tg_\ell(x) + (1-t)g_\ell(y) \ge g_\ell\big(tx+(1-t)y\big),
\end{equation}
where the left-hand side is obtained by the formula~\eqref{eq:MatrDerivForm}.

Thus, $g_\ell$ is convex and hence continuous on its domain $\R_{>0}$. This means that there are $\gamma_2,\ldots,\gamma_d\in\R$ such that
\begin{equation}
	g_\ell =-\gamma_\ell \log{}
\end{equation}
for every $\ell = 2,\ldots,d$. We now aim at showing that $\gamma_\ell \ge 0$. Defining, for all $t\in(0,\frac{1}{2}]$, the matrix
\begin{equation}
P_t=\left(\begin{array}{ccccccc}
	\frac{1}{2}&\cdots&\frac{1}{2}&1-t&\frac{1}{2}&\cdots&\frac{1}{2}\\[2pt]
	\frac{1}{2}&\cdots&\frac{1}{2}&t&\frac{1}{2}&\cdots&\frac{1}{2}
\end{array}\right),
\end{equation}
where the entries containing $t$ are in the $\ell$-th column, we note that the function
\begin{equation}
\Delta(P_t) \stackrel{\eqref{eq:MatrDerivForm}}{=} \frac{1}{2}\gamma_\ell \log{\frac{1}{4 t(1-t)}}
\end{equation}
should be non-increasing in $t \in (0,\frac{1}{2}]$, which implies $\gamma_\ell \ge 0$. By~\eqref{eq:MatrDerivForm} again, this proves that every monotone derivation must be of the claimed form on all matrices whose first column has rational $1$-norm.

Let us finally show that the map $\Delta^{(1)}_{\mathuline{\gamma}}$ is a monotone derivation at $f_{e_1}$ if $\gamma_2, \ldots, \gamma_d \ge 0$. As already noted, the fact that it is a derivation follows by a straightforward calculation. For monotonicity, let $P \in\mc V^d_n$ for arbitrary $n$, without loss of generality of full support, and let $T=(T_{i,j})$ be an $(m\times n)$-stochastic matrix without zero rows. With $Q := TP \in\mc V^d_m$, we get by convexity of the function $x\mapsto\log{\frac{1}{x}}$,
\begin{align}
\Delta^{(1)}_{\mathuline{\gamma}}(P)&=\sum_{j=1}^n\sum_{\ell=2}^d \gamma_\ell p^{(1)}_j\log{\frac{p^{(1)}_j}{p^{(\ell)}_j}}=\sum_{i=1}^m\sum_{\ell=2}^d\sum_{j=1}^n \gamma_\ell q^{(1)}_i\frac{T_{i,j}p^{(1)}_j}{q^{(1)}_i}\log{\frac{p^{(1)}_j}{p^{(\ell)}_j}}\\
&\geq\sum_{i=1}^m\sum_{\ell=2}^d\gamma_\ell q^{(1)}_i\log{\frac{q^{(1)}_i}{\sum_{j=1}^n T_{i,j}p^{(\ell)}_j}}=\sum_{i=1}^m\sum_{\ell=1}^d\gamma_\ell q^{(1)}_i\log{\frac{q^{(1)}_i}{q^{(\ell)}_i}}\\
&=\Delta^{(1)}_{\mathuline{\gamma}}(Q),
\end{align}
showing that $\Delta^{(1)}_{\mathuline{\gamma}}:S^d\to\R_+$ is indeed monotone.
\end{proof}

\subsection{Sufficient conditions for large-sample and catalytic matrix majorization}

Our goal is now to apply Theorem~\ref{thm:Fritz8.6} to $S^d$ in order to obtain sufficient and generically necessary conditions for large-sample and catalytic matrix majorization. Here and throughout this subsection, let us write
\[
	\alpha_{\max} := \max_{k=1,\dots,d} \alpha_k, \qquad\quad
	\alpha_{\min} := \min_{k=1,\dots,d} \alpha_k.
\]
Then for every $\mathuline{\alpha}\in (A_+ \cup A_-) \setminus \{e_1, \ldots, e_d\}$ and $\mathuline{\beta} \in B_- \setminus \{0\}$, let us define the associated \emph{matrix $\mathuline{\alpha}$-divergence} on any $P \in \mc V^d_{<\infty}$ as
\begin{align}
\begin{split}\label{eq:MatrDiv}
D_{\mathuline{\alpha}}(P)&:=\frac{1}{\alpha_{\max} - 1}\log{f_{\mathuline{\alpha}}(P)}=\frac{1}{\alpha_{\max} - 1} \log \sum_i \prod_{k=1}^d \left( p_i^{(k)} \right)^{\alpha_k},\\
D^\T_{\mathuline{\beta}}(P)&:=\frac{1}{\beta_{\max}} \log{f^\T_{\mathuline{\beta}}(P)}=\frac{1}{\beta_{\max}} \log \max_i \prod_{k=1}^d \left( p_i^{(k)} \right)^{\beta_k}, \\ 
\end{split}
\end{align}
where $i$ ranges over the non-zero rows of $P$. We make these definitions mainly since in contrast to the $f_{\mathuline{\alpha}}$, the $D_{\mathuline{\alpha}}$ are monotone as maps $S^d\to\R$, where $\R$ is equipped with its standard order. As we will see shortly, they also allow us to write the tropical maps and the derivations as limits of the $D_{\mathuline{\alpha}}$.

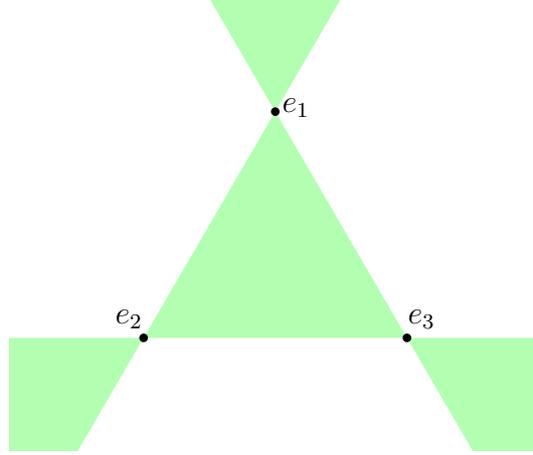
\begin{figure}
\captionsetup{width=.9\linewidth}
\begin{center}
	\begin{tikzpicture}[every node/.style={draw,circle,inner sep=1pt,fill=black}]
	\clip (-3.5,-2.5) rectangle (3.5,3.5);
	\foreach \s in {0,1,2}
		\fill[fill=green!30] (90 + 120 * \s : 2) -- + (60 + 120 * \s : 3) -- + (120 + 120 * \s : 3) -- cycle;
	\fill[fill=green!30] (0,2) node[label={[right]{$e_1$}}] {} -- (-1.732,-1) node[label={[above left]{$e_2$}}] {} -- (1.732,-1) node[label={[above right]{$e_3$}}] {} -- cycle;
\end{tikzpicture}
\end{center}
\caption{The parameter range $(A_- \cup A_+) \setminus \{e_1, \dots, e_d\}$ for the nondegenerate monotone homomorphisms $f_{\mathuline{\alpha}}$ on $S^3$, drawn in green as a subset of the $2$-dimensional affine space $A$. The set $A_+$ is the probability simplex on three outcomes (the triangular area in the centre). The other regions that constitute $A_-$ are cones emerging from the vertices, where the latter correspond to the degenerate homomorphisms $f_{e_k}$.}
\label{fig:d3}
\end{figure}

\begin{remark}
Different choices of normalization in~\eqref{eq:MatrDiv} are possible and may be worth considering. For example, a natural choice would be to utilize the logarithmic evaluation maps from~\cite{Fritz2021b}, which make sense for any preordered semiring with a fixed power universal element. In our present case, by Lemma~\ref{lem:MatrPowUniv} this amounts to choosing a matrix $U$ of full support and such that all columns are normalized and distinct. The logarithmic evaluation maps are then given by
\begin{equation}
P \longmapsto \frac{\log \Phi(P)}{\log \Phi(U)}
\end{equation}
for every nondegenerate monotone homomorphism $\Phi$.

For example, a nice choice could be to take $U$ to be the matrix with $\frac{2}{d+1}$ as every diagonal entry and $\frac{1}{d+1}$ off the diagonal. With this choice, one obtains the same $D_{\mathuline{\alpha}}$ as in~\eqref{eq:MatrDiv}, but with
\begin{equation}
\log\left(\frac{1}{d+1} \sum_k 2^{\alpha_k} \right)
\end{equation}
in place of $\alpha_{\max} - 1$. While this has the advantage of being a smooth function of $\mathuline{\alpha}$ and not requiring any case distinctions in any of the upcoming proofs, it is at the same time more cumbersome to calculate with, and so we proceed with~\eqref{eq:MatrDiv} as our preferred normalization.
\end{remark}

\begin{remark}\label{rem:MatrTropNorm}
For every $\lambda > 0$, the normalization~\eqref{eq:MatrDiv} satisfies
\begin{equation}
D^\T_{\lambda \mathuline{\beta}} = D^\T_{\mathuline{\beta}}.
\end{equation}
Thus it is enough to normalize $\mathuline{\beta}$ e.g.~such that $\beta_{\max} = 1$, and this implements the normalization of the tropical homomorphisms (Remark~\ref{tropical_normalization}).
\end{remark}

\begin{proposition}\label{prop:MatrixDivLim}
Let $\mathuline{\beta} \in B_-$ be such that $\beta_k = 1$, and hence $\beta_\ell \le 0$ for all $\ell \neq k$. For $\lambda \ge 0$, define
\begin{align}
\mathuline{\alpha}^\lambda&:=e_k + (\lambda - 1) \mathuline{\beta} \\
&=\big((\lambda-1)\beta_1,\ldots,(\lambda-1)\beta_{k-1},\lambda,(\lambda-1)\beta_{k+1},\ldots,(\lambda-1)\beta_d\big).
\end{align}
Then:
\begin{enumerate}[label=(\roman*)]
\item $\mathuline{\alpha}^\lambda \in (A_+ \cup A_-) \setminus \{e_1, \dots, e_d\}$ for all $\lambda \in (0,1) \cup (1,\infty)$.
\item On any $P$ of full support and with normalized columns, we have
\begin{align}
D^\T_{\mathuline{\beta}}(P)		& = \lim_{\lambda\to\infty} D_{\mathuline{\alpha}^\lambda}(P),\label{eq:limtropic}\\
\Delta^{(k)}_{-\mathuline{\beta}}(P)	& = \lim_{\lambda\to 1}D_{\mathuline{\alpha}^\lambda}(P).\label{eq:limderiv}
\end{align}
\item For such $P$, the function $\lambda \mapsto D_{\mathuline{\alpha}^\lambda}$ has a continuous extension to $[0,\infty]$, and this extension is non-decreasing for $\lambda \ge \frac{\beta_{\min}}{\beta_{\min}-1}$.
\end{enumerate}
\end{proposition}

In terms of Figure~\ref{fig:d3}, the trajectory $\lambda \mapsto \mathuline{\alpha}^\lambda$ starts at $\lambda = 0$ on the face of the probability simplex $A_+$ defined by $\alpha_k = 0$, goes through the vertex $e_k$ at $\lambda = 1$ and then continues in $A_-$ in the direction $\mathuline{\beta}$. As a function of $\lambda\geq0$, $D_{\mathuline{\alpha}^\lambda}(P)$ is non-decreasing after $\mathuline{\alpha}^\lambda$ passes the line drawn from the centre of $A_+$ to the middle of the side of $A_+$ closest to the trajectory.

\begin{proof}
	For item (i), it is straightforward to check that $\mathuline{\alpha}^\lambda \in A_+$ for $\lambda \in [0,1]$ and $\mathuline{\alpha}^\lambda \in A_-$ for $\lambda \in [1,\infty)$.

	Let us next prove the claims regarding the limits in item (ii). Both formulas follow by direct calculation as follows.
	For \eqref{eq:limtropic}, recall that for all real numbers $r_1, \ldots, r_m > 0$ and $s_1, \ldots, s_m > 0$, we have\footnote{This type of formula occurs frequently in the context of tropical mathematics and Maslov dequantization, and the reader will be able to easily work out the straightforward proof.}
\begin{equation}
	\label{eq:maslov_tropical}
	\lim_{\lambda \to \infty} \frac{1}{\lambda - 1} \log \sum_i r_i s_i^\lambda = \log \max_i s_i.
\end{equation}
This directly implies
\begin{align}
\lim_{\lambda \to \infty} D_{\mathuline{\alpha}^\lambda}(P)&=\lim_{\lambda\to\infty} \frac{1}{\lambda-1} \log{\sum_i\left[\prod_{\ell\neq k} \left(p^{(\ell)}_i\right)^{-\beta_\ell}\right]\left[p^{(k)}_i\prod_{\ell\neq k}\left(p^{(\ell)}_i\right)^{\beta_\ell}\right]^\lambda}\\
&=\log{\max_i p^{(k)}_i \prod_{\ell\neq k}\left(p^{(\ell)}_i\right)^{\beta_\ell}}= D^\T_{\mathuline{\beta}}(P),
\end{align}
as was to be shown.
For \eqref{eq:limderiv}, using numbers as in~\eqref{eq:maslov_tropical} and assuming $\sum_i r_i = 1$ in addition, we similarly note
\begin{equation}
	\lim_{\lambda \to 1} \frac{1}{\lambda - 1} \log \sum_i r_i s_i^{\lambda - 1} = \sum_i r_i \log s_i,
\end{equation}
as follows by a straightforward application of l'H\^{o}pital's rule. Plugging this in gives
\begin{align}
	\lim_{\lambda\to 1}D_{\mathuline{\alpha}^\lambda}(P)&=\lim_{\lambda\to 1}\frac{1}{\lambda-1}\log{\sum_{i} p^{(k)}_i \prod_{\ell=1}^d \left(p^{(\ell)}_i\right)^{\beta_\ell(\lambda-1)}}\\
	&= \sum_i p^{(k)}_i\log{\prod_{\ell=1}^d \left(p^{(\ell)}_i\right)^{\beta_\ell}}\\
	&= -\sum_i\sum_{\ell\neq k}\beta_\ell p^{(k)}_i\log{\frac{p^{(k)}_i}{p^{(\ell)}_i}} = \Delta^{(k)}_{-\mathuline{\beta}}(P),
\end{align}
	as was to be shown.

	For item (iii), the existence of the limits just proven shows that we obtain a continuous extension taking the values $\Delta^{(k)}_{-\mathuline{\beta}}(P)$ at $\lambda = 1$ and $D^\T_{\mathuline{\beta}}$ at $\lambda = \infty$.
	For the monotonicity, let $\frac{\beta_{\rm min}}{\beta_{\rm min}-1}\leq\lambda\leq\lambda' \le \infty$.
	It is enough to assume that $\lambda, \lambda' \neq 1, \infty$ and prove the monotonicity
	\begin{equation}
		D_{\mathuline{\alpha}^\lambda}(P) \le D_{\mathuline{\alpha}^{\lambda'}}(P)
	\end{equation}
	then, since the remaining cases then follow by continuity.
	With $z := \frac{\lambda-1}{\lambda'-1}$, we distinguish three cases: 
	\begin{itemize}
		\item $\lambda\leq\lambda'<1$: Now $z \ge 1$, so that $x\mapsto x^z$ is convex.
		\item $\lambda<1<\lambda'$: Now $z<0$, so that $x\mapsto x^z$ is convex again.
		\item $1<\lambda\leq\lambda'$: Now $0<z\leq 1$, so that $x\mapsto x^z$ is concave.
	\end{itemize}
	Considering all the above cases and taking the signs of the denominators into account, we have by Jensen's inequality,
	\begin{equation}
		\label{monotone_jensen}
		\frac{1}{\lambda-1}\log{\sum_i r_i s_i^z}\leq \frac{z}{\lambda-1}\log{\sum_i r_i s_i} = \frac{1}{\lambda'-1}\log{\sum_i r_i s_i}
	\end{equation}
	for all probability vectors $r$ and all positive reals $s_i$.
	Thus, we have
	\begin{align}
		D_{\mathuline{\alpha}^\lambda}(P)&=\frac{1}{\lambda-1}\log{\sum_i \left(p^{(k)}_i\right)^{\lambda}\prod_{\ell\neq k}\left(p^{(\ell)}_i\right)^{(\lambda-1)\beta_\ell}}\\
		&=\frac{1}{\lambda-1}\log{\sum_i p^{(k)}_i \prod_{\ell\neq k}\left(\frac{p^{(\ell)}_i}{p^{(k)}_i}\right)^{(\lambda-1)\beta_\ell}}\\
		&=\frac{1}{\lambda-1}\log{\sum_i p^{(k)}_i \left[\prod_{\ell\neq k}\left(\frac{p^{(\ell)}_i}{p^{(k)}_i}\right)^{(\lambda'-1)\beta_\ell}\right]^z}\\
		&\leq\frac{1}{\lambda'-1}\log{\sum_i p^{(k)}_i \prod_{\ell\neq k}\left(\frac{p^{(\ell)}_i}{p^{(k)}_i}\right)^{(\lambda'-1)\beta_\ell}}=D_{\mathuline{\alpha}^{\lambda'}}(P),
	\end{align}
	as was to be shown.
\end{proof}

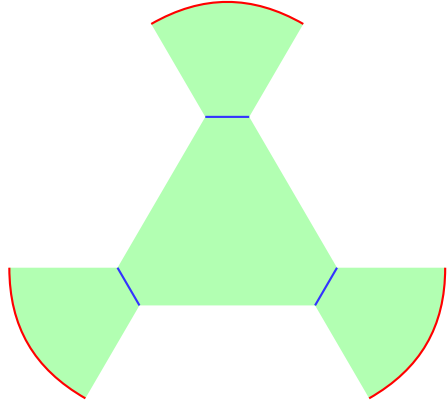
\begin{figure}
\captionsetup{width=.9\linewidth}
\begin{center}
\begin{tikzpicture}
	\fill[fill=green!30,name path=innertriangle] (0,2) -- (-1.732,-1) -- (1.732,-1) -- cycle;
	\foreach \s in {0,1,2} {
		\fill[fill=green!30,name path=outertriangle\s] (90 + 120 * \s : 1) -- + (60 + 120 * \s : 2) node (A\s) {} to [bend right] + (120 + 120 * \s : 2) node (B\s) {} -- cycle;
		\path [name intersections={of=innertriangle and outertriangle\s}];
		\draw[thick,blue!80] (intersection-1) -- (intersection-2);
		\draw[thick,red=!80] (A\s.center) to [bend right] (B\s.center);
	}
\end{tikzpicture}
\end{center}
\caption{The relevant monotones on the preordered semiring $S^3$ in the sense of the \emph{test spectrum} of~\cite{Fritz2021b} represented schematically. In green, we have the matrix $\mathuline{\alpha}$-divergences $D_{\mathuline{\alpha}}$ for $\mathuline{\alpha} \in (A_+ \cup A_-) \setminus \{e_1, \dots, e_d\}$ from~\eqref{eq:MatrDiv}. In red, we have the tropical quantities $D^\T_{\mathuline{\beta}}$ for $\mathuline{\beta} \in B_- \setminus \{0\}$, modulo scaling per Remark~\ref{rem:MatrTropNorm}. In blue, we have the derivation quantities $\Delta^{(k)}_{\mathuline{\gamma}}$ for nonzero $\mathuline{\gamma}$, also modulo scaling. The topology corresponds to the limits of Proposition~\ref{prop:MatrixDivLim}. In comparison with Figure~\ref{fig:d3}, the red tropical part can be thought of as points at infinity, specified by a direction $\mathuline{\beta} \in B_- \setminus \{0\}$, which serve as a compactification similar to points at infinity in the projective plane. The blue derivation part ``blows up'' every vertex in Figure~\ref{fig:d3} to a line.}
\label{fig:d3div}
\end{figure}

The set of all relevant monotones can thus be illustrated as in Figure~\ref{fig:d3div}.

\begin{proposition}\label{prop:ConvexFaithful}
	Let $P \in\mc V^d_n$ be stochastic. Then:
	\begin{enumerate}[label=(\roman*)]
		\item For $\mathuline{\alpha}\in(A_+\cup A_-)\setminus\{e_1,\ldots,e_d\}$, we have $D_{\mathuline{\alpha}}(P)\geq 0$.
			This holds with equality if and only if all columns $p^{(k)}$ with $\alpha_k\neq0$ are equal.
		\item For $\mathuline{\beta}\in B_- \setminus \{0\}$, we have $D^\T_{\mathuline{\beta}}(P)\geq0$.
			This holds with equality if and only if all columns $p^{(k)}$ with $\beta_k\neq0$ are equal.
	\end{enumerate}
\end{proposition}

\begin{proof}
We formulate the proof for case (i), noting that case (ii) is perfectly analogous. We prove equivalently that $f_{\mathuline{\alpha}}(P) \ge 1$ with equality if and only if all columns $p^{(k)}$ with $\alpha_k \neq 0$ are equal. To this end, note that the inequality is obvious by monotonicity and because of
\begin{equation}
	\label{eq:Psucceq1}
	P \succeq \left( 1 \dots 1 \right),
\end{equation}
which holds by the assumption that $P$ is stochastic. For the condition of when equality holds, suppose first that no two columns of $P$ are equal. Then $P$ is power universal in $S^d$ by Lemma~\ref{lem:MatrPowUniv}. So if $f_{\mathuline{\alpha}}(P) = 1$ was the case, then $f_{\mathuline{\alpha}}$ would be degenerate, which by Proposition~\ref{prop:MatrMonHom} we know not to be the case. This finishes the proof in case that $P$ has distinct columns.

If $P$ has some coinciding columns, then let us assume without loss of generality that the final two columns of $P$ coincide. We reduce the problem to the case of smaller $d$, which lets us conclude the overall proof by induction on $d$. Duplicating the last column of the matrices under consideration defines a monotone semiring homomorphism $S^{d-1} \to S^d$, and in this way we may consider $S^{d-1} \subseteq S^d$ as a subsemiring with the induced preorder. Restricting the matrix divergence $D_{\mathuline{\alpha}}$ to this subsemiring defines a matrix divergence of the same type, but with the new parameter tuple
\[
	\mathuline{\alpha}' := (\alpha_1, \dots, \alpha_{d-2}, \alpha_{d-1} + \alpha_d)
\]
in place of $\mathuline{\alpha}$. If $\mathuline{\alpha}$ is supported on the final two columns, then the restricted homomorphism is degenerate, and the to be proven $f_\alpha(P) = 1$ follows by~\eqref{eq:Psucceq1}. Otherwise, the restricted homomorphism is still nondegenerate since $\mathuline{\alpha}'$ belongs to $A_+ \cup A_-$ but is still not a standard basis vector. Since $P$ belongs to the subsemiring $S^{d-1}$, we have completed our reduction to $d-1$.
\end{proof}

We now put things together and apply Theorem~\ref{thm:Fritz8.6} to $S^d$. Modulo some minor detail discussed in Remark~\ref{rem:TropicalDiff}, this provides a positive answer to a conjecture stated in Section 6 and Appendix K of \cite{Mu_et_al_2020}.

\begin{theorem}\label{theor:MatrMajSuff}
Let $P = (p^{(1)},\ldots,p^{(d)})$ and $Q = (q^{(1)},\ldots,q^{(d)})$ be $d$-tuples of probability vectors, each with uniform support across the tuple. In terms of the matrix divergences from~\eqref{eq:MatrDiv} and the Kullback-Leibler divergence $D_1(\cdot\|\cdot)$, suppose that
\begin{align}
D_{\mathuline{\alpha}}(P)&>D_{\mathuline{\alpha}}(Q),\label{eq:MatrMajNonTropic}\\
D^\T_{\mathuline{\beta}}(P)&>D^\T_{\mathuline{\beta}}(Q), \label{eq:MatrMajTropic}\\
D_1\left(p^{(k)}\middle\|p^{(\ell)}\right)&>D_1\left(q^{(k)}\middle\|q^{(\ell)}\right)\label{eq:MatrMaj_KL}
\end{align}
for all $\mathuline{\alpha} \in (A_+ \cup A_-) \setminus \{e_1, \dots e_d\}$, $\mathuline{\beta} \in B_- \setminus \{0\}$ and $k\neq\ell$, where the parameter sets are as in~\eqref{A-}--\eqref{B-}. Then, for all $n\in\N$ sufficiently large, there is a stochastic matrix $T_n$ such that
\begin{equation}
(q^{(k)})^{\otimes n}=T_n(p^{(k)})^{\otimes n} \qquad \forall k=1,\ldots,d.
\end{equation}
Conversely, if there is such $T_n$ for some $n \ge 1$, then the above inequalities hold non-strictly.
\end{theorem}

By Remark \ref{rem:MatrTropNorm}, in~\eqref{eq:MatrMajTropic} one can moreover impose a normalization condition such as $\beta_{\max} = 1$. Note also that the conditions~\eqref{eq:MatrMajNonTropic}--\eqref{eq:MatrMaj_KL} really form part of a continuous family of inequalities, as illustrated in Figure~\ref{fig:d3div}.

\begin{proof}
We already noted that the auxiliary assumptions of Theorem~\ref{thm:Fritz8.6} are satisfied. Thus it remains to be shown that the inequalities in the assumption match those that are obtained from there, as well as that $P = (p^{(1)}, \dots, p^{(d)})$ is power universal. The latter holds because the assumed~\eqref{eq:MatrMaj_KL} implies that $p^{(k)} \neq p^{(\ell)}$ for $k \neq \ell$, and hence $P$ is power universal by Lemma~\ref{lem:MatrPowUniv}.

By Proposition~\ref{prop:MatrMonHom}, the nondegenerate monotone homomorphisms $S^d \to \R_+$ are exactly the $f_{\mathuline{\alpha}}$ for $\alpha \in A_- \setminus \{e_1, \ldots, e_d\}$. But since
\begin{equation}
f_{\mathuline{\alpha}}(P) > f_{\mathuline{\alpha}}(Q) \qquad \Longleftrightarrow \qquad D_{\mathuline{\alpha}}(P) > D_{\mathuline{\alpha}}(Q)
\end{equation}
for such $\mathuline{\alpha}$ by $\alpha_{\max} > 1$, the inequalities indeed match for ${\mb K} = \R_+$. Similarly for ${\mb K} = \R_+^{\rm op}$, the relevant quantities are the $f_{\mathuline{\alpha}}$ with $\mathuline{\alpha} \in A_+$. Since then $\alpha_{\max} < 1$, we have
\begin{equation}
f_{\mathuline{\alpha}}(P) < f_{\mathuline{\alpha}}(Q) \qquad \Longleftrightarrow \qquad D_{\mathuline{\alpha}}(P) > D_{\mathuline{\alpha}}(Q),
\end{equation}
and again the inequalities match. For ${\mb K} = \T\R_+$, we clearly have
\begin{equation}
f^\T_{\mathuline{\beta}}(P) > f^\T_{\mathuline{\beta}}(Q) \qquad \Longleftrightarrow \qquad D^\T_{\mathuline{\beta}}(P) > D^\T_{\mathuline{\beta}}(Q),
\end{equation}
which also finishes this case. The case ${\mb K} = \T\R_+^{\rm op}$ is empty by Proposition~\ref{prop:MatrMonHom}.

It remains to consider the monotone derivations at $f_{e_k}$ for every $k = 1, \dots, d$. By Proposition~\ref{prop:MatrDerivation}, these derivations are all of the form $\Delta^{(k)}_{\mathuline{\gamma}}$ for $\gamma \in \R_+^d$. Since $\Delta^{(k)}_{\mathuline{\gamma}}$ depends additively on $\mathuline{\gamma}$, it is enough to postulate the inequalities for the extremal $\mathuline{\gamma}$, which are exactly the standard basis vectors. Moreover, the derivation $\Delta^{(k)}_{\mathuline{\gamma}}$ is non-zero if and only if $\gamma_\ell \neq 0$ for some $\ell \neq k$. Taking these statements together, we conclude that it is enough to consider the derivations $\Delta^{(k)}_{e_\ell}$ with $\ell \neq k$, and this results precisely in~\eqref{eq:MatrMaj_KL}.
\end{proof}

\begin{remark}\label{rem:TropicalDiff}
Translating the assumptions made by Mu {\it et al.}\ in the online appendix of \cite{Mu_et_al_2020} for their conjecture into our framework reveals the following difference in the sufficient condition: On the tropical matrix divergences $D^\T_{\mathuline{\beta}}$, they assume strict inequality only for $\mathuline{\beta} = e_k - e_\ell$ for $k,\ell = 1, \ldots, d$, $k\neq\ell$ (this is their genericity condition), while they assume non-strict inequality otherwise (implicitly by our Proposition~\ref{prop:MatrixDivLim}). In other words, our sufficient condition \eqref{eq:MatrMajTropic} for matrix majorization in large samples, which requires strict inequality on $D^\T_{\mathuline{\beta}}$ for all $\mathuline{\beta} \in B_- \setminus \{0\}$, is replaced by the condition
\begin{equation}
D_\infty\left(p^{(k)}\|p^{(\ell)}\right)>D_\infty\left(q^{(k)}\|q^{(\ell)}\right)
\end{equation}
for all $k,\,\ell\in\{1,\ldots,d\}$, $k\neq\ell$, in \cite{Mu_et_al_2020}. We do not know whether there are any $P,Q \in \mc V^d_{<\infty}$ that satisfy their assumptions but not ours.
\end{remark}

\begin{remark}
\label{rem:MatrMajCat}
Let us also note that, under the assumptions of Theorem~\ref{theor:MatrMajSuff} given by \eqref{eq:MatrMajNonTropic}, \eqref{eq:MatrMajTropic}, and \eqref{eq:MatrMaj_KL}, by item (b) of Theorem~\ref{thm:Fritz8.6} we also have the following catalytic result: defining, for every $n\in\N$ large enough, the probability vectors
\begin{equation}\label{eq:MatrCatalystForm}
r^{(k)}:=\frac{1}{n}\bigoplus_{s=0}^{n-1} \left(p^{(k)}\right)^{\otimes(n-1-s)}\otimes\left(q^{(k)}\right)^{\otimes s},
\end{equation}
for $k=1,\ldots,d$, there is a stochastic map $T_n$ such that $T_n(p^{(k)}\otimes r^{(k)})=q^{(k)}\otimes r^{(k)}$ for all $k=1,\ldots,d$.
\end{remark}

Now that we have provided sufficient and generically necessary conditions for catalytic matrix majorization, we note that as a corollary using the same preorder we can also obtain sufficient conditions for \textit{asymptotic catalytic matrix majorization}.

\begin{theorem}
\label{theor:MatrApproxCatMaj}
Let $n\in\N$ and let $\mc P_n$ be the set of probability vectors with $n$ entries. Let $P = (p^{(1)}, \dots, p^{(d)}) \in \mc P_n^d$ and $Q = (q^{(1)}, \dots, q^{(d)}) \in \mc P_n^d$ be $d$-tuples of probability vectors with full support such that $p^{(k)}\neq p^{(\ell)}$ for all $k \neq \ell$. Then the following conditions are equivalent: 
\begin{enumerate}[label=(\roman*)]
\item For every $\varepsilon > 0$ there exist $(q_\varepsilon^{(1)}, \ldots, q_\varepsilon^{(d)})\in\mc P_n^{d}$ and $(r_\varepsilon^{(1)}, \ldots, r_\varepsilon^{(d)}) \in \mc P_n^{d}$ and a stochastic matrix $T_\varepsilon$ such that
\begin{equation}
\| q^{(k)} - q_\varepsilon^{(k)} \|_1 \leq \varepsilon \qquad \forall k = 1, \dots, d,
\end{equation}
and
\begin{equation}\label{eq:MatrMajCatApprox}
T_\varepsilon(p^{(k)} \otimes r_\varepsilon^{(k)}) = q_\varepsilon^{(k)} \otimes r_\varepsilon^{(k)},\qquad \forall k=1,\ldots,d.
\end{equation}
\item For all $\mathuline{\alpha}\in (A_- \cup A_+) \setminus \{e_1, \dots, e_d\}$, we have that
\begin{equation}
D_{\mathuline{\alpha}}(P) \geq D_{\mathuline{\alpha}}(Q).
\end{equation}
\end{enumerate}
Moreover, for one $k\in\{1,\ldots,d\}$, we can take $q_\varepsilon^{(k)} = q^{(k)}$ in (i) for all $\varepsilon>0$.
\end{theorem}

\begin{proof}
Throughout this proof, we write $P = (p^{(1)}, \ldots, p^{(d)})$ and similarly for the other tuples. Let us first assume condition (ii) of the claim. For an arbitrary probability vector $w \in {\mc P_n}$, we define the $n \times d$ noise matrix 
\begin{equation}
	W = (w, \ldots, w).
\end{equation}
Define $Q_\varepsilon:=(1-\frac{\varepsilon}{2})Q+\frac{\varepsilon}{2} W$ for all $\varepsilon \in (0,2]$. Denoting the columns of $Q_\varepsilon$ by $q_\varepsilon^{(k)}$, we then have
\begin{equation}
\|q^{(k)}-q_\varepsilon^{(k)}\|_1=\left\|q^{(k)}-\left(1-\frac{\varepsilon}{2}\right)q^{(k)}-\frac{\varepsilon}{2} w\right\|_1\leq\frac{\varepsilon}{2}+\frac{\varepsilon}{2}=\varepsilon
\end{equation}
for every $k$. Taking $w := q^{(k)}$ with a fixed $k\in\{1,\ldots,d\}$ throughout the following proof will take care of the final claim.

We aim at proving that $Q_\varepsilon$ and $P$ satisfy all the relevant inequalities~\eqref{eq:MatrMajNonTropic}--\eqref{eq:MatrMaj_KL} strictly. Given any $\mathuline{\alpha}\in A_+\setminus\{e_1,\ldots,e_d\}$, we evaluate
\begin{align}
	D_{\mathuline{\alpha}}\left(Q_\varepsilon\right) &= \frac{1}{\alpha_{\max}-1}\log{\underbrace{f_{\mathuline{\alpha}}\left(\left(1-\frac{\varepsilon}{2}\right)Q+\frac{\varepsilon}{2}W\right)}_{\geq \left(1-\frac{\varepsilon}{2}\right)f_{\mathuline{\alpha}}(Q)+\frac{\varepsilon}{2}f_{\mathuline{\alpha}}(W)}}\\
	&\leq \frac{1}{\alpha_{\max}-1}\log{\left[\left(1-\frac{\varepsilon}{2}\right)f_{\mathuline{\alpha}}(Q)+\frac{\varepsilon}{2}f_{\mathuline{\alpha}}(W)\right]}\\[2pt]
	&\leq \frac{1}{\alpha_{\max}-1}\left[\left(1-\frac{\varepsilon}{2}\right)\log{f_{\mathuline{\alpha}}(Q)}+\frac{\varepsilon}{2}\log{f_{\mathuline{\alpha}}(W)}\right]\\[2pt]
	&=\left(1-\frac{\varepsilon}{2}\right)D_{\mathuline{\alpha}}(Q)+\frac{\varepsilon}{2}\underbrace{D_{\mathuline{\alpha}}(W)}_{=0} \leq D_{\mathuline{\alpha}}(Q)\leq D_{\mathuline{\alpha}}(P),
\end{align}
where we have used the concavity of $f_{\mathuline{\alpha}}$ (see the proof of Proposition~\ref{prop:MatrMonHom}) and the fact that $(\alpha_{\max}-1)^{-1}\log{}$ is non-increasing in the first inequality, and the convexity of $(\alpha_{\max}-1)^{-1}\log{}$ in the second inequality, and the assumption in the final inequality. If $D_{\mathuline{\alpha}}(Q)=0$, then the final inequality is strict by $D_{\mathuline{\alpha}}(P)>0$ due to the assumptions on $P$ and Proposition~\ref{prop:ConvexFaithful}. If $D_{\mathuline{\alpha}}(Q)>0$, then the penultimate inequality is strict. Thus, $D_{\mathuline{\alpha}}(Q_\varepsilon)<D_{\mathuline{\alpha}}(P)$ in all cases. Similarly, for $\mathuline{\alpha}\in A_-\setminus\{e_1,\ldots,e_d\}$, we can use essentially the same arguments to show that $D_{\mathuline{\alpha}}(Q_\varepsilon) < D_{\mathuline{\alpha}}(P)$ as well.

For the tropical case, let $\mathuline{\beta}\in B_- \setminus \{0\}$, and we will show that $D^\T_{\mathuline{\beta}}(Q_\varepsilon)<D^\T_{\mathuline{\beta}}(P)$. By Remark~\ref{rem:MatrTropNorm}, we can assume $\beta_{\max} = 1$. Then for example by the previous case and Proposition~\ref{prop:MatrixDivLim}, we have
\begin{equation}
	D^\T_{\mathuline{\beta}}(Q_\eps) \le D^\T_{\mathuline{\beta}}(Q),
\end{equation}
or equivalently
\begin{equation}
	\max_i \prod_{k=1}^d \left( \left(1 - \frac{\varepsilon}{2}\right) q^{(k)}_i + \frac{\varepsilon}{2} w_i \right)^{\beta_k} \le \max_i \prod_{k=1}^d \left( q^{(k)}_i \right)^{\beta_k}
\end{equation}
for every $\eps \in (0,2]$. The left-hand side is non-increasing in $\eps$ by $Q_{\varepsilon} \succeq Q_{\varepsilon'}$ for $\varepsilon \le \varepsilon'$ and coincides with the right-hand side as $\eps \to 0$. Now if this inequality is non-strict for some $\eps \in (0,2]$, then the left-hand side must be constant in $\eps$ all the way down to $\eps \to 0$. Since the left-hand side is a pointwise maximum of finitely many real-analytic functions, it follows that the maximum must be achieved on one and the same row index $i$ for all these $\eps$, and that the product is identically constant for this $i$. Again by real-analyticity, this implies that the left-hand side is constant across the whole range $[0,2]$. Therefore if the inequality is non-strict, then we must already have $D^\T_{\mathuline{\beta}}(Q) = D^\T_{\mathuline{\beta}}(W) = 0$.

Hence if $D^\T_{\mathuline{\beta}}(Q) > 0$, then we obtain $D^\T_{\mathuline{\beta}}(Q_\eps) < D^\T_{\mathuline{\beta}}(Q)$. And if $D^\T_{\mathuline{\beta}}(Q) = 0$, then we have $D^\T_{\mathuline{\beta}}(Q) < D^\T_{\mathuline{\beta}}(P)$ by Proposition~\ref{prop:ConvexFaithful}. In both cases, we obtain
\begin{equation}
	D^\T_{\mathuline{\beta}}(Q_\eps) \le D^\T_{\mathuline{\beta}}(Q) \le D^\T_{\mathuline{\beta}}(P),
\end{equation}
where at least one of the inequalities is strict, and therefore $D^\T_{\mathuline{\beta}}(Q_\eps) < D^\T_{\mathuline{\beta}}(P)$.

For the Kullback-Leibler divergences~\eqref{eq:MatrMaj_KL}, fix distinct column indices $k,\ell \in\{1,\ldots,d\}$. By convexity of $D_1$, we obtain\footnote{This convexity can also be concluded from the convexity of the $D_{\mathuline{\alpha}}$, which is an instance of~\eqref{eq:PhiConvex}, and writing $\Delta_{e_k - e_\ell}$ as a pointwise limit of the $D_{\mathuline{\alpha}}$ as per Proposition~\ref{prop:MatrixDivLim}.}
\begin{align}
	D_1 \left( q_\varepsilon^{(k)} \| q_\varepsilon^{(\ell)} \right) & \le \left(1-\frac{\varepsilon}{2}\right) D_1(q^{(k)} \| q^{(\ell)}) + \frac{\varepsilon}{2} \underbrace{D_1(w \| w)}_{=0} \nonumber \\
	      & \le D_1(q^{(k)} \| q^{(\ell)}) \le D_1(p^{(k)} \| p^{(\ell)}).
\end{align}
As before, if $D_1(q^{(k)} \| q^{(\ell)}) > 0$, then the penultimate inequality above is strict and, while if $D_1(q^{(k)} \| q^{(\ell)})=0$, then the final inequality is strict. All in all, we get $D_1(q_\varepsilon^{(k)} \| q_\varepsilon^{(\ell)}) < D_1(p^{(k)} \| p^{(\ell)})$ as well.

Thus, the conditions of Theorem~\ref{theor:MatrMajSuff} hold for $Q_\varepsilon$ (with any $\varepsilon\in(0,1]$) and $P$, and we conclude that there is a stochastic map $T_\varepsilon$ and a $d$-tuple of probability vectors $r_\varepsilon^{(k)}\in\mc P_n$ such that $T_\varepsilon(p^{(k)}\otimes r_\varepsilon^{(k)})=q_\varepsilon^{(k)}\otimes r_\varepsilon^{(k)}$ for all $k=1,\ldots,d$ (Remark~\ref{rem:MatrMajCat}).

For the converse direction, suppose that we have $\tilde{Q}$ as in the statement. Then applying $D_{\mathuline{\alpha}}$ to \eqref{eq:MatrMajCatApprox} and using monotonicity together with additivity produces $D_{\mathuline{\alpha}}(P) \ge D_{\mathuline{\alpha}}(\tilde{Q})$ for all $\mathuline{\alpha} \in A_- \cup A_+$. But then also $D_{\mathuline{\alpha}}(P) \ge D_{\mathuline{\alpha}}(Q)$ by the continuity of $D_{\mathuline{\alpha}}$ on $\R_{>0}^{n \times d}$.
\end{proof}

\subsection{Case \texorpdfstring{$d=2$}{d=2}: relative majorization}\label{subsec:d2}

Let us have a quick look at the case where $d=2$, i.e.,\ where the matrices reduce to {\it dichotomies} $(p^{(1)}, p^{(2)})\in\mc V^2_{<\infty}$. In this setting, matrix majorization is also called {\it relative majorization}.

In a dichotomy $(p,q)\in\mc V^2_{<\infty}$, the vectors $p$ and $q$ are required to have coinciding supports. However, let us remove this assumption for a while and consider just a pair $(p,q)$ of vectors with non-negative entries. Whenever two such vectors $p = (p_1,\ldots,p_n)$ and $q = (q_1,\ldots,q_n)$ are under joint consideration, we will assume that padding by zeros has been applied so that they have the same length $n$. With this in mind, and defining $0\log{0}:=0$, let us define the {\it R\'{e}nyi divergences} for all finite probability vectors $p$ and $q$ by
\begin{align}
	\label{eq:Renyirelentr}
	\!\!\! D_\alpha(p\|q)= {}&\left\{\begin{array}{@{}cl}
		-\log{\sum_{i\in{\rm supp}\,p}q_i} &{\rm if}\ \alpha = 0\ {\rm and}\ {\rm supp}\,p\cap{\rm supp}\,q\neq\emptyset, \\[2mm]
		\frac{1}{\alpha-1}\log{\sum_{i=1}^n p_i^\alpha q_i^{1-\alpha}} &{\rm if}\ \alpha \in (0,1)\ {\rm and}\ {\rm supp}\,p\cap{\rm supp}\,q\neq\emptyset, \\[2mm]
		\sum_{i=1}^n p_i\log{\frac{p_i}{q_i}} &{\rm if}\ \alpha = 1\ {\rm and}\ {\rm supp}\,p\subseteq{\rm supp}\,q, \\[2mm]
		\frac{1}{\alpha-1}\log{\sum_{i=1}^n p_i^\alpha q_i^{1-\alpha}} &{\rm if}\ \alpha \in (1,\infty)\ {\rm and}\ {\rm supp}\,p\subseteq{\rm supp}\,q,\\[2mm]
		\max_{i \in {\rm supp}\, q} \log \frac{p_i}{q_i} &{\rm if}\ \alpha = \infty\ {\rm and}\ {\rm supp}\,p\subseteq{\rm supp}\,q, \\[2mm]
		\infty&{\rm otherwise},
	\end{array}\right.
\end{align}
for all R\'enyi parameter values $\alpha \geq 0$. Note that $D_1$ is the Kullback-Leibler divergence. From Theorem~\ref{theor:MatrMajSuff}, we obtain a variation of the main result of \cite{Mu_et_al_2020}, namely the following:

\begin{corollary}\label{cor:RelativeMaj}
Let $p^{(1)}$ and $p^{(2)}$ be probability vectors with coinciding supports, and likewise for $q^{(1)}$ and $q^{(2)}$. If
\begin{align}
D_\alpha\left(p^{(1)}\middle\|p^{(2)}\right)&>D_\alpha\left(q^{(1)}\middle\|q^{(2)}\right),\label{eq:Renyiright}\\
D_\alpha\left(p^{(2)}\middle\|p^{(1)}\right)&>D_\alpha\left(q^{(2)}\middle\|q^{(1)}\right)\label{eq:Renyileft}
\end{align}
for all $\alpha\in[1/2,\infty]$, then, for every $n\in\N$ large enough, we have
\begin{equation}
\left(\big(p^{(1)}\big)^{\otimes n},\big(p^{(2)}\big)^{\otimes n}\right)\succeq\left(\big(q^{(1)}\big)^{\otimes n},\big(q^{(2)}\big)^{\otimes n}\right).
\end{equation}
\end{corollary}

\begin{proof}
Let us assume that \eqref{eq:Renyiright} and \eqref{eq:Renyileft} hold for all $\alpha\in[1/2,\infty]$ and write $p^{(k)}=(p^{(k)}_1,\ldots,p^{(k)}_m)$ and $q^{(k)}=(q^{(k)}_1,\ldots,q^{(k)}_n)$ for $k=1,2$, where all the entries are positive. First, let $\alpha\in\R\setminus\{0,1\}$. One may easily check that
\begin{equation}
D_{(\alpha,1-\alpha)}(p^{(1)},p^{(2)})=\left\{\begin{array}{ll}
		D_\alpha(p^{(1)}\|p^{(2)}),&\alpha\in[1/2,1)\cup(1,\infty),\\[2pt]
		D_{1-\alpha}(p^{(2)}\|p^{(1)}),&\alpha\in(-\infty,0)\cup(0,1/2).
\end{array}\right.
\end{equation}
Thus, \eqref{eq:Renyiright} and \eqref{eq:Renyileft} with $\alpha\in(0,1)\cup(1,\infty)$ correspond to the conditions \eqref{eq:MatrMajNonTropic} in the case $d=2$. Next, let $\beta>0$. Recalling Remark~\ref{rem:MatrTropNorm} one easily sees that
\begin{align}
	D^\T_{(\beta,-\beta)}(p^{(1)},p^{(2)})&=D^\T_{(1,-1)}(p^{(1)},p^{(2)})=\max_{1\leq i\leq m}\log{\frac{p^{(1)}_i}{p^{(2)}_i}}=D_\infty(p^{(1)}\|p^{(2)}),\\
	D^\T_{(-\beta,\beta)}(p^{(1)},p^{(2)})&=D^\T_{(-1,1)}(p^{(1)},p^{(2)})=\max_{1\leq i\leq m}\log{\frac{p^{(2)}_i}{p^{(1)}_i}}=D_\infty(p^{(2)}\|p^{(1)}).
\end{align}
Thus, \eqref{eq:Renyiright} and \eqref{eq:Renyileft} with $\alpha=\infty$ correspond to the conditions \eqref{eq:MatrMajTropic} in the case $d=2$. Finally, \eqref{eq:Renyiright} and \eqref{eq:Renyileft} with $\alpha=1$ clearly correspond to conditions \eqref{eq:MatrMaj_KL} in the case $d=2$. All in all, the matrix divergences have a simple connection to the R\'{e}nyi divergences, see also Figure~\ref{fig:d2div}. Thus, the claim follows from Theorem~\ref{theor:MatrMajSuff}.
\end{proof}

Let us also note that Theorem \ref{theor:MatrApproxCatMaj} implies the following result for catalytic asymptotic relative majorization which  can be seen as a strengthening of Theorem 20 in \cite{GoToma2022}.

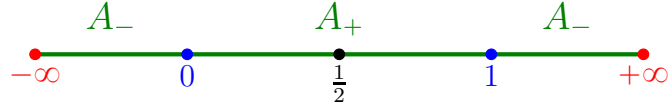
\begin{figure}
\captionsetup{width=.9\linewidth}
\begin{tikzpicture}
\draw[green!50!black, ultra thick] (-4,0) -- node[above=2pt]{\begin{large}$A_-$\end{large}} (-2,0);
\draw[green!50!black, ultra thick] (-2,0) -- node[above=2pt]{\begin{large}$A_+$\end{large}} (2,0);
\draw[green!50!black, ultra thick] (2,0) -- node[above=2pt]{\begin{large}$A_-$\end{large}} (4,0);
\filldraw [red] (-4,0) circle (2pt) node[anchor=north]{$-\infty$};
\filldraw [red] (4,0) circle (2pt) node[anchor=north]{$+\infty$};
\filldraw [blue] (-2,0) circle (2pt) node[anchor=north]{$0$};
\filldraw [blue] (2,0) circle (2pt) node[anchor=north]{$1$};
\filldraw [black] (0,0) circle (2pt) node[anchor=north]{$\frac{1}{2}$};
\end{tikzpicture}
\caption{Here we illustrate the test spectrum of the relative majorization semiring $S^2$. It corresponds to the set of R\'{e}nyi divergences as is shown in the proof of Corollary~\ref{cor:RelativeMaj}, which is here drawn as the two-point compactification of $\R$. On the right-hand side of the centre point $\alpha=1/2$ (in black), we have the divergences $(p^{(1)},p^{(2)})\mapsto D_\alpha(p^{(1)}\|p^{(2)})$, and on the left-hand side of the centre point, we have the divergences $(p^{(1)},p^{(2)})\mapsto D_\alpha(p^{(2)}\|p^{(1)})$. Points $\alpha=0,1$ (in blue) correspond to the two essential derivations which are simply the Kullback-Leibler divergences with the two different orderings of its arguments, and the points $\alpha=\pm\infty$ (in red) correspond to the divergence $D_\infty$ with the two different orderings.}
\label{fig:d2div}
\end{figure}

\begin{corollary}
For a dichotomy $(p^{(1)},p^{(2)})$ of probability vectors with common support and another dichotomy $(q^{(1)},q^{(2)})$ of probability vectors with common support, both in $\mc P_n^2$, the following are equivalent:
\begin{enumerate}[label=(\roman*)]
\item There is a sequence $(q^{(1)}_k)_{k=1}^\infty$ of probability vectors in $\mc P_n$ with $\|q^{(1)}-q^{(1)}_k\|_1\to0$ as $k\to\infty$ and a sequence $(r^{(1)}_k,r^{(2)}_k)_{k=1}^\infty$ of dichotomies of probability vectors such that
\begin{equation}
\left(p^{(1)}\otimes r^{(1)}_k,p^{(2)}\otimes r^{(2)}_k\right)\succeq\left(q^{(1)}_k\otimes r^{(1)}_k,q^{(2)}\otimes r^{(2)}_k\right)
\end{equation}
for all $k \in \mathbb{N}$.
\item For all $\alpha\geq1/2$,
\begin{align}
D_\alpha\left(p^{(1)}\middle\|p^{(2)}\right)&\geq D_\alpha\left(q^{(1)}\middle\|q^{(2)}\right),\\
D_\alpha\left(p^{(2)}\middle\|p^{(1)}\right)&\geq D_\alpha\left(q^{(2)}\middle\|q^{(1)}\right).
\end{align}
\end{enumerate}
\end{corollary}

The proof follows directly from Theorem~\ref{theor:MatrApproxCatMaj} and the identification of the divergences $D_{(\alpha,1-\alpha)}$ with the R\'{e}nyi divergences as seen in the preceding proof. Note that only $q^{(1)}$ is approximated but not $q^{(2)}$, or in other words we require exact catalytic majorization in the second component of the dichotomies, which can be achieved by Theorem~\ref{theor:MatrApproxCatMaj}. Naturally, we may also do it the other way around instead. Theorem 20 of~\cite{GoToma2022} is the same result but with approximations for both $q^{(1)}$ and $q^{(2)}$, and it is in this sense that our present result is stronger.

\section{Conclusion}
\label{sec:conclusion}

In this work we have connected several problems of majorization to the theory of preordered semirings. Using this framework, we have derived new sufficient and generically necessary conditions for large-sample and catalytic matrix majorization, i.e.,\ Blackwell dominance in the finite-outcome setting. In this latter context, we have in particular identified the relevant monotone quantities as \emph{matrix $\mathuline{\alpha}$-divergences} which can be viewed as a generalization of the one-parameter family of the standard R\'{e}nyi divergences.

We would like to point out that this theory has some well-known connections not only to problems of classical statistics, but also to \emph{resource theories} in the context of quantum information processing. All basic quantum devices\,---\,quantum states, channels, and observables\,---\,can also be described as suitable generalizations of statistical experiments (see, e.g.,\ \cite{Kuramochi2017}). Specifically, in the resource theory of entanglement of pure quantum states, where we may say that a pure (vector) state $\fii$ of a compound system $AB$ is `above' another pure state $\psi$ if there is a local-operations-and-classical-communication (LOCC) quantum channel $\Phi$ such that $\Phi(|\fii\>\<\fii|)=|\psi\>\<\psi|$. It is well-known \cite{nielsen99} that this happens if and only if the Schmidt vector of $\psi$ majorizes that of $\fii$. Thus one obtains conditions for catalytic LOCC-ordering for pure states from Theorem \ref{thm:klimeshTheorem}, as also pointed out in \cite{klimesh2007inequalities}. Similar results also hold for incoherent transformations of pure states \cite{DuBaiGuo2015}.

Multiple majorization of quantum states is also a relevant question. This is simply the quantum counterpart of matrix majorization: Given $d$-tuples $(\rho_1,\ldots,\rho_d)$ and $(\sigma_1,\ldots,\sigma_d)$ of quantum states (possibly on different systems), we may ask if there is a quantum channel $\Lambda$ such that $\Lambda(\rho_k)=\sigma_k$ for $k=1,\ldots,d$. The one-shot conditions for this problem can be obtained, e.g.,\ from \cite{Chefles2009,Shmaya2005}. The algebraic machinery consisting of the \emph{Vergleichsstellens\"atze} from \cite{Fritz2021a,Fritz2021b} is applicable also in this quantum case: We can build a quantum majorization semiring of polynomial growth in the same way as we did for matrix majorization. In this way, majorization of quantum statistical experiments in large samples or catalytically can be characterized by inequalities involving the monotone homomorphisms and monotone derivations (jointly forming its test spectrum) on the quantum majorization semiring.

The problem with this approach is, however, that the characterization of the test spectrum of the quantum majorization semiring remains, for now, intractable. For one thing, we cannot similarly decompose the additive maps in the quantum case as we did in Propositions \ref{prop:MatrMonHom} and \ref{prop:MatrDerivation}, because the states in the $d$-tuple typically have different eigenbases. Problems arise already in the relative case $d=2$ \cite{Tomamichel2016Book}, where an innumerable family of quantum divergences generalizing the classical R\'{e}nyi divergences has been defined and we do not yet know which of them are actually relevant for majorization questions. Therefore we cannot give specific and explicit conditions for quantum majorization in large samples or catalytically yet, even in the case $d=2$. Some special cases can be investigated with the methodology presented in the current work, but this will be the topic of future research. Also, using relaxed preorders like a quantum version of submajorization may be fruitful, as then one is sometimes able to narrow things down to a manageable family of sufficient conditions \cite{perry2021semiring, bunth2021asymptotic}.

\appendix

\section{The majorization and submajorization semirings}\label{sec:majsemirings}

In these appendices we concentrate on (simple) majorization of probability vectors and rederive some previously established results on large-sample and catalytic majorization in this mode. We will see that, as in the case of matrix majorization, the new algebraic machinery at our disposal is remarkably powerful in questions on simple majorization as well.

A probability vector $p$ is said to majorize $q$ whenever there exists a bistochastic matrix $T$ such that $Tp = q$. Here, we view $p$ and $q$ as vectors of the same size $n$ by padding with zeros if necessary, and $T$ being bistochastic means that it is an $(n\times n)$-matrix with non-negative entries whose rows and columns all sum up to 1. According to the Hardy-Littlewood-P\'{o}lya theorem, $p$ majorizes $q$ if and only if
\begin{align}\label{eq:majorizationProb}
\sum_{i=1}^k p^\downarrow_i\geq\sum_{i=1}^k q^\downarrow_i
\end{align}
for all $k=1,\ldots,n$, where $p^\downarrow=(p_i^\downarrow)_{i=1}^n$ and $q^\downarrow=(q^\downarrow_i)_{i=1}^n$ are the vectors obtained from $p$ and $q$ by listing their entries in non-increasing order. The definition generalizes to vectors $p$ and $q$ having non-negative entries, i.e.,\ we do not have to assume that $\|p\|_1=1=\|q\|_1$. Given the importance of majorization and related preorders, our aim in the following subsections is to analyze majorization algebraically in terms of a commutative semiring of nonnegative vectors equipped with the majorization preorder or one of its variants, and to show how when using the general theory of separation theorems introduced in Section \ref{sec:preliminaries}, we recover some of the classical results on catalytic and large-sample transformations, which were originally proven by more specific methods that do not apply to other cases.

\subsection{Majorization Semirings}

Let us concentrate on finitely supported vectors with non-negative entries. These form the set $\mc V_{<\infty}:=\bigcup_{n=1}^\infty\R_+^n$. For every $p\in \mc V_{<\infty}$, denote by $p^\downarrow$ the decreasing rearrangement of $p$, so that the vectors $p^\downarrow$ and $p$ have the same components up to permutation. We consider $p,q\in\mc V_{<\infty}$ as equivalent, and denote this by $p\approx q$, if $p^\downarrow$ and $q^\downarrow$ coincide modulo padding by zeros. We write $\mc V_{<\infty}/\!\approx$ for the set of the $\approx$-equivalence classes $[p]$. We abuse notation by writing $1:=[(1)]$, the equivalence class of the single-element vector with entry $1$, and $0:=[(0)]$. The sum and product of $[p],[q]\in\mc V_{<\infty}/\!\approx$ are defined by
\begin{equation}
[p]+[q]:=[p\oplus q],\qquad [p]\cdot[q]:=[p\otimes q],
\end{equation}
where the order of components in the Kronecker product is irrelevant up to $\approx$.

\begin{definition}
We define three preorders on $\mc V_{<\infty}/\!\approx$ as follows. Given two generic elements $p=(p_1,\ldots,p_n)\in\mc V_{<\infty}$ and $q=(q_1,\ldots,q_n)\in\mc V_{<\infty}$, we assume that the vectors $p$ and $q$ have the same length $n$, which can always be achieved up to $\approx$ by padding with zeros.
\begin{enumerate}[label={(\alph*)}]
\item {\it Submajorization preorder}: $[q]\rleq[p]$ if, for all $k=1,\ldots,n$, 
\begin{equation}
	\sum_{i=1}^k q^\downarrow_i\leq\sum_{i=1}^k p^\downarrow_i,
\end{equation}
or equivalently if there exists a bistochastic matrix $T=(T_{i,j})_{i,j=1}^n\in\mc M_n(\R_+)$ such that $q \leq Tp$ (entrywise inequality). When equipped with the submajorization order, $\mc V_{<\infty}/\!\approx$ becomes a preordered semiring (see below). We denote it by $S$ and call it the {\it submajorization semiring}. Between individual vectors $p$ and $q$, we denote from now on $p\minor q$ or $q\major p$ when $[q]\rleq[p]$.
\item {\it Majorization preorder}: $[q]\rleq_1 [p]$ if $[q]\rleq[p]$ and $\|p\|_1=\|q\|_1$.
When equipped with the majorization order, $\mc V_{<\infty}/\!\approx$ becomes a preordered semiring which we call the {\it majorization semiring} and denote by $S_1$. Between individual vectors $p$ and $q$, we denote from now on $p\minor_1 q$ or $q\major_1 p$ when $[q]\rleq_1[p]$.
\item {\it Modified majorization order}: $[q]\rleq_{0,1}[p]$ if $[q]\rleq_1[p]$ and $\|p\|_0=\|q\|_0$, where $\|\cdot\|_0$ denotes the size $|{\rm supp}\,p|$ of the support of $p$, i.e.,\ the number of outcomes $i$ such that $p_i>0$. When equipped with the modified majorization order, $\mc V_{<\infty}/\!\approx$ becomes a preordered semiring which we call the {\it modified majorization semiring} and denote by $S_{0,1}$. Between individual vectors $p$ and $q$, we denote from now on $p\minor_{0,1} q$ or $q\major_{0,1} p$ when $[q]\rleq_{0,1}[p]$.
\end{enumerate}
\end{definition}

\begin{proposition}
	$S$, $S_1$ and $S_{0,1}$ are preordered semirings.
\end{proposition}
\begin{proof}[Proof for $S$]
For $p,q \in\mc V_{<\infty}$, suppose that $q \major p$ is witnessed by a bistochastic matrix $T$ with $Tp \ge q$,\footnote{Recall that the symbol $\geq$ here denotes entry-wise inequality between vectors.} where we assume without loss of generality that $p$ and $q$ have the same length $n$. We need to show that this majorization relation is stable under addition and multiplication by any $r \in\mc V_{<\infty}$. Let us define the block matrix $S:=T\oplus I_m$, where $m$ is the size of $r$ and $I_m$ is the $(m\times m)$-identity matrix. Thus
\begin{equation}
S(p \oplus r) = Tp \oplus r \geq q \oplus r.
\end{equation}
Clearly $S$ is bistochastic, and hence $[q] + [r] \rleq [p] + [r]$.

Similarly, by using the map $T$ tensored with the identity matrix $I_m$, we obtain
\begin{equation}
(T \otimes I_m)(p \otimes r) = Tp\otimes r \geq q \otimes r,
\end{equation}
and therefore the required $[p] \cdot [r] \rgeq [q] \cdot [r]$.
\end{proof}

\begin{proof}[Proof for $S_1$ and $S_{0,1}$]
By definition, the preorder on $S_1$ is the intersection of the preorder on $S$ with the relation of having equal $\|\cdot\|_1$. Since the latter map is a homomorphism $S \to \R_+$, the relation of having equal $\|\cdot\|_1$ is itself a (symmetric) semiring preorder. This implies the claim for $S_1$ since the intersection of two semiring preorders is again a semiring preorder. The proof for $S_{0,1}$ is analogous.
\end{proof}

\begin{remark}\label{S_group_semiring}
For algebraically minded readers, it may help to note that the plain semiring structure (that is, without the preorder) of our majorization semirings is exactly that of the \emph{group semiring} $\N[\R_{>0}]$, by which we mean the semiring consisting of all formal sums of elements of the group $(\R_{>0}, \cdot)$, where two such formals sums multiply by first applying distributivity and then using multiplication in $\R_{>0}$ in order to get a new formal sum. This semiring coincides with our $S$, since indeed every element of $S$ can be viewed as a formal sum of positive reals, and these formal sums multiply in the way just described.

This group semiring is characterized by the universal property that the homomorphisms $S \to T$ to any other semiring $T$ correspond bijectively to the group homomorphisms $\R_{>0} \to T^\times$, where $T^\times$ is the group of units of $T$. This universal property is what will enable us below to find a simple characterization of the monotone homomorphisms on our majorization semirings.
\end{remark}

In order to make use of the separation theorems on our majorization semirings, we need to show that these preordered semirings are all of polynomial growth.

\begin{proposition}\label{prop:PolynomialGrowth}
The preordered semirings $S$, $S_1$, and $S_{0,1}$ are of polynomial growth. Moreover, $S$ and $S_1^{\rm op}$ have power universal elements.
\end{proposition}

\begin{proof}
Let us first consider the submajorization semiring $S$. The single-entry element $[(a)]=:u$ with any $a>1$ is a power universal. To see this, let $p=(p_1,\ldots,p_n)\in\mc V_{<\infty}$ and $q=(q_1,\ldots,q_n)\in\mc V_{<\infty}$ be non-zero. It follows immediately that, for $a>1$, there is $k\in\N$ such that
\begin{equation}
\sum_{i=1}^\ell q^{\downarrow}_i\leq a^k\sum_{i=1}^\ell p^{\downarrow}_i
\end{equation}
for all $\ell=1,\ldots,n$, i.e., $a^k p \minor q$. This means that $u$ is a power universal in $S$. We could similarly show that also $S^{\rm op}$ has a power universal element given by $[(a)]$ where $0<a<1$.

Let us go on to the majorization semiring $S_1$. We show that $S_1^{\rm op}$ has a power universal element so that $S_1^{\rm op}$ and, hence, $S_1$ is of polynomial growth. Concretely, we prove that $[r]$ where $r=(t,1-t)$ with any $t\in(0,1)$ is a power universal for $S_1^{\rm op}$. We may freely assume that $t\geq 1/2$. Let $p,q\in\mc V_{<\infty}$ be such that $\|p\|_1=\|q\|_1$. We need to show that there is $k\in\N$ such that $p\minor_1 r^{\otimes k}\otimes q$. Denote the smallest non-zero entry of $p$ by $p_{\rm min}$ and the largest entry of $q$ by $q_{\rm max}$. It is immediate that for every sufficiently large $k\in\N$, the largest entry $t^k q_{\rm max}$ of $r^{\otimes k}\otimes q$ is upper bounded by $p_{\rm min}$. This implies $p\minor_1 r^{\otimes k}\otimes q$. More concretely, we may choose any
\begin{equation}
k\geq\frac{\log{p_{\rm min}}-\log{q_{\rm max}}}{\log{t}}
\end{equation}
in order to get the desired $p \minor_1 r^{\otimes k} \otimes q$.

Let us finally show that the modified majorization semiring $S_{0,1}$ is of polynomial growth. For every $s\in\N$, recall that we write
\begin{equation}
	u_s := (1/s,\ldots,1/s)
\end{equation}
for the uniform distribution on $s$ outcomes. Define $t_-:=u_2$ and $t_+:=(1-\eps,\eps)$ for every fixed $\eps\in(0,1/2)$. We then show that $([t_-], [t_+])$ is a power universal pair for $S_{0,1}$. Indeed $t_- \major_{0,1} t_+$ is obvious. For the main condition, it is enough to show that for every $p,q\in\mc V_{<\infty}$ with $\|p\|_1=\|q\|_1$ and $\|p\|_0=\|q\|_0$, there is $k\in\N$ such that $t_-^{\otimes k}\otimes p\major_{0,1} t_+^{\otimes k}\otimes q$. Clearly we may assume that $\|p\|_1=1=\|q\|_1$. Let us denote $n:=\|p\|_0=\|q\|_0$ and by $p_{\rm min}$ the smallest non-zero entry of $p$. Assuming $n \ge 2$, we now define
\begin{equation}
\tilde{p}:=\Big(1-p_{\rm min},\underbrace{\frac{p_{\rm min}}{n-1},\ldots,\frac{p_{\rm min}}{n-1}}_{n-1\ {\rm copies}}\Big),
\end{equation}
so that $\tilde{p}\minor_{0,1} p$. We next prove that there is $k\in\N$ such that $t_-^{\otimes k}\otimes\tilde{p}\major_{0,1} t_+^{\otimes k}\otimes u_n$. This will imply what we are trying to show because
\begin{equation}
t_-^{\otimes k}\otimes p\major_{0,1} t_-^{\otimes k}\otimes\tilde{p}\major_{0,1} t_+^{\otimes k}\otimes u_n\major_{0,1} t_+^{\otimes k}\otimes q,
\end{equation}
where the final inequality follows from the fact that $q\minor_{0,1} u_n$. Thus, we next show the validity of the middle inequality above for sufficiently large $k\in\N$.

The initial partial sums of $(t_-^{\otimes k}\otimes\tilde{p})^{\downarrow}$ increase linearly until the $2^k$-th sum, where they reach the value $1-p_{\rm min}$. Afterwards the initial sums grow again linearly until finally reaching the value $1$. Due to the fact that the graph of initial sums is a concave function reaching the value 1 at the final sum, in order to find $k$ so that $t_-^{\otimes k}\otimes\tilde{p}\major_{0,1} t_+^{\otimes k}\otimes u_n$, it is enough to guarantee that the $2^k$-th initial sum of $(t_+^{\otimes k}\otimes u_n)^{\downarrow}$ is at least $1-p_{\rm min}$.

We can bound this initial sum as
\begin{equation}
\sum_{i=1}^{2^k} (t_+^{\otimes k}\otimes u_n)^{\downarrow}_i\geq \sum_{j=0}^{a} n \binom{k}{j} \frac{(1-\eps)^{k - j} \eps^j}{n}=\sum_{j=0}^{a} \binom{k}{j} (1-\eps)^{k - j} \eps^j \label{eq:binomial}
\end{equation}
for every $a \in [0, k]$ such that $\sum_{j=0}^{a} n \binom{k}{j} \leq 2^k$. The right-hand side is the cumulative distribution function at $a$ of a binomial random variable with mean $\eps k$. Let us now choose $\delta > 0$ such that $\eps + 2 \delta < \frac12$ and $k$ with $k\delta > 1$. Choosing furthermore any $a \in [k(\eps+\delta), k(\eps+2\delta)]$, as we increase $k$ the cumulative distribution function in~\eqref{eq:binomial} will converge to $1$ by the law of large numbers, and in particular at some point exceed $1 - p_{\min}$. It remains thus to check that, for such a choice of $a$, the condition $\sum_{j=0}^{a} n \binom{k}{j} \leq 2^k$ is satisfied for large enough $k$. To verify this, we use the well-known bound (see, e.g., Section 1.4 in \cite{Lint99}),
\begin{align}\label{eq:binomialineq}
	\sum_{j=0}^{a} \binom{k}{j} \leq \exp{\big(k\, h_1(\eps + 2\delta)\big)} ,
\end{align}
where $h_1:[0,1]\to\R$,
\begin{align}
h_1(x)&=-x\log{x}-(1-x)\log{(1-x)},\\
0\log{0}&:=0,
\end{align}
is the binary entropy and note that, since $h_1(\eps + 2\delta) < 1$ by our choice of $\delta$, we evidently have $n \exp{\big(k\, h_1(\eps + 2\delta)\big)} < \exp{k}$ for sufficiently large $k$, completing the proof.
\end{proof}

\begin{remark}\label{rem:polygrowthmaj}
	Whenever $r\in\mc P_{<\infty}$ is any probability vector with $[r] \neq 1$ (i.e.,\ $\|r\|_0\geq2$), then $[r]$ is already a power universal for $S_1^{\rm op}$. To see this, denote the smallest non-zero entry of $r$ by $r_{\rm min}$ and $t:=1-r_{\rm min}\in(0,1)$. It follows that $\tilde{r}:=(t,1-t)$ majorizes $r$. Using the beginning of the proof of Proposition~\ref{prop:PolynomialGrowth} for $S_1$, where it was shown that $[\tilde{r}]$ is a power universal, we have for all $p,q\in\mc V_{<\infty}$ with $\|p\|_1=\|q\|_1$ and for a sufficiently large $k\in\N$,
\begin{equation}
r^{\otimes k}\otimes q\major_1 \tilde{r}^{\otimes k}\otimes q\major_1 p,
\end{equation}
proving the claim.
\end{remark}

\begin{remark}
We have seen that $S$ has a power universal element, and it is also easy to see that it is zerosumfree. Given that also $0 \rleq 1$ in $S$, the conditions of Theorem \ref{thm:Fritz7.15} are fulfilled for $S$. Concerning $S_1$ and $S_{0,1}$, we now argue that both satisfy the auxiliary conditions of Theorems \ref{thm:Fritz7.1} and \ref{thm:Fritz8.6}. Indeed, denoting the equivalence relation on $\mc V_{<\infty}$ generated by $\rleq_1$ by $\sim_1$, we have $p \sim_1 q$ if and only if $\|p\|_1 = \|q\|_1$. This means that $S_1/\!\sim_1\,\simeq\R_+$, which clearly has quasi-complements and quasi-inverses in the sense of~\cite{Fritz2021b}, and
\begin{equation}
\ms{Frac}(S_1/\!\sim_1)\otimes\Z\cong\R
\end{equation}
is a field and therefore trivially a finite product of fields. 

Similarly, we have then $p\sim_{0,1} q$ if and only if $\|p\|_1=\|q\|_1$ and $\|p\|_0=\|q\|_0$. Therefore
\begin{equation}
S_{0,1}/\!\sim_{0,1}\,\simeq(\R_{>0}\times\N)\cup\{(0,0)\},
\end{equation}
which is easily seen to have quasi-complements and quasi-inverses. Moreover, we have
\begin{equation}
\ms{Frac}(S_{0,1}/\!\sim_{0,1})\otimes\Z\cong\R\times\Q,
\end{equation}
which is a product of two fields.
\end{remark}

In order to apply the separation theorems, we first concentrate on $S_{0,1}$ and characterize all nondegenerate monotone homomorphisms $\Phi:S_{0,1}\to\mb K$ where $\mb K\in\{\R_+,\R_+^{\rm op},\T\R_+,\T\R_+^{\rm op}\}$, as well as the monotone derivations at the degenerate monotone homomorphisms. We will then deduce the corresponding results for $S$ and $S_1$ from those of $S_{0,1}$.

The application of the separation theorems then yields sufficient and generically necessary criteria for large-sample and catalytic majorization between any pair of vectors $p, q \in\mc V_{<\infty}$. Then Theorem~\ref{thm:Fritz7.15} applied to $S$ will let us conclude large-sample and catalytic majorization if the values of all relevant monotones are strictly ordered. Since $\|\cdot\|_1 : S \to \R_+$ is a monotone homomorphism, this in particular requires strict ordering between $\|p\|_1$ and $\|q\|_1$, and therefore we obtain no conclusion if both $p$ and $q$ are probability vectors. However, this is where $S_1$ comes to the rescue: applying Theorem~\ref{thm:Fritz8.6} to $S_1$ with $\|\cdot\| = \|\cdot\|_1$ only requires strict ordering with respect to \emph{other} monotone homomorphisms, and therefore it still gives sufficient conditions for large-sample and catalytic majorization that are necessary in generic cases. But since $\|\cdot\|_0$ is one of the relevant monotone homomorphisms, again this does not apply in the non-generic case $\|p\|_0 = \|q\|_0$. But this is the case in which we can apply Theorem~\ref{thm:Fritz7.1} to $S_{0,1}$ in order to obtain a sufficient and generically necessary criterion for catalytic majorization.

\subsection{Monotone homomorphisms and derivations on the modified majorization semiring}

We first treat the case of the modified majorization semiring. Among our majorization semirings, this one has the largest number of monotone homomorphims since the preorder $\minor_{0,1}$ is the smallest among the majorization preorders. Throughout the following sections, we abuse notations by treating monotone homomorphisms and derivations on $S_{0,1}$ as well as $S_1$ and $S$ as functions on $\mc V_{<\infty}$ although they are actually defined on $\mc V_{<\infty}/\!\approx$; this just means that as functions on $\mc V_{<\infty}$ they are constant on each $\approx$-equivalence class. This abuse of notation is used to avoid extra brackets indicating the equivalence classes and, thus, to lighten our notation.

The monotone homomorphisms on $S_{0,1}$ will be found among the following functions, defined on any $p=(p_1,\ldots,p_n)\in\mc V_{<\infty}$ with support $I:={\rm supp}(p)$ as
\begin{align}
f_\alpha(p)&=\sum_{i\in I}p_i^\alpha,\qquad\alpha\in\R,\label{eq:falpha}\\
f_\infty(p)&=\max_{i\in I} p_i,\label{eq:finfty}\\
f_{-\infty}(p)&=\left(\min_{i\in I} p_i\right)^{-1},\label{eq:neqinfty}\\
H_1(p)&=-\sum_{i\in I}p_i\log{p_i},\label{eq:H1}\\
H'_0(p)&=\sum_{i\in I}\log{p_i}.\label{eq:H_0}
\end{align}
Here, we stipulate that empty sums vanish, $\max\emptyset:=0$, $\min\emptyset:=\infty$, $1/\infty:=0$ and $0 \log 0 := 0$. Note that the map $H_1$ defined in Equation~\eqref{eq:H1} is the Shannon entropy. Also $f_1=\|\cdot\|_1$ and $f_0=\|\cdot\|_0$, where the two `norms' are as before.

\begin{proposition}
\label{prop:monhomMod}
The nondegenerate monotone homomorphisms $S_{0,1} \to \mb K$ are exactly the following:
\begin{enumerate}[label=(\roman*)]
\item For $\mb K = \R_+$, the maps $f_\alpha$ for $\alpha \in (-\infty,0) \cup [1,\infty)$.
\item For $\mb K = \R_+^{\rm op}$, the maps $f_\alpha$ for $\alpha\in (0,1)$.
\item For $\mb K = \T\R_+$, the maps $f_{{\rm sign}(\alpha)\infty}(\cdot)^{|\alpha|}$ for $\alpha \in (-\infty,0) \cup (0,\infty)$.
\item For $\mb K = \T\R_+^{\rm op}$, there are none.
\end{enumerate}
The only degenerate homomorphisms $S_{0,1} \to \R_+$ are $f_0 = \|\cdot\|_0$ and $f_1=\|\cdot\|_1$.
\end{proposition}

\begin{proof}
Given a monotone homomorphism $\Phi : S_{0,1} \to \mb K$, we write $\fii:\R_+\to \mb K$ for the restriction of $\Phi$ to the set of (equivalence classes of) single-entry vectors, $\fii(x) := \Phi([(x)])$. Thus, $\fii$ is multiplicative, i.e.,\
\begin{equation}
	\fii(xy)=\fii(x)\fii(y) \qquad \forall x,y\in\R_+, \qquad \fii(1) = 1,
\end{equation}
by the multiplicativity of $\Phi$. This almost ensures that $\fii$ is a monomial function $\fii(x) = x^\alpha$: by the standard theory of the Cauchy functional equation, to reach this conclusion it is enough to show in addition that $\fii$ is bounded in some interval $[x,y]$ for arbitrary $0\leq x < y$ \cite[Chapter 2, Theorem 8]{Aczel89}\footnote{Note that this theorem treats the case of the linear form of the Cauchy functional equation. We may transform our problem to this form by considering the function $\log\circ\fii\circ\exp$ instead.}, which we now go on to show. To this end, note that for all $t\in[0,1]$,
\begin{equation}
\left(\frac{x+y}{2},\frac{x+y}{2}\right)\major_{0,1}\big(tx+(1-t)y,(1-t)x+ty\big)\major_{0,1}(x,y),
\end{equation}
as witnessed by the bistochastic matrices
\begin{equation}
T_1=\left(\begin{array}{cc}
\frac{1}{2} & \frac{1}{2} \\[2pt]
\frac{1}{2} & \frac{1}{2}
\end{array}\right),\qquad
T_2=\left(\begin{array}{cc}
t&1-t\\
1-t&t
\end{array}\right).
\end{equation}
Thus, by the additivity and monotonicity of $\Phi$ we have
\begin{align}
\textrm{case (i): }2\fii\left(\frac{x+y}{2}\right)\leq{}&\fii\big(tx+(1-t)y\big)+\fii\big((1-t)x+ty\big)\\
\leq{}&\fii(x)+\fii(y),\\
\textrm{case (ii): }2\fii\left(\frac{x+y}{2}\right)\geq{}&\fii\big(tx+(1-t)y\big)+\fii\big((1-t)x+ty\big)\\
\geq{}&\fii(x)+\fii(y),\\
\textrm{case (iii): }\fii\left(\frac{x+y}{2}\right)\leq{}&\max_{s\in\{t,1-t\}}\fii\big(sx+(1-s)y\big)\\
\leq{}&\max\{\fii(x),\fii(y)\},\\
\textrm{case (iv): }\fii\left(\frac{x+y}{2}\right)\geq{}&\max_{s\in\{t,1-t\}}\fii\big(sx+(1-s)y\big)\\
\geq{}&\max\{\fii(x),\fii(y)\},
\end{align}
where the inequalities are with respect to the usual order of numbers. It is now enough to note that in every case, these inequalities imply a $t$-independent upper bound on $\fii\big(tx + (1-t)y\big)$ based on the assumption that $\fii$ only takes nonnegative values. Hence $\fii$ is bounded on $[x,y]$, and therefore there is $\alpha\in\R$ such that $\fii(x)=x^\alpha$ for all $x\geq0$.

Turning to the main claims, consider any $p \in\mc V_{<\infty}$ and denote $I:={\rm supp}\,p$. We now have
\begin{equation}\label{eq:(i)et(ii)}
\Phi(p)=\Phi\left(\bigoplus_{i\in I}(p_i)\right)=\sum_{i\in I}\fii(p_i)=\sum_{i\in I}p_i^\alpha = f_\alpha(p)
\end{equation}
in cases (i) and (ii) and
\begin{equation}
\Phi(p)=\Phi\left(\bigoplus_{i\in I}(p_i)\right)=\max_{i\in I}\fii(p_i)=\max_{i\in I}p_i^\alpha=\begin{cases}
f_\infty(p)^\alpha & \textrm{if } \alpha \ge 0,	\\
f_{-\infty}(p)^{-\alpha} & \textrm{if } \alpha \le 0
	\end{cases}\label{eq:(iii)et(iv)}
\end{equation}
in cases (iii) and (iv). Therefore $\Phi$ is of the desired form in all cases, where the range of values of $\alpha$ that make these maps into nondegenerate monotone homomorphisms remains to be determined.

It is easy to see that the maps $f_\alpha$ are indeed homomorphisms $S_{0,1} \to \R_+$ for all $\alpha \in \R$, and similarly that the maps $f_\infty(\cdot)^\alpha$ and $f_{-\infty}(\cdot)^{-\alpha}$ are homomorphisms $S_{0,1} \to \T\R_+$ for all $\alpha \in \R$. To finish the proof of the main claims, it is therefore enough to show that these maps are nondegenerate and monotone if and only if $\alpha$ satisfies the stated conditions. Let us concentrate on cases (i) and (ii) for now, and consider $\alpha\in[0,1]$ first. To analyze the monotonicity of $f_\alpha$, let $p,q\in\mc V_{<\infty}$ be such that $p\minor_{0,1} q$. Thus, we may assume that there is $n\in\N$ such that $p=(p_1,\ldots,p_n)\in\R_{>0}^n$, $q=(q_1,\ldots,q_n)\in\R_{>0}^n$ and that there is bistochastic $T=(T_{i,j})_{i,j=1}^n$ such that $q=Tp$. Using the fact that the function $x\mapsto x^\alpha$ is concave for $\alpha \in [0,1]$, we may evaluate
\begin{equation}
\label{falpha_mon}
f_\alpha(q)=\sum_{i=1}^n\left(\sum_{j=1}^n T_{i,j}p_j\right)^\alpha\geq\sum_{i,j=1}^n T_{i,j}p_j^\alpha=\sum_{j=1}^n p_j^\alpha=f_\alpha(p).
\end{equation}
Therefore $f_\alpha : S_{0,1} \to \R_+^{\rm op}$ is a monotone homomorphism for $\alpha \in [0,1]$. At $\alpha = 0$ and $\alpha = 1$, we recover the homomorphisms $\|\cdot\|_0$ and $\|\cdot\|_1$, which are degenerate by definition of $\succeq_{0,1}$. For $\alpha \in (0,1)$, we can show the nondegeneracy for example by considering $p_t=(1-t,t)$ for $t\in(0,1/2]$, so that, whenever $0<s\leq t$, we have $p_s\minor_{0,1} p_t$. We then have $f_\alpha(p_t)=(1-t)^\alpha+t^\alpha$, which is easily seen to be a strictly increasing function of $t$ for $\alpha \in (0,1)$.

Similarly, using the convexity of $x\mapsto x^\alpha$ when $\alpha\leq0$ or $\alpha\ge 1$, we see that $f_\alpha:S_1\to\R_+$ is a monotone homomorphism in these cases.\footnote{Note that, when $\alpha<0$, the assumption on coinciding support sizes (i.e.,\ the possibility of assuming that all the entries of $p$ and $q$ as above are non-zero) is crucial because, otherwise, we end up with ill-defined expressions of the form $0^\alpha$.} The nondegeneracy proof for $\alpha \neq 0,1$ works similarly. Thus we obtain the claims (i) and (ii).

Let us finally look at cases (iii) and (iv). The map $f_\infty(\cdot)^\alpha:S_{0,1}\to\T\R_+$ is a monotone homomorphism for every $\alpha \ge 0$. To see this, take any $p,q\in\mc V_{<\infty}$ as in~\eqref{falpha_mon}, and similarly evaluate
\begin{equation}
	\label{finfty_mon}	
	f_\infty(q)^\alpha = \left( \max_{1\leq i\leq n} \sum_{j=1}^n T_{i,j} p_j \right)^\alpha \leq \left( \max_{1\leq i\leq n} p_i \right)^\alpha = f_\infty(p)^\alpha,
\end{equation}
where we have used the fact that a convex combination can never increase the maximum of the elements which one is convexly combining. We may do exactly the same for $f_{-\infty}(\cdot)^\alpha$ with $\alpha\geq0$. Since for $\alpha > 0$ these inequalities are strict in generic cases like the $p_t$ above, $f_{\pm\infty}(\cdot)^\alpha$ is not monotone as a map to $\T\R_+^{\rm op}$ unless $\alpha = 0$. In particular, we also obtain nondegeneracy as a map $S_{0,1} \to \T\R_+$ for all $\alpha \neq 0$.
\end{proof}

Since we have two degenerate homomorphisms of $S_{0,1}$ into $\R_+$, we have to consider monotone derivations at both $\|\cdot\|_1$ and $\|\cdot\|_0$. While the classification of monotone derivations at $\|\cdot\|_1$ can be done using existing results on axiomatics for Shannon entropy (see Remark~\ref{rem:axiomatic}), we provide a full proof as it illustrates the utility of semiring-theoretic methods again. As we shall see later, the case of derivations at $\|\cdot\|_0$ is much simpler than that at $\|\cdot\|_1$.

Before diving into the details of the classification of these derivations, let us present and prove a useful lemma which we will need when characterizing the monotone derivations at $\|\cdot\|_1$.

\begin{lemma}\label{lemma:UpperLowerBound}
Fix $n\in\N$ and let $\Delta:\R_+^n\to\R$ be such that
\begin{equation}
	p\minor_{0,1}q \quad \Longrightarrow \quad \Delta(p)\geq\Delta(q)
\end{equation}
for all $p,q \in \R_+^n$.
	Suppose that there is a continuous map $\Lambda:\R_+^n\to\R$
	such that $\Delta(p)=\Lambda(p)$ for all $p\in \Q_+^n$. Then $\Delta(p) = \Lambda(p)$ for all $p \in \R_+^n$ with $\|p\|_1\in\Q$.
\end{lemma}

\begin{proof}
We prove the claim by proving the following: For every $p\in \R_+^n$ with $\|p\|_1\in\Q$ and $\eps>0$, there are $q,r\in \Q_+^n$ such that $q\minor_{0,1} p\minor_{0,1} r$ and $\|q-r\|_1 \le \eps$ as well as $\|q-p\|_1 \le \eps$. This is enough since applying $\Delta$ and $\Lambda$ to this inequality gives
\begin{equation}
	\Lambda(q)=\Delta(q)\geq \Delta(p)\geq \Delta(r)=\Lambda(r),
\end{equation}
where we have used the $\minor_{0,1}$-monotonicity of $\Delta$.
Since the set of nonnegative vectors of fixed normalization is compact, we know that $\Lambda$ is uniformly continuous on it. Therefore taking $\eps \to 0$ then shows that $\Delta(p)=\Lambda(p)$. Thus it suffices to prove that there are $q$ and $r$ as above.

To this end, we may freely assume that our given $p = (p_1,\ldots,p_m,0,\ldots,0) \in \R_+^n$ is already in non-increasing order with $p_m>0$, since the monotonicity of $\Delta$ implies that $\Delta$ is invariant under permutations. By restricting to $\R_+^m$ we may furthermore assume that $m = n$, so that $p$ has full support. By rescaling everything by the rational $\|p\|_1^{-1}$, we may also assume that $\|p\|_1=1$. For given $\eps>0$ we put $\delta:=\frac{\eps}{4(n-1)}$. For every $i=1,\ldots,n-1$, we fix a rational number $q_i \in [p_i, p_i + \delta]$ such that $q_1\geq\cdots\geq q_{n-1}$. We then define the rational number $q_n:=1-\sum_{i=1}^{n-1}q_i$, which by $\|p\|_1 = 1$ satisfies
\begin{equation}
p_n-(n-1)\delta \le q_n \le p_n.
\end{equation}
For sufficiently small $\eps$ this in particular guarantees that $q_n > 0$, and therefore $q = (q_1, \ldots, q_n)$ is a probability vector with $\|q\|_0=\|p\|_0$. The entries of $q$ are already in non-increasing order as well, and we have $\sum_{i=1}^k q_i \ge \sum_{i=1}^k p_i$ for all $k=1,\ldots,n-1$ by construction. Thus $p\major_{0,1} q$. Moreover,
\begin{equation}
\|q-p\|_1=\sum_{i=1}^{n-1}(q_i-p_i)+p_n-q_n \le 2(n-1)\delta=\frac{\eps}{2}.
\end{equation}
To construct the desired vector $r$, let us also assume that $p$ is not uniform, since we can just take $r = p$ in the uniform case. Let then $\ell$ be the unique index with $p_\ell > p_{\ell+1} = p_n$ (i.e.,\ the final $n-\ell$ entries in $p$ coincide with the lowest probability $p_n$). Let us define
\begin{equation}
\delta':=\min\left\{\frac{\eps}{4\ell},\frac{n-\ell}{n}(p_\ell-p_{\ell+1})\right\}.
\end{equation}
We then fix, for all $i=1,\ldots,\ell$, positive rational numbers $r_i$ such that $\frac{n-\ell-1}{n-\ell}\delta' \le p_i-r_i \le \delta'$ and $r_1\geq\cdots\geq r_\ell$. We also let
\begin{equation}
r_{\ell+1} = \cdots=r_n=\frac{1}{n-\ell}\left(1-\sum_{i=1}^\ell r_i\right) \le p_n+\frac{\ell}{n-\ell}\delta'.
\end{equation}
The definition of $\delta'$ guarantees that
\begin{equation}
r_\ell \ge p_\ell - \delta' \ge p_{\ell+1} + \frac{\ell}{n-\ell}\delta' \ge r_{\ell+1} = \cdots = r_n,
\end{equation}
so that the whole vector $r:=(r_1,\ldots,r_n)$ is in non-increasing order as well and all its entries are non-zero. We immediately see that $\sum_{i=1}^k p_i \ge \sum_{i=1}^k r_i$ for all $k=1,\ldots,\ell$. If $\ell = n - 1$, this already means that $p\minor_{0,1}r$. We now show the same in general by proving the remaining inequalities. For $j=1,\ldots,n-\ell-1$,
\begin{align}
\sum_{i=1}^{\ell+j}r_i &= \sum_{i=1}^\ell r_i+\sum_{i=\ell+1}^{\ell+j}r_i\\
	& \le \sum_{i=1}^\ell p_i - \ell\frac{n-\ell-1}{n-\ell}\delta'+\sum_{i=\ell+1}^{\ell+j}p_i+j\frac{\ell}{n-\ell}\delta'\\
	& = \sum_{i=1}^{\ell+j}p_i - \ell\frac{n-\ell-1-j}{n-\ell}\delta'\le \sum_{i=1}^{\ell+j}p_i,
\end{align}
where we use $j \le n-\ell-1$ in the final inequality, which is the last case that needs to be considered before the normalization equation $\sum_{i=1}^n r_i = \sum_{i=1}^n p_i$. Thus, $p\minor_{0,1} r$. Moreover,
\begin{equation}
\|p-r\|_1=\sum_{i=1}^\ell(p_i-r_i)+\sum_{i=\ell+1}^n(r_i-p_i)\leq 2\ell\delta'\leq\frac{\eps}{2},
\end{equation}
where the final inequality follows from the definition of $\delta'$.

To sum up, we have $q\minor_{0,1} p\minor_{0,1} r$ and $\|q-r\|_1\leq\|q-p\|_1+\|p-r\|_1<\frac{\eps}{2}+\frac{\eps}{2}=\eps$, and hence the auxiliary result is proven.
\end{proof}

We now go on to study the two types of derivations. The Leibniz rules for derivations $\Delta_0, \Delta_1:S_{0,1}\to\R$ at $\|\cdot\|_1$ respectively $\|\cdot\|_0$ now read
\begin{align}
\Delta_1(p\otimes q)=&\Delta_1(p)\|q\|_1+\|p\|_1\Delta_1(q),\label{eq:Leibniz1}\\
\Delta_0(p\otimes q)=&\Delta_0(p)\|q\|_0+\|p\|_0\Delta_0(q)\label{eq:Leibniz0}
\end{align}
for all $p,q\in\mc V_{<\infty}$. One may check quite easily that the Shannon entropy $H_1$ satisfies \eqref{eq:Leibniz1} and the function $H'_0$ from~\eqref{eq:H_0} satisfies \eqref{eq:Leibniz0}. We next show that, up to scaling, these are the only monotone derivations at least on those vectors that have rational $1$-norm. However, the scaling is negative to account for the ordering of $S_{0,1}$.

\begin{proposition}\label{prop:ModDeriv}
On vectors of rational $1$-norm, the monotone derivations $\Delta:S_{0,1}\to\R$ are exactly the following:
\begin{enumerate}[label=(\roman*)]
	\item At $\|\cdot\|_1$, the negative multiples of $H_1$.
	\item At $\|\cdot\|_0$, the negative multiples of $H'_0$.
\end{enumerate}
\end{proposition}

\begin{proof}
	For (i), it is well-known that $H_1$ is monotone under majorization in the sense that if $p\minor_{0,1}q$ (or even just $p\minor_1 q$) then $H_1(p)\leq H_1(q)$. So suppose that $\Delta_1:S_{0,1}\to\R$ is any monotone derivation at $\|\cdot\|_1$. Let $\delta_1$ be the single-entry restriction of $\Delta_1$, i.e.,\ $\delta_1(x) := \Delta_1\big([(x)]\big)$ for all $x\in\R_+$.
	The Leibniz rule for $\delta_1$ takes the form
	\begin{equation}
		\delta_1(xy) = \delta_1(x)y + x \delta_1(y) \qquad \forall x,y \ge 0.
	\end{equation}
	It is easy to see that for all $x,y > 0$ and $m,n \in \N$,
\begin{equation}
\Big(\underbrace{\frac{m}{m+n}x+\frac{n}{m+n}y,\ldots,\frac{m}{m+n}x+\frac{n}{m+n}y}_{m+n\ {\rm copies}}\Big)\preceq_{0,1} (\underbrace{x,\ldots,x}_{m\ {\rm copies}},\underbrace{y,\ldots,y}_{n\ {\rm copies}}).
\end{equation}
Applying $\Delta_1$ to this inequality and using monotonicity and additivity, we find that
\begin{equation}
	\label{eq:frac_concavity}
	\delta_1\big(tx+(1-t)y\big)\leq t\delta_1(x)+(1-t)\delta_1(y)
\end{equation}
for all $t\in[0,1]\cap\Q$ and $x,y>0$. We next show that $\delta_1(\frac{1}{2})\leq0$. Since
\begin{equation}
0=\delta_1(1)=\delta_1\left(2\cdot\frac{1}{2}\right)=2\delta_1\left(\frac{1}{2}\right)+\frac{1}{2}\delta_1(2),
\end{equation}
we have $\delta_1(2)=-4\delta_1(\frac{1}{2})$. Using this and~\eqref{eq:frac_concavity} gives
\begin{align}
0&=\delta_1(1)=\delta_1\left(\frac{2}{3}\cdot\frac{1}{2}+\frac{1}{3}\cdot 2\right)\leq\frac{2}{3}\delta_1\left(\frac{1}{2}\right)+\frac{1}{3}\delta_1(2)\\
&=\left(\frac{2}{3}-\frac{4}{3}\right)\delta_1\left(\frac{1}{2}\right)=-\frac{2}{3}\delta_1\left(\frac{1}{2}\right).
\end{align}
Thus, $\delta_1(\frac{1}{2})\leq0$. But then by $\delta_1(1) = 0$ and~\eqref{eq:frac_concavity} again, we have $\delta_1(x) \leq 0$ for all rational $x \in [\frac{1}{2},1]$, and the Leibniz rule extends this inequality to all rational $x \in [0,1]$.

	Let us now consider $\eta(x) := \frac{\delta_1(x)}{x}$ for $x > 0$. The above Leibniz rule for $\delta_1$ shows that this function satisfies the multiplicative-to-additive version of the Cauchy functional equation,
	\begin{equation}
		\eta(xy) = \eta(x) + \eta(y)	\qquad \forall x,y > 0.
	\end{equation}
	We will now argue that $\eta(x) = c \log x$ for all $x \in \Q_{>0}$ and some fixed $c \ge 0$.
	Since $\eta(x) \leq 0$ for rational $x \in (0,1]$, the functional equation implies that $\eta$ is monotonically non-decreasing on rational arguments.
	We can now proceed as in the standard solution to the Cauchy functional equation~\cite[p.~14]{Aczel89}: suppose that there are $x,y \in \Q_{>0}$ such that the points
	\begin{equation}
		\label{cauchy_proof}
		(\log x, \eta(x)), \qquad (\log y, \eta(y))
	\end{equation}
	are linearly independent in $\R^2$.
	Then they span all of $\R^2$, and in particular we can find coefficients $\ell_x, \ell_y \in \Z$ such that 
	\begin{equation}
		\ell_x \log x + \ell_y \log y \ge 0, \qquad \ell_x \eta(x) + \ell_y \eta(y) < 0.
	\end{equation}
	This means that the positive rational $z := x^{\ell_x} y^{\ell_y}$ satisfies $z \ge 1$ and $\eta(z) < 0$, which contradicts $\eta(1) = 0$ and the monotonicity of $\eta$.
	Therefore the points~\eqref{cauchy_proof} are linearly dependent, or equivalently
	\begin{equation}
		\frac{\eta(x)}{\log x} = \frac{\eta(y)}{\log y} \qquad \forall x, y > 0.
	\end{equation}
	Fixing e.g.~$y := 2$ shows that we have $\eta(x) = c \log x$ with $c = \frac{\eta(2)}{\log 2}\geq0$.
	But then it follows that
	\begin{equation}
		\delta_1(x) = c x \log x \qquad \forall x \in \Q_+.
	\end{equation}

	By additivity of $\Delta_1$, we therefore have for all $p \in\mc V_{<\infty}$ with rational entries,
	\begin{equation}
		\Delta_1(p) = \sum_i \delta_1(p_i) =  c \sum_i p_i \log p_i,
	\end{equation}
	or equivalently $\Delta_1(p)=-cH_1(p)$ for all $p\in\Q_+^n$, $n\in\N$. Noting that $H_1$ is continuous and $\Delta_1$ is monotone, Lemma \ref{lemma:UpperLowerBound} implies that $\Delta_1$ coincides with $-cH_1$ on all $p\in\mc V_{<\infty}$ with $\|p\|_1\in\Q$.

For (ii), we first show that $H'_0:S_{0,1}^{\rm op}\to\R$ is indeed monotone (i.e.,\ negative multiples of $H'_0$ are monotone). To this end, let $p,q\in\mc V_{<\infty}$ be such that $p\minor_{0,1} q$. Thus there are $n\in\N$ such that $p=(p_1,\ldots,p_n)\in\R_{>0}^n$ and $q=(q_1,\ldots,q_n)\in\R_{>0}^n$ and a bistochastic matrix $T=(T_{i,j})_{i,j=1}^n$ such that $Tp = q$, and we freely assume that $p$ and $q$ have full support. We may evaluate
\begin{align}
	H'_0(p)= {}&\sum_{j=1}^n \log{p_j}=\sum_{i,j=1}^n T_{i,j}\log{p_j}\\
	\leq {}&\sum_{i=1}^n \log{\sum_{j=1}^n T_{i,j}p_j}=\sum_{i=1}^n \log{q_i}=H'_0(q),
\end{align}
where, in the second equality, we have used the fact that the columns of $T$ sum to 1 and, in the inequality we have used the fact that the rows of $T$ sum to 1 together with the concavity of the logarithm. Thus $H'_0$ is indeed monotone.

Let now $\Delta_0:S_{0,1}\to\R$ be any monotone derivation at $\|\cdot\|_0$, and let $\delta_0$ be the single-entry restriction of $\Delta_0$. In particular, it follows that $\delta_0(0)=0$ and $\delta_0(xy)=\delta_0(x)+\delta_0(y)$ for all $x,y>0$. Let us define the function $k:(0,1)\to\R$ by
\begin{equation}
	k(x)=\Delta_0(x,1-x)=\delta_0(x)+\delta_0(1-x)=\delta_0(x-x^2).
\end{equation}
For every $t\in[0,1]$, let us consider $p_t=(1-t,t)\in\mc V_{<\infty}$. Using the monotonicity of $\Delta_0$ and the fact that $p_s\minor_{0,1} p_t$ for all $s,t\in(0,\frac{1}{2}]$ with $s\leq t$, we see that $k$ is non-increasing on $(0,\frac{1}{2}]$. Thus, $\delta_0$ is non-increasing on $(0,\frac{1}{4}]$. But then $x\mapsto\delta_0(2^x)$ is an additive function $\R \to \R$ which is monotone on a nontrivial interval, and therefore of the form $x\mapsto -c x$ for some constant $c\in\R$. From this we see that $\delta_0(x)=-c\log{x}$ for $x>0$. Hence $\Delta_0=-c H'_0$, and the monotonicity of $\Delta_0$ implies $c \ge 0$.
\end{proof}

\begin{remark}\label{rem:axiomatic}
Let us outline how to obtain the above classification of monotone derivations at $\|\cdot\|_1$ using a known characterization of Shannon entropy. Suppose that $\Delta_1$ is a monotone derivation at $\|\cdot\|_1$ with $\Delta_1(\frac{1}{2},\frac{1}{2})=-1$. This derivation has the following three properties:
\begin{itemize}
\item When viewed as a function on $\mc V_{<\infty}$, as we usually do, $\Delta_1$ is permutation invariant, i.e.,\ the order in which the entries of the argument vector appear does not matter. This is simply because, within our framework, $\Delta_1$ is formally a function on the set $\mc V_{<\infty}/\!\approx$ of equivalence classes of differently ordered nonnegative vectors.
\item The function $h:[0,1]\to\R$ with $h(x)=-\Delta_1(x,1-x)$ is measurable. This follows from the fact that, as a real function, $h$ is non-decreasing on the interval $(0,\frac{1}{2}]$ and non-increasing on $[\frac{1}{2},1)$. This, in turn, follows from the monotonicity of $\Delta_1$ under the preorder $\minor_{0,1}$ and the fact that
\begin{equation}
(1-x,x)\minor_{0,1} (1-y,y)
\end{equation}
whenever $0< x\leq y \leq \frac{1}{2}$, and the order reverses when $\frac{1}{2}\leq x\leq y<1$. Moreover, $h(\frac{1}{2})=1$ by the assumed normalization.
\item For every $p=(p_1,\ldots,p_n)\in\mc V_{<\infty}$ and $t\in[0,1]$, we have
\begin{equation}
\Delta_1\big(tp_1,(1-t)p_1,p_2,\ldots,p_n\big)=p_1 h(t)+\Delta(p).
\end{equation}
This follows from the Leibniz rule of $\Delta_1$ together with additivity.
\end{itemize}
According to \cite{Lee64}, this means that $-\Delta_1$ coincides on $\mc P_{<\infty}$ with the Shannon entropy $H_1$. This is all that we need for our applications. However, using the same technique as in the end of the proof of Proposition~\ref{prop:ModDeriv}(i), we may prove that $-\Delta_1$ coincides with $H_1$ on all $p \in\mc V_{\infty}$ with $\|p\|_1\in\Q$.
\end{remark}

\begin{remark}
	\label{mod_derivation_normalization}
	Theorem~\ref{thm:Fritz7.1} is formulated with the monotone derivations $\Delta$ normalized such that $\Delta(u_+)=\Delta(u_-)+1$, where $(u_-,u_+)$ is a chosen power universal pair. But of course this normalization is arbitrary, and in applying Theorem~\ref{thm:Fritz7.1} to probability vectors in $S_{0,1}$, we may as well normalize them such that the relevant monotone derivations are exactly $-H_1$ and $-H'_0$.
\end{remark}

\subsection{Monotone homomorphisms and derivations on the majorization semiring}\label{subsec:majorizationHom}

Moving on to the majorization semiring $S_1$, we will now derive a simple characterization of the monotone homomorphisms $\Phi:S_1\to\mb K$, where $\mb K\in\{\R_+,\R_+^{\rm op},\T\R_+,\T\R_+^{\rm op}\}$, in terms of the maps \eqref{eq:falpha} and \eqref{eq:finfty}.

\begin{proposition}\label{prop:monhomMaj}
The nondegenerate monotone homomorphisms $S_1 \to \mb K$ are exactly the following:
\begin{enumerate}[label=(\roman*)]
	\item For $\mb K = \R_+$, the maps $f_\alpha$ for $\alpha\in(1,\infty)$.
	\item For $\mb K = \R_+^{\rm op}$, the maps $f_\alpha$ for $\alpha \in [0,1)$.
	\item For $\mb K = \T\R_+$, the maps $f_\infty(\cdot)^\alpha$ for $\alpha \in (0,\infty)$.
	\item For $\mb K = \T\R_+^{\rm op}$, there is none.
\end{enumerate}
The only degenerate homomorphism $S_1 \to \R_+$ is $f_1=\|\cdot\|_1$.
\end{proposition}

\begin{proof}
	Since the majorization preorder $\rleq_1$ is an extension of the modified order $\rleq_{0,1}$ which we have already treated in Proposition~\ref{prop:monhomMod}, we already know that the monotone homomorphisms in each case are among those considered there. Also nondegeneracy is trivially preserved, and the degenerate homomorphism $f_0 = \|\cdot\|_0$ is now nondegenerate monotone with values in $\R_+$.

	So to prove the claim in case (ii), it is enough to prove that the $f_\alpha$ for $\alpha \in [0,1)$ are still monotone. But this follows by the same inequality~\eqref{falpha_mon} as before, which still applies since the assumption that all probabilities are nonzero was not used there.

	The same applies in case (i), where in addition we need to show that $f_\alpha$ with $\alpha \in (-\infty,0)$ is now no longer monotone.
	To this end, consider the following two-outcome vectors and their images under $f_{\alpha}$,
	\begin{equation}
		\label{eq:non_monotone}
		\begin{tikzcd}[column sep=tiny]
			[(\frac{1}{2},\frac{1}{2})] \ar[mapsto]{d}{f_{\alpha}} 	& \preceq_1	& {[(1-\eps,\eps)]}	\ar[mapsto]{d}{f_{\alpha}} 	& \preceq_1	& 1 \ar[mapsto]{d}{f_{\alpha}} \\
			2^{1+|\alpha|}				& <	& \left(1-\eps\right)^{-|\alpha|} + \eps^{-|\alpha|}		& >	& 1
		\end{tikzcd}
	\end{equation}
	where the inequalities in the second row hold for all sufficiently small $\eps>0$.
	Thus, the case (i) corresponds exactly to the functions $f_\alpha$ with $\alpha \in (1,\infty)$.

	The claim in case (iii) immediately follows from Proposition~\ref{prop:monhomMod}(iii).
	In case (iv), the map $f_\infty$ and its positive powers are still monotone by~\eqref{finfty_mon}, since there we did not use the assumption that all probabilities are nonzero.
	The map $f_{-\infty}$ and its positive powers are no longer monotone, as one can see by applying it to the same vectors as in~\eqref{eq:non_monotone}.
\end{proof}

Since the map $f_0=\|\cdot\|_0$ is no longer degenerate on $S_1$, we only need to consider monotone derivations at $\|\cdot\|_1$. Because the Shannon entropy $H_1$ is still monotone, Proposition \ref{prop:ModDeriv} immediately implies the following.

\begin{proposition}\label{prop:MajDeriv}
On vectors of rational $1$-norm, the monotone derivations $\Delta:S_1\to\R$ at $\|\cdot\|_1$ are exactly the negative multiples of $H_1$.
\end{proposition}

In the application of Theorem~\ref{thm:Fritz8.6} to normalized probability vectors in $S_1$, we can thus restrict to Shannon entropy $H_1$ itself as the only relevant derivation.

\subsection{Monotone homomorphisms on the submajorization semiring}

We will now derive a simple characterization of the monotone homomorphisms $\fii:S\to\mb K$ for $\mb K\in\{\R_+,\T\R_+\}$. We formulate the classification still in terms of the functions $f_\alpha$ from~\eqref{eq:falpha} and \eqref{eq:finfty}. It turns out that $\alpha\in[-\infty,1)$ is no longer relevant now.

\begin{proposition}\label{prop:monhomSub}
The monotone homomorphisms $S \to \mb K$ are exactly the following:
\begin{enumerate}[label=(\roman*)]
	\item For $\mb K = \R_+$, the maps $f_\alpha$ for $\alpha \in (1,\infty)$.
	\item For $\mb K = \T\R_+$, the maps $f_\infty(\cdot)^{\alpha}$ for $\alpha \in (0,\infty)$.
\end{enumerate}
\end{proposition}

\begin{proof}
	Since the submajorization preorder $\rleq$ is an extension of the majorization preorder $\rleq_1$, the monotone homomorphisms in each case are a subset of those of Proposition~\ref{prop:monhomMaj}, and it so happens that they are all still monotone. In case (i), the monotonicity proof for $\alpha \in [1,\infty)$ works essentially as before, with the minor difference that the first step consists of using $q \le Tp$ rather than substituting $q = Tp$. Similarly in case (ii).
\end{proof}

\section{Applications to majorization problems of probability vectors}\label{sec:classicalResults}

In this section, we apply the separation theorems of Theorems~\ref{thm:Fritz7.15},~\ref{thm:Fritz7.1} and~\ref{thm:Fritz8.6} to our majorization semirings in order to rederive results on large-sample and catalytic majorization. While these have been known before, our main original point here is that these results all follow from the same overarching framework. Before proceeding to these questions, let us make a useful observation.

\begin{remark}\label{rem:asymptocatalytic}
Suppose that $p$ and $q$ are finite probability vectors such that $p$ majorizes $q$ in large samples, i.e.,\ there is $n\in\N$ such that $p^{\otimes n} \minor_1 q^{\otimes n}$. Then $p$ also catalytically majorizes $q$, i.e.,\ there is a finite probability vector $r$ such that $p\otimes r \minor_1 q\otimes r$. To see this, we apply a well-known trick \cite{DuFeLiYi2005}:\footnote{See also Lemma 5.4 in~\cite{Fritz2015} for the general formulation of this statement.} Define
\begin{equation}\label{eq:FormOfCatalyst}
r:=\bigoplus_{j=1}^n \left( p^{\otimes(n-j)} \otimes q^{\otimes(j-1)} \right),\qquad s:=\bigoplus_{j=1}^{n-1} \left( p^{\otimes(n-j)} \otimes q^{\otimes j} \right).
\end{equation}
Then it is easy to see that $p\otimes r = p^{\otimes n} \oplus s$ and $r \otimes q = s \oplus q^{\otimes n}$, so that
\begin{equation}
p\otimes r=p^{\otimes n}\oplus s\minor_1 q^{\otimes n}\oplus s=q\otimes r.
\end{equation}
This proves the claims since one can normalize $r$ without affecting the validity of the inequality.
\end{remark}

\subsection{Asymptotic large-sample and catalytic majorization}\label{sec:AuNe}

For a given finite probability vector $p\in\mc P_{<\infty}$, define the sets of probability vectors majorized by $p$ catalytically and in large samples as follows:
\begin{align}
T_{<\infty}(p)&:=\{q\in\mc P_{<\infty}\,|\,q\otimes r\preceq_1 p\otimes r\ {\rm for\ some}\ r\in\mc P_{<\infty}\},\\
M_{<\infty}(p)&:=\{q\in\mc P_{<\infty}\,|\,q^{\otimes n}\preceq_1 p^{\otimes n}\ {\rm for\ some}\ n\in\N_{> 0} \}.
\end{align}
Subsequently, we will denote the closures of these sets in $\mc P_{<\infty}$ with respect to the total variation distance by $\overline{T_{<\infty}(p)}$ and $\overline{M_{<\infty}(p)}$.

We present a new proof for a result due to Aubrun and Nechita \cite{AubrunNechita2007} which characterizes the conditions for asymptotic catalytic or large-sample majorization in terms of the R\'{e}nyi entropies $H_\alpha$, $\alpha\geq1$, defined for $p=(p_1,\ldots,p_n)\in\mc P_{<\infty}$ through
\begin{align}\label{eq:Renyientr}
H_\alpha(p)&=\log{\|p\|_0}-D_{\alpha}(p\|u_{{\rm supp}\,p})\\
&=\left\{\begin{array}{@{}cl}
\frac{1}{1-\alpha}\log{\sum_{i=1}^n p_i^\alpha} &{\rm if}\ \alpha\in(0,1)\cup(1,\infty),\\[2mm]
\log{\|p\|_0} &{\rm if}\ \alpha=0 ,\\[2mm]
-\sum_{i \in {\rm supp}\,p} p_i \log{p_i} &{\rm if}\ \alpha=1,\\[2mm]
-\log{\max_{1\leq i\leq n}p_i}  &{\rm if}\ \alpha = \infty.
\end{array}\right.
\end{align}
Note, however, that we only need $H_\alpha$ with $\alpha\geq1$ below.

\begin{theorem}[Theorem 1 in \cite{AubrunNechita2007}]\label{thm:aubrunNechitaTheorem}
For every pair of finite probability vectors $p, q \in\mc P_{<\infty}$, the following are all equivalent:
\begin{enumerate}[label=(\roman*)]
	\item $q\in\overline{M_{<\infty}(p)}$.
	\item $q\in\overline{T_{<\infty}(p)}$.
	\item $H_\alpha(p) \leq H_\alpha(q) \; \text{for all} \,\ \alpha \geq 1$.
\end{enumerate}
\end{theorem}

Let us note that item (iii) is equivalent to $f_\alpha(p)\geq f_\alpha(q)$ for all $\alpha\geq1$, which in \cite{AubrunNechita2007} was formulated as $\|p\|_\alpha\geq\|q\|_\alpha$ with $\|\cdot\|_p$ being the usual $\ell^p$-norm for finite probability vectors. Our new proof for this theorem uses the same techniques as developed in \cite{AubrunNechita2007}, with the crucial difference being that the authors in \cite{AubrunNechita2007} use Cram\'{e}r's large deviation theorem for the key step in the proof of implication (iii)$\implies$(i), while for us this step directly follows from Theorem~\ref{thm:Fritz7.15} applied to $S$ together with our preceding identification of the monotone homomorphisms $S \to \R_+$ and $S \to \T\R_+$ from Proposition~\ref{prop:monhomSub}. Since our new techniques add nothing new to the proofs of the implications (i)$\implies$(ii), which essentially uses Remark~\ref{rem:asymptocatalytic}, and (ii)$\implies$(iii), we concentrate on (iii)$\implies$(i) only.

\begin{proof}
{\bf (iii)$\implies$(i):} Assuming item (iii), we know that $f_\alpha(p) \ge f_\alpha(q)$ for all $\alpha \in [1,\infty]$. Since Theorem~\ref{thm:Fritz7.15} requires strict inequalities in order to conclude ordering in large samples, the idea is to lower bound $q$ by some $q'_\eps$ within a given error range $\eps$ in order to create strict inequalities $f_\alpha(p) > f_\alpha(q'_\eps)$. After this, the lower bound $q'_\eps$ will be completed to a normalized vector while ensuring that it is still majorized by $p$ in large samples.

So let $\eps\in(0,q_{\rm min}\|q\|_0)$ be given, where $q_{\rm min}=f_{-\infty}(q)^{-1}$ is the smallest non-zero entry of $q$. We define $q'_\eps\in\mc V_{<\infty}$ by subtracting $\frac{\eps}{\|q\|_0}$ from all non-zero entries of $q$. This is a subnormalized vector with $\|q-q'_\eps\|_1 = \eps$. Since trivially $f_\alpha(q'_\eps) < f_\alpha(q)$ for all $\alpha \in [1,\infty]$, the assumption (iii) implies the strict inequalities $f_\alpha(q'_\eps)<f_\alpha(p)$ for all $\alpha\in[1,\infty]$. Theorem~\ref{thm:Fritz7.15} applied to $S$ together with Proposition~\ref{prop:monhomSub} now implies that there is $n_\eps\in\N$ such that $p^{\otimes n_\eps}$ majorizes $(q'_\eps)^{\otimes n_\eps}$.

We now complete $q'_\eps$ to the normalized vector
\begin{equation}
	q_\eps := q'_\eps \oplus \Big( \underbrace{\frac{\eps}{m}, \ldots, \frac{\eps}{m}}_{m \textrm{ copies}} \Big),
\end{equation}
i.e.,\ we append $m$ additional entries of value $\frac{\eps}{m}$ to $q'_\eps$, where $m$ is assumed large enough to ensure that $p^{\otimes n_\eps}$ still majorizes $q_\eps^{\otimes n_\eps}$. One way to do so is to require
\begin{equation}
	\frac{\eps}{m} \le \min \left\{ (q'_\eps)_{\min}, p_{\min}^{n_\eps} \right\}.
\end{equation}
Let us show that this does the trick. We claim that, whenever $k \geq 1$,
\begin{equation} \label{check_def_maj} \sum_{i=1}^k (q_\eps^{\otimes n_\eps})^{\downarrow}_i \leq \sum_{i=1}^k (p^{\otimes n_\eps})^{\downarrow}_i .
\end{equation}
Indeed, $m$ has been chosen so that the $\|q\|_0^{n_\eps}$ largest entries of $q_\eps^{\otimes n_\eps}$ are exactly the entries of $(q'_{\eps})^{\otimes n_\eps}$, so when $1 \leq k \leq \|q\|_0^{n_\eps}$, the inequality \eqref{check_def_maj} follows from the fact that $p^{\otimes n_\eps}$ majorizes $(q'_{\eps})^{\otimes n_\eps}$. If $\|q\|_0^{n_\eps}<k\leq \|p\|_0^{n_\eps}$, then the inequality follows inductively since the choice of $m$ guarantees that $(q_\eps^{\otimes n_\eps})^{\downarrow}_k \leq \frac{\eps}{m} \le (p^{\otimes n_\eps})^{\downarrow}_k$. Finally if $k > \|p\|_0^{n_\eps}$, then \eqref{check_def_maj} holds trivially since the right-hand side equals 1. In conclusion, $p^{\otimes n_\eps}$ majorizes $q_{\eps}^{\otimes n_\eps}$, and thus $q_\varepsilon \in M_{<\infty}(y)$. But $q_\varepsilon$ has been constructed such that $\|q-q_\varepsilon\|_1 \le \eps$. Since $\eps$ was arbitrary, we conclude $q \in \overline{M_{<\infty}(p)}$, which shows that property (i) holds.
\end{proof}

\subsection{Exact catalytic majorization}\label{subsec:Klimesh}

We now move on to the problem of characterizing exact catalytic majorization for finite probability distributions. Let us recall the R\'{e}nyi divergences $D_\alpha(\cdot\|\cdot)$, $\alpha\geq0$, of \eqref{eq:Renyirelentr}. Note that, whenever $0< \alpha < 1$ and ${\rm supp}\,p \cap {\rm supp}\,q \neq \emptyset$, we have
\begin{align}\label{eq:alphadual}
D_\alpha(p\|q)=\frac{\alpha}{1-\alpha}D_{1-\alpha}(q\|p).
\end{align}
Using these maps, we give a slight reformulation of the main result of \cite{klimesh2007inequalities}:

\begin{theorem}\label{thm:klimeshTheorem}
	Suppose that $p,q\in\mc P_{<\infty}$ and denote by $u:=u_{{\rm supp}\,p\,\cup\,{\rm supp}\,q}$ the uniform distribution on the union of the supports of $p$ and $q$. If, for all $\alpha \in [\frac{1}{2},\infty)$,
\begin{align}
D_\alpha(p\|u)&>D_\alpha(q\|u),\label{eq:directineq}\\
D_\alpha(u\|p)&>D_\alpha(u\|q),\label{eq:reverseineq}
\end{align}
then there is $r\in\mc P_{<\infty}$ such that $p\otimes r$ majorizes $q\otimes r$.
Conversely, if such $r$ exists, then the above inequalities hold at least non-strictly.
\end{theorem}

Note that the left-hand side of the inequalities in \eqref{eq:reverseineq} may attain the value $\infty$, in which case these conditions hold if the right-hand side is finite. Let us remark that we may order the vectors $p$ and $q$ in Theorem~\ref{thm:klimeshTheorem} so that one of their supports is contained in the other, e.g.,\ by choosing $p=p^\downarrow$ and $q=q^\downarrow$. In this way, $u$ is the uniform distribution on the larger one of ${\rm supp}\,p$ and ${\rm supp}\,q$. If the inequalities hold, then this is always ${\rm supp}\,q$, since $\|p\|_0>\|q\|_0$ would imply $D_\alpha(u\|p)<\infty=D_\alpha(u\|q)$, a violation of \eqref{eq:reverseineq}.
If one makes a less apt choice of ordering in which the supports are not contained, then this will make both sides of~\eqref{eq:reverseineq} infinite, and the theorem does not apply.

\begin{remark}
	The sufficient conditions given in Theorem~\ref{thm:klimeshTheorem} are actually also essentially necessary. Indeed for $p^\downarrow \neq q^\downarrow$, catalytic majorization already implies strict inequality in~\eqref{eq:directineq} and~\eqref{eq:reverseineq} thanks to the strict joint convexity of the R\'enyi divergences~\cite{klimesh2007inequalities}. However, in this work, we concentrate on the sufficiency proof and discuss the converse implication merely to indicate that the sufficient conditions are very close to necessary.
\end{remark}

Before proving Theorem~\ref{thm:klimeshTheorem}, we restate the claim in an equivalent way where the connection to the functions $f_\alpha$, $H_1$, and $H'_0$ is transparent:

\begin{theorem}
	\label{thm:klimeshTheorem2}
	Suppose that $p, q \in {\mc P}_{< \infty}$ are finite probability vectors. Suppose that the following hold:
	\begin{enumerate}[label=(\roman*)]
		\item $f_\alpha(p) < f_\alpha(q)$ for all $\alpha \in (0,1)$,
		\item $H_1(p) < H_1(q)$,
		\item $f_\alpha(p) > f_\alpha(q)$ for all $\alpha \in (1,\infty)$,
	\end{enumerate}
	and that either of the following holds:
	\begin{enumerate}[label=(\roman*),resume]
		\item $\|p\|_0<\|q\|_0$, or
		\item $\|p\|_0=\|q\|_0$ and $H'_0(p)<H'_0(q)$ and $f_\alpha(p)>f_\alpha(q)$ for all $\alpha\in(-\infty,0)$.
	\end{enumerate}
	Then there is a finite probability vector $r$ such that $p \otimes r$ majorizes $q \otimes r$.
	Conversely, if such $r$ exists, then the strict inequalities above hold non-strictly.
\end{theorem}

Let us clarify how the forward direction of this statement is equivalent to Theorem~\ref{thm:klimeshTheorem}. Let $p$, $q$, and $u$ be as in Theorem~\ref{thm:klimeshTheorem}, and let us assume without loss of generality that ${\rm supp}\,p \subseteq {\rm supp}\,q$. If $\|p\|_0=\|q\|_0$, then there are no infinities in the inequalities in \eqref{eq:directineq} and \eqref{eq:reverseineq}, and
\begin{align}
D_\alpha(p\|u)&=\frac{1}{\alpha-1}\log{f_\alpha(p)}+\log{\|p\|_0},\\
D_\alpha(u\|p)&=\frac{1}{\alpha-1}\log{f_{1-\alpha}(p)}+\frac{\alpha}{1-\alpha}\log{\|p\|_0}
\end{align}
for all $\alpha\neq1$, and
\begin{align}
D_1(p\|u)&=\log{\|p\|_0}-H_1(p),\\
D_1(u\|p)&=-\log{\|p\|_0}-\frac{1}{\|p\|_0}H'_0(p),
\end{align}
and likewise for $q$. Taking into account the sign of $\alpha-1$, we obtain the equivalence of Theorem~\ref{thm:klimeshTheorem2} with Theorem~\ref{thm:klimeshTheorem} in the case $\|p\|_0 = \|q\|_0$, since for example the inequality $D_1(u\|p)>D_1(u\|q)$ turns into $H'_0(p)<H'_0(q)$. If $\|p\|_0<\|q\|_0$, then the inequalities \eqref{eq:reverseineq} are automatically satisfied as the left-hand side is infinite whereas the right-hand side is finite. Thus only the inequalities \eqref{eq:directineq} are relevant, and we arrive at items (i)--(iii) of Theorem~\ref{thm:klimeshTheorem2}. Finally, as alluded to above, $\|p\|_0>\|q\|_0$ is excluded by the assumptions of both theorems.

We now prove Theorem~\ref{thm:klimeshTheorem2}.

\begin{proof}
	Let us consider the case $\|p\|_0 < \|q\|_0$ first. Since we have a strict inequality for $\|\cdot\|_0$, the relevant semiring is $S_1$. Recall from Proposition \ref{prop:monhomMaj} that the set of nondegenerate monotone homomorphisms is now constituted by $f_\alpha:S_1\to\R_+$ for $\alpha\in(1,\infty)$ and $f_\alpha:S_1\to\R_+^{\rm op}$ for $\alpha\in[0,1)$ as well as $f_\infty : S_1 \to \T\R_+$, and the only relevant derivation at the degenerate homomorphism $\|\cdot\|_1$ is $-H_1$ by Proposition~\ref{prop:MajDeriv}. Recalling that
\begin{equation}\label{eq:alphalimit1}
	f_\infty(r) = \|r\|_\infty = \lim_{\alpha \to \infty} \|r\|_\alpha = \lim_{\alpha \to \infty} f_\alpha(r)^{1/\alpha}
\end{equation}
for all $r\in\mc V_{<\infty}$, where $\|\cdot\|_\alpha$ for $\alpha\in[1,\infty]$ are the usual $\alpha$-norms, it follows that $f_\infty(p)\geq f_\infty(q)$. So if $f_\infty(p)>f_\infty(q)$, then the claim immediately follows by applying Theorem~\ref{thm:Fritz7.1} to $S_1$ and using Propositions~\ref{prop:monhomMaj} and~\ref{prop:MajDeriv}. Thus consider the subcase $f_\infty(p)=f_\infty(q)$. We may freely also assume that $p=p^\downarrow = (p_1,\ldots,p_n)$ and $q=q^\downarrow = (q_1,\ldots,q_n)$ as well as $p \neq q$. We also write $I := \{1,\ldots,n\}$ as the disjoint union $I=I_{\rm max}\cup I_0$ where $I_{\rm max}=\{1,\ldots,m\}$ is the longest initial sequence with $p_i=q_i$ for $i\in I_{\rm max}$, so that $p_{m+1}\neq q_{m+1}$. Let us use the notations
\begin{align}
p'&:=(p_i)_{i\in I_0},\\
q'&:=(q_i)_{i\in I_0},\\
p''=q''&:=(p_i)_{i\in I_{\rm max}} = (q_i)_{i \in I_{\rm max}}.
\end{align}
We now have, for all $\alpha\in[0,1)$,
\begin{equation}
f_\alpha(p')+f_\alpha(p'')=f_\alpha(p)<f_\alpha(q)=f_\alpha(q')+f_\alpha(q'')=f_\alpha(q')+f_\alpha(p''),
\end{equation}
where we have used the additivity of $f_\alpha$. Therefore $f_\alpha(p')<f_\alpha(q')$ for all $\alpha\in[0,1)$. Similarly, we obtain $H_1(p')<H_1(q')$ and $f_\alpha(p')>f_\alpha(q')$ for all $\alpha>1$. Using the limit in \eqref{eq:alphalimit1}, we conclude $f_\infty(p')\geq f_\infty(q')$, i.e.,\ $p_{m+1}\geq q_{m+1}$. As these entries do not coincide, we must have $p_{m+1}>q_{m+1}$, or equivalently $f_\infty(p')>f_\infty(q')$. Thus, Theorem~\ref{thm:Fritz7.1} can be applied to $p',q' \in S_1$, and therefore there is $r\in\mc V_{<\infty}\setminus\{0\}$ such that $p'\otimes r\minor_1 q'\otimes r$. Clearly, by renormalizing, we may assume that $\|r\|_1=1$. Putting things together, we have
\begin{align}
p\otimes r&=(p'\otimes r)\oplus(p''\otimes r)=(p'\otimes r)\oplus(q''\otimes r)\\
&\minor_1(q'\otimes r)\oplus(q''\otimes r)=q\otimes r,
\end{align}
as was to be shown.

It remains to consider the case $\|p\|_0 = \|q\|_0$, which works similarly. Now the relevant semiring is $S_{0,1}$, since $\|\cdot\|_0$ is degenerate. In addition to the monotone homomorphisms on $S_1$, by Proposition~\ref{prop:monhomMod} we now also need the $f_\alpha:S_{0,1}\to\R_+$ for $\alpha\in(-\infty,0)$ as well as $f_{-\infty}$, and the only relevant monotone derivation at the now degenerate $\|\cdot\|_0$ is $-H'_0$ by Proposition~\ref{prop:ModDeriv}. Using the fact that for all $r=(r_1,\ldots,r_n)\in\mc V_{<\infty}$ with $\tilde{r}:=(1/r_i)_{i\in{\rm supp}\,r}$ we have
\begin{equation}\label{eq:alphalimit2}
	f_{-\infty}(r) = \|\tilde{r}\|_\infty = \lim_{\alpha \to -\infty} \|\tilde{r}\|_{-\alpha} = \lim_{\alpha \to -\infty} f_\alpha(\tilde{r})^{-1/\alpha},
\end{equation}
it follows that $f_{-\infty}(p)\geq f_{-\infty}(q)$ in addition to $f_\infty(p) \ge f_\infty(q)$ from \eqref{eq:alphalimit1}. If $f_{\pm\infty}(p) > f_{\pm\infty}(q)$, the claim follows by another application of Theorem~\ref{thm:Fritz7.1}. In general, we may again assume that $p=p^\downarrow$ and $q=q^\downarrow$ and write the common support $I=\{1,\ldots,n\}$ of these vectors as a disjoint union $I=I_{\rm max}\cup I_0\cup I_{\rm min}$, where $I_{\rm max}=\{1,\ldots,\ell\}$ is the longest initial part of $I$ and $I_{\rm min}=\{\ell+m+1,\ldots,n\}$ is the longest final part of $I$ on which $p_i=q_i$ for all $i\in I_{\rm max}\cup I_{\rm min}$. In particular, we have $p_{\ell+1}\neq q_{\ell+1}$ and $p_{\ell+m}\neq q_{\ell+m}$. Let us now denote
\begin{align}
p'&:=(p_i)_{i\in I_0},\\
q'&:=(q_i)_{i\in I_0},\\
p''=q''&:= (p_i)_{i\in I_{\rm max}\cup I_{\rm min}}  = (q_i)_{i\in I_{\rm max}\cup I_{\rm min}}.
\end{align}
Similarly as in the first case, $p'$ and $q'$ satisfy the same inequalities (i)--(iii) and (iv). Moreover, using again the limits in \eqref{eq:alphalimit1} and \eqref{eq:alphalimit2}, we have $f_{\pm\infty}(p')\geq f_{\pm\infty}(q')$, or equivalently $p_{\ell+1}\geq q_{\ell+1}$ and $p_{\ell+m}\leq q_{\ell+m}$. Since we do not have equalities in these inequalities, we conclude that they hold strictly, or equivalently $f_{\pm\infty}(p') > f_{\pm\infty}(q')$. Theorem~\ref{thm:Fritz7.1} applied to $S_{0,1}$ now gives us $r\in\mc V_{<\infty}$ such that $p'\otimes r\minor_{0,1} q'\otimes r$. Normalizing $r$ and proceeding as in the first case proves the claim.
\end{proof}

\subsection{Exact large-sample majorization}

We now turn our attention to exact majorization in large samples. The following result first appeared in \cite{Jensen2019}.

\begin{theorem}[Proposition 3.7 in \cite{Jensen2019}] \label{thm:jensenTheorem}
Let $p,q\in\mc P_{<\infty}$. If, for all $\alpha\in[0,\infty]$,
\begin{align}
H_\alpha(p)&<H_\alpha(q),\label{eq:Jensenineq1}
\end{align}
then $p^{\otimes n}$ majorizes $q^{\otimes n}$ for all sufficiently large $n\in\N$.
Conversely, if this is the case for all $n \in \N$, then the above inequalities hold non-strictly.
\end{theorem}

In \cite{Jensen2019}, the above condition for large-sample transformation was stated in the form
\begin{equation}
\min_{\alpha\in[0,\infty]}\frac{H_\alpha(q)}{H_\alpha(p)}>1.
\end{equation}
But since the family of maps $\{H_\alpha\}_{\alpha\in(0,\infty]}$ forms a compact topological space in the weakest topology which makes the evaluation maps continuous\footnote{This can be seen as an instance of another general fact about preordered semirings~\cite[Proposition 8.5]{Fritz2021b}.}, the existence of the above minimum is already implied by the strict inequalities in \eqref{eq:Jensenineq1}. Using the fact that
\begin{equation}
H_\alpha(r)=\frac{1}{1-\alpha}\log{f_\alpha(r)}
\end{equation}
for all $r\in\mc P_{<\infty}$ and $\alpha\in(0,1)\cup(1,\infty)$, we may also formulate Theorem~\ref{thm:jensenTheorem} as follows:

\begin{theorem}
	Let $p,q \in\mc P_{<\infty}$ be finite probability vectors. If
	\begin{align}
		f_\alpha(p)&<f_\alpha(q) \qquad \forall \alpha\in[0,1),\label{eq:Jensenineq2}\\
		f_\alpha(p)&>f_\alpha(q) \qquad \forall \alpha\in(1,\infty],\label{eq:Jensenineq3}\\
		H_1(p)&<H_1(q),\label{eq:Jensenineq4}
	\end{align}
	then $p^{\otimes n}$ majorizes $q^{\otimes n}$ for all sufficiently large $n\in\N$. Conversely, if this is the case for all $n \in \N$, then the above inequalities hold non-strictly.
\end{theorem}

This is the version of the theorem we will now prove. Note that the derivation $H'_0$ is not directly involved in these conditions; instead, the relevant quantity at $\alpha=0$ is $H_0$, i.e.,\ essentially $\|\cdot\|_0$. 

\begin{proof}
	Let us assume the inequalities in \eqref{eq:Jensenineq2}, \eqref{eq:Jensenineq3}, and \eqref{eq:Jensenineq4}; in particular $\|p\|_0<\|q\|_0$. Thus the relevant preordered semiring is $S_1$ where by Proposition~\ref{prop:monhomMaj} the set of nondegenerate monotone homomorphisms $\fii:S_1\to\mb K\in\{\R_+,\R_+^{\rm op},\T\R_+,\T\R_+^{\rm op}\}$ consists of $f_\alpha$ for $\alpha\in[0,1)\cup (1,\infty]$ and the only relevant monotone derivation at the degenerate $\|\cdot\|_1$ is $-H_1$. Since $\|q\|_0 > \|p\|_0$, we know that $[q] \neq 1$, and therefore $q$ is power universal in $S_1^{\rm op}$ by Remark~\ref{rem:polygrowthmaj}. This is why we have to consider $S_1^{\rm op}$ so that we can apply Theorem \ref{thm:Fritz8.6}.
 
    Moving to $S_1^{\rm op}$ from $S_1$, we may still characterize the monotone homomorphisms of $S_1^{\rm op}$ into $\mb K\in\{\R_+,\R_+^{\rm op},\T\R_+,\T\R_+^{\rm op}\}$ and relevant derivations similarly as in Propositions \ref{prop:monhomMaj} and \ref{prop:MajDeriv} with the exception that $\R_+$ and $\T\R_+$ switch places respectively with $\R_+^{\rm op}$ and $\T\R_+^{\rm op}$ and the sign of the derivation changes. This means that the nondegenerate monotone homomorphisms $\Phi:S_1^{\rm op}\to\mb K$ is given by $\Phi=f_\alpha$ with $\alpha\in[0,1)$ when $\mb K=\R_+$, by $\Phi=f_\alpha$ with $\alpha\in(1,\infty)$ when $\mb K=\R_+^{\rm op}$, and by $\Phi=f_\infty$ when $\mb K=\T\R_+^{\rm op}$ and there are none when $\mb K=\T\R_+^{\rm op}$. The relevant monotone derivation at the only degenerate homomorphism $f_1:S_1^{\rm op}\to\R_+$ is $H_1$ (with positive sign now). Applying item (a) of Theorem~\ref{thm:Fritz8.6} to $S_1^{\rm op}$, we obtain the claim.

	The converse part is straightforward based on the fact that the $H_\alpha$ are additive under $\otimes$ and monotone under majorization.
\end{proof}

Let us note that, when $p,q\in\mc P_{<\infty}$ satisfy $H_\alpha(p)<H_\alpha(q)$, according to Remark~\ref{rem:asymptocatalytic} and part (b) of Theorem \ref{thm:Fritz8.6}, $p\otimes r$ majorizes $q\otimes r$ where $r$ has the form \eqref{eq:FormOfCatalyst} with some sufficiently large $n\in\N$.

However, using our results on matrix majorization, namely Theorem~\ref{theor:MatrMajSuff}, together with Theorem~\ref{thm:jensenTheorem} we obtain an extension to Theorem~\ref{thm:jensenTheorem} characterizing large-sample simple majorization. As in Theorem~\ref{thm:klimeshTheorem}, we assume that, if the left-hand side of the inequalities in \eqref{eq:LSdirect} or \eqref{eq:LSinverse} is infinite and the right-hand side is finite, the inequality holds.

\begin{corollary}\label{cor:JensenExtension}
Consider probability vectors $p,q\in\mc P_{<\infty}$. Denote by $u$ the uniform distribution on ${\rm supp}\,p\cup {\rm supp}\,q$. Whenever
\begin{align}
D_\alpha(p\|u)&>D_\alpha(q\|u),\label{eq:LSdirect}\\
D_\alpha(u\|p)&>D_\alpha(u\|q)\label{eq:LSinverse}
\end{align}
for all $\alpha\in[\frac{1}{2},\infty]$, then, for sufficiently large $n\in\N$, $p^{\otimes n}$ majorizes $q^{\otimes n}$.
\end{corollary}

\begin{proof}
We may again assume that $p=p^\downarrow$ and $q=q^\downarrow$. As we observed in the discussion just after Theorem~\ref{thm:klimeshTheorem}, we now see that, if conditions \eqref{eq:LSdirect} and \eqref{eq:LSinverse} hold, then ${\rm supp}\,p\subseteq{\rm supp}\,q$, so we may assume that $u$ is the uniform distribution on ${\rm supp}\,q$. Similarly as in our discussion just after Theorem~\ref{thm:klimeshTheorem2}, we now see that the conditions in the claim split into two depending on the sizes of the supports of $p$ and $q$:
\begin{align}
f_\alpha(p)&<f_\alpha(q)\ {\rm for\ all}\ \alpha\in(0,1),\\
H_1(p)&<H_1(q),\ {\rm and}\\
f_\alpha(p)&>f_\alpha(q)\ {\rm for\ all}\ \alpha\in(1,\infty]
\end{align}
and, additionally, either $\|p\|_0<\|q\|_0$ or $\|p\|_0=\|q\|_0$ and $H'_0(p)<H'_0(q)$ and $f_\alpha(p)>f_\alpha(q)$ for all $\alpha\in[-\infty,0)$.

If $\|p\|_0<\|q\|_0$, the claim follows now directly from Theorem~\ref{thm:jensenTheorem}. Let us concentrate, hence, on the case $\|p\|_0=\|q\|_0$. Now the uniform distribution $u$ shares the same support with $p$ and $q$. Recalling the definition of the matrix majorization semiring $S^2$, we now see that $[(p,u)],[(q,u)]\in S^2$. According to Corollary~\ref{cor:RelativeMaj}, the conditions in \eqref{eq:LSdirect} and \eqref{eq:LSinverse} imply that, for $n\in\N$ sufficiently large, $(p^{\otimes n},u^{\otimes n})$ majorizes $(q^{\otimes n},u^{\otimes n})$ and, because $u^{\otimes n}$ is the uniform distribution sharing the same support with $p^{\otimes n}$ and $q^{\otimes n}$, this means that $p^{\otimes n}$ majorizes $q^{\otimes n}$.
\end{proof}

Note that the difference between the conditions appearing in Theorem~\ref{thm:klimeshTheorem} and Corollary~\ref{cor:JensenExtension} is that the conditions at the endpoint $\alpha=\infty$ are included in Corollary~\ref{cor:JensenExtension} whereas in Theorem~\ref{thm:klimeshTheorem} they are not included. This is in line with the fact that large-sample majorization implies catalytic majorization.

\section*{Acknowledgment}

T.\ F.\ would like to thank Omer Tamuz and Luciano Pomatto for insightful discussions. E.\ H.\ and M.\ T.\ are supported by the National Research Foundation, Singapore and A*STAR under its CQT Bridging Grant. M.\ T.\ is also supported by NUS startup grants (R-263-000-E32-133 and R-263-000-E32-731).

\bibliographystyle{ultimate}
\bibliography{library}

\end{document}